\documentclass[11pt]{amsart}
\usepackage{color,amssymb}
\usepackage{bbm}
\usepackage[normalem]{ulem}

\usepackage[all]{xy}

\newtheorem{theorem}{Theorem}[section]
\newtheorem{claim}[theorem]{Claim}

\newtheorem{proposition}[theorem]{Proposition}
\newtheorem{corollary}[theorem]{Corollary}
\newtheorem{lemma}[theorem]{Lemma}
\newtheorem{problem}[theorem]{Problem}

\theoremstyle{definition}
\newtheorem{remark}[theorem]{Remark}
\newtheorem{definition}[theorem]{Definition}
\newtheorem{example}[theorem]{Example}

\newtheorem*{definition*}{Definition}

\newcommand{\IN}{\mathbb N}
\newcommand{\IR}{\mathbb R}

\newcommand{\F}{\mathcal F}
\newcommand{\U}{\mathcal U}
\newcommand{\C}{\mathcal C}
\newcommand{\V}{\mathcal V}
\newcommand{\E}{\mathcal E}

\newcommand{\e}{\varepsilon}
\newcommand{\w}{\omega}

\newcommand{\Fil}{\mathrm{\varphi}}
\newcommand{\Ra}{\Rightarrow}
\newcommand{\Haus}{\mathsf{T_{\!2}S}}

\newcommand{\Zero}{\mathsf{T_{\!z}S}}

\newcommand{\Tau}{\mathcal T}

\newcommand{\korin}[2]{\!\sqrt[#1]{\!#2}}
\newcommand{\TS}{\mathcal T_{\mathcal D}}

\usepackage{relsize}
\newcommand*{\defeq}{\stackrel{\mathsmaller{\mathsf{def}}}{=}}

\newcommand{\two}{\mathbbm 2}
\newcommand{\Lambdae}{\Lambda^{\!e}}
\newcommand{\lesseq}{\lesssim}
\newcommand{\Zeta}{\mathcal B}

\title{Injectively closed commutative semigroups}
\author{Taras Banakh}
\address{T.Banakh: Ivan Franko National University of Lviv (Ukraine) and Jan Kochanowski University in Kielce (Poland)}
\email{t.o.banakh@gmail.com}
\makeatletter
\@namedef{subjclassname@2020}{%
  \textup{2020} Mathematics Subject Classification}
\makeatother


\begin{document}
\begin{abstract} Let $\C$ be a class of topological semigroups. A semigroup $X$ is  {\em injectively $\C$-closed} if $X$ is closed in each topological semigroup $Y\in\C$ containing $X$ as a subsemigroup. 
Let $\mathsf{T_{\!2}S}$ (resp. $\mathsf{T_{\!z}S}$) be the class of Hausdorff (and  zero-dimensional) topological semigroups. We prove that a commutative semigroup $X$ is injectively $\mathsf{T_{\!2}S}$-closed if and only if $X$ is injectively $\mathsf{T_{\!z}S}$-closed if and only if $X$ is bounded, chain-finite, group-finite, nonsingular and not Clifford-singular.
\end{abstract}
\maketitle

\section{Introduction and Main Results}

In many cases,  completeness properties of various objects of General Topology or  Topological Algebra can be characterized externally as closedness in ambient objects. For example, a metric space $X$ is complete if and only if $X$ is closed in any metric space containing $X$ as a subspace. A uniform space $X$ is complete if and only if $X$ is closed in any uniform space containing $X$ as a uniform subspace. A topological group $G$ is Ra\u\i kov complete  if and only if it is closed in any topological group containing $G$ as a subgroup.

On the other hand, for topological semigroups there are no reasonable notions of (inner) completeness. Nonetheless we can define many completeness properties of semigroups via their closedness in ambient topological semigroups.

A {\em topological semigroup} is a topological space $X$ endowed
with a continuous associative binary operation $X\times X\to
X$, $(x,y)\mapsto xy$.

\begin{definition*} Let $\C$ be a class of topological semigroups.
A topological
semigroup $X$ is called
\begin{itemize}
\item {\em $\C$-closed} if for any isomorphic topological
embedding $h:X\to Y$ to a topological semigroup $Y\in\C$
the image $h[X]$ is closed in $Y$;
\item {\em injectively $\C$-closed} if for any injective continuous homomorphism $h:X\to Y$ to a topological semigroup $Y\in\C$ the image $h[X]$ is closed in $Y$;
\item {\em absolutely $\C$-closed} if for any continuous homomorphism $h:X\to Y$ to a topological semigroup $Y\in\C$ the image $h[X]$ is closed in $Y$.
\end{itemize}
\end{definition*}

For any topological semigroup we have the implications:
$$\mbox{absolutely $\C$-closed $\Ra$ injectively $\C$-closed $\Ra$ $\C$-closed}.$$

\begin{definition*} A semigroup $X$ is defined to be ({\em injectively, absolutely}) {\em $\C$-closed\/} if so is $X$ endowed with the discrete topology.
\end{definition*}

We will be interested in the (absolute, injective) $\C$-closedness for the classes:
\begin{itemize}
\item $\mathsf{T_{\!1}S}$ of topological semigroups satisfying the separation axiom $T_1$;
\item $\Haus$ of Hausdorff topological semigroups;
\item $\Zero$ of Hausdorff zero-dimensional topological
semigroups.
\end{itemize}
A topological space satisfies the separation axiom $T_1$ if all its finite subsets are closed.
A topological space is {\em zero-dimensional} if it has a base of
the topology consisting of {\em clopen} (=~closed-and-open) sets.

Since $\Zero\subseteq\Haus\subseteq\mathsf{T_{\!1}S}$, for every semigroup we have the implications:
$$
\xymatrix{
\mbox{absolutely $\mathsf{T_{\!1}S}$-closed}\ar@{=>}[r]\ar@{=>}[d]&\mbox{absolutely $\mathsf{T_{\!2}S}$-closed}\ar@{=>}[r]\ar@{=>}[d]&\mbox{absolutely $\mathsf{T_{\!z}S}$-closed}\ar@{=>}[d]\\
\mbox{injectively $\mathsf{T_{\!1}S}$-closed}\ar@{=>}[r]\ar@{=>}[d]&\mbox{injectively $\mathsf{T_{\!2}S}$-closed}\ar@{=>}[r]\ar@{=>}[d]&\mbox{injectively $\mathsf{T_{\!z}S}$-closed}\ar@{=>}[d]\\
\mbox{$\mathsf{T_{\!1}S}$-closed}\ar@{=>}[r]&\mbox{$\mathsf{T_{\!2}S}$-closed}\ar@{=>}[r]&\mbox{$\mathsf{T_{\!z}S}$-closed.}
}
$$

$\C$-Closed topological groups for various classes $\C$ were investigated by many authors
~\cite{AC,AC1,Ban,DU,G,L}. In particular, the closedness of commutative topological groups in the class of Hausdorff topological semigroups was investigated
in~\cite{Z1,Z2}; $\mathcal{C}$-closed topological semilattices were investigated
in~\cite{BBm, BBh, BBR, GutikPagonRepovs2010, GutikRepovs2008, Stepp69, Stepp75}.
For more information about complete topological semilattices and pospaces we refer to the survey~\cite{BBc}. This paper is a continuation of the papers \cite{CCUS}, \cite{ICT1S}, \cite{BB}, \cite{BBm}, \cite{ACVS}, devoted to finding inner characterization of various closedness properties of discrete semigroups. In order to formulate such inner characterizations, let us recall some properties of semigroups.


A semigroup $X$ is called
\begin{itemize}
\item {\em unipotent} if $X$ has a unique idempotent;
\item {\em chain-finite} if any infinite set $I\subseteq X$ contains two elements  $x,y\in I$ such that $xy\notin\{x,y\}$;
\item {\em singular} if there exists an infinite set $A\subseteq X$ such that $AA$ is a singleton;
\item {\em periodic} if for every $x\in X$ there exists $n\in\IN$ such that $x^n$ is an idempotent;
\item {\em bounded} if there exists $n\in\IN$ such that for every $x\in X$ the $n$-th power $x^n$ is an idempotent;
\item {\em group-finite} if every subgroup of $X$ is finite;
\item {\em group-bounded} if every subgroup of $X$ is bounded.
\end{itemize}
We recall that an element $x$ of a semigroup is an {\em idempotent} if $xx=x$.

The following theorem (proved in \cite{BB}) characterizes $\C$-closed commutative semigroups.

\begin{theorem}[Banakh--Bardyla]\label{t:C} Let $\C$ be a class of topological semigroups such that $\mathsf{T_{\!z}S}\subseteq\C\subseteq \mathsf{T_{\!1}S}$. A commutative semigroup $X$ is $\C$-closed if and only if $X$ is chain-finite, nonsingular,  periodic, and group-bounded.
\end{theorem}

For unipotent commutative semigroups, this characterization was simplified in \cite{CCUS} as follows.

\begin{theorem}[Banakh--Vovk]\label{t:CUS} Let $\C$ be a class of topological semigroups such that $\mathsf{T_{\!z}S}\subseteq\C\subseteq \mathsf{T_{\!1}S}$. A unipotent commutative semigroup $X$ is $\C$-closed if and only if $X$ is nonsingular and bounded.
\end{theorem}

\begin{definition*}Let $\C$ be a class of topological semigroups. A semigroup $X$ is defined to be {\em $\C$-discrete} (or else {\em $\C$-nontopologizable}) if for any injective homomorphism $h:X\to Y$ to a topological semigroup $Y\in\C$ the image $h[X]$ is a discrete subspace of $Y$.
\end{definition*}

The study of topologizable and nontopologizable semigroups is a classical topic in Topological Algebra that traces its history back to Markov's problem \cite{Markov} of topologizabilioty of infinite groups, which was resolved by \cite{Shelah}, \cite{Hesse} and \cite{Ol}. For some other results on topologizability of semigroups, see \cite{BGP,BM,BPS,DS2,vD,DI,GLS,Gutik,KOO,Kotov,Taimanov}.

The following characterization of injectively $\mathsf{T_{\!1}S}$-closed semigroups was proved in \cite[3.2]{CCCS}.

\begin{theorem}[Banakh--Bardyla] A semigroup $X$ is inejctively $\mathsf{T_{\!1}S}$-closed if and only if it is $\mathsf{T_{\!1}S}$-closed and $\mathsf{T_{\!1}S}$-discrete.
\end{theorem}

Theorem~\ref{t:C} implies that each subsemigroup of a $\C$-closed commutative semigroup is $\C$-closed. On the other hand, quotient semigroups of $\C$-closed commutative semigroups are not necessarily $\C$-closed, see Example 1.8 in \cite{BB}. This motivated the authors of \cite{BB} to introduce the notions of ideally and projectively $\C$-closed semigroups. 

Let us recall that a {\em congruence} on a semigroup $X$ is an equivalence relation $\approx$ on $X$ such that for any elements $x\approx y$ of $X$ and any $a\in X$ we have $ax\approx ay$ and $xa\approx ya$. For any congruence $\approx$ on a semigroup $X$, the quotient set $X/_\approx$ has a unique semigroup structure such that the quotient map $X\to X/_\approx$ is a semigroup homomorphism. The semigroup $X/_\approx$ is called the {\em quotient semigroup} of $X$ by the congruence $\approx$~.
A subset $I$ of a semigroup $X$ is called an {\em ideal} in $X$ if $IX\cup XI\subseteq  I$. Every ideal $I\subseteq X$ determines the congruence $(I\times I)\cup \{(x,y)\in X\times X:x=y\}$ on $X\times X$. The quotient semigroup of $X$ by this congruence is denoted by $X/I$ and called the {\em quotient semigroup} of $X$ by the ideal $I$. If $I=\emptyset$, then the quotient semigroup $X/\emptyset$ can be identified with the semigroup $X$.

\begin{definition*} A semigroup $X$ is called
\begin{itemize}
\item {\em projectively $\C$-closed} if for any congruence $\approx$ on $X$ the quotient semigroup $X/_{\approx}$ is $\C$-closed;
\item {\em ideally $\C$-closed} if for any ideal $I\subseteq X$ the quotient semigroup $X/I$ is $\C$-closed.
\end{itemize}
\end{definition*}

It is easy to see that for every semigroup the following implications hold:
$$\mbox{absolutely $\C$-closed $\Ra$ projectively $\C$-closed $\Ra$ ideally $\C$-closed $\Ra$ $\C$-closed.}$$
Observe that a semigroup $X$ is absolutely $\C$-closed if and only if for any congruence $\approx$ on $X$ the semigroup $X/_\approx$ is injectively $\C$-closed.

Projectively and ideally $\C$-closed commutative semigroups were characterized in \cite{BB}. To formulate those characterizations, we need to introduce the next portion of definitions.

For a semigroup $X$, let $E(X)\defeq\{x\in X:xx=x\}$ be the set of idempotents of $X$. For an idempotent $e$ of a semigroup $X$, let $H_e$ be the maximal subgroup of $X$ that contains $e$. The union $H(X)=\bigcup_{e\in E(X)}H_e$ of all subgroups of $X$ is called the {\em Clifford part} of $S$.

A semigroup $X$ is called
\begin{itemize}
\item {\em Clifford}  if $X=H(X)$;
\item {\em Clifford+finite} if $X\setminus H(X)$ is finite;
\item {\em Clifford-finite} if $H(X)$ is finite;
\item {\em Clifford-singular} if there exists an infinite set $A\subseteq X\setminus H(X)$ such that $AA\subseteq H(X)$.
\end{itemize}

Ideally and projectively $\C$-closed commutative semigroups were characterized in  \cite{BB} as follows.

\begin{theorem}[Banakh--Bardyla]\label{t:mainP} Let $\C$ be a class of topological semigroups such that $\mathsf{T_{\!z}S}\subseteq\C\subseteq \mathsf{T_{\!1}S}$. For a commutative semigroup $X$ the following conditions are equivalent:
\begin{enumerate}
\item $X$ is projectively $\C$-closed;
\item $X$ is ideally $\C$-closed;
\item the semigroup $X$ is chain-finite, group-bounded and Clifford+finite.
\end{enumerate}
\end{theorem}



Absolutely $\C$-closed commutative semigroups were characterized in \cite{ACVS}. The characterizations essentially depend on the class $\C$, where we distinguish two cases: $\C=\mathsf{T_{\!1}S}$ and $\mathsf{T_{\!z}S}\subseteq\C\subseteq\mathsf{T_{\!2}S}$.

\begin{theorem}[Banakh-Bardyla] A commutative semigroup $X$ is absolutely $\mathsf{T_{\!1}S}$-closed if and only if $X$ is finite.
\end{theorem}

\begin{theorem}[Banakh--Bardyla]\label{t:aC} Let $\C$ be a class of topological semigroups such that $\mathsf{T_{\!z}S}\subseteq\C\subseteq\mathsf{T_{\!2}S}$. A commutative semigroup $X$ is absolutely $\C$-closed if and only if $X$ is chain-finite, bounded, group-finite and Clifford+finite.
\end{theorem}

This paper completes a series of papers  \cite{ICT1S}, \cite{CCUS} devoted to characterizations of injectively $\C$-closed commutative semigroups. The paper \cite{CCUS} treats unipotent semigroups and contains the following characterization:

\begin{theorem}[Banakh--Vovk]\label{t:iCUS} Let $\C$ be a class of topological semigroups such that $\mathsf{T_{\!z}S}\subseteq\C\subseteq\mathsf{T_{\!1}S}$. A commutative unipotent semigroup $X$ is injectively $\C$-closed if and only if $X$ is bounded, nonsingular and group-finite.
\end{theorem}

The paper \cite{ICT1S} contains the following characterization of injectively $\mathsf{T_{\!1}S}$-closed commutative semigroups.

\begin{theorem}[Banakh]\label{t:iT1} For a commutative semigroup $X$ the following conditions are equivalent:
\begin{enumerate}
\item $X$ is injectively $\mathsf{T_{\!1}S}$-closed;
\item $X$ is $\mathsf{T_{\!1}S}$-closed and $\mathsf{T_{\!1}S}$-discrete;
\item $X$ is $\mathsf{T_{\!z}S}$-closed and $\mathsf{T_{\!z}S}$-discrete;
\item $X$ is bounded, nonsingular and Clifford-finite.
\end{enumerate}
\end{theorem}

In this paper we characterize injectively $\C$-closed commutative semigroups for classes $\C$ with $\mathsf{T_{\!z}S}\subseteq\C\subseteq\mathsf{T_{\!2}S}$, and prove the following main result.

\begin{theorem}\label{t:main-iC} Let $\C$ be a class of topological semigroups such that $\mathsf{T_{\!z}S}\subseteq\C\subseteq\mathsf{T_{\!2}S}$. A commutative semigroup $X$ is injectively $\C$-closed if and only if $X$ is bounded, chain-finite, group-finite, nonsingular and not Clifford-singular.
\end{theorem}

Theorems~\ref{t:iT1} and \ref{t:main-iC} imply that the injective $\C$-closedness if preserved by subsemigroups of commutative semigroups.

\begin{corollary}\label{c:her} Let $\C$ be a class of topological semigroups such that either $\C=\mathsf{T_{\!1}S}$ or  $\mathsf{T_{\!z}S}\subseteq\C\subseteq\mathsf{T_{\!2}S}$. Any subsemgroup of an injectively $\C$-closed commutative semigroup is injectively $\C$-closed.
\end{corollary}

\begin{remark}  Corollary~\ref{c:her} is specific for commutative semigroups and does not generalize to noncommutative groups: by Theorem 1.10 in \cite{CCCS},  every countable bounded group $G$ without elements of order 2 is a subgroup of an absolutely $\mathsf{T_{\!1}S}$-closed countable simple bounded group $X$. If the group $G$ has infinite center, then $G$ is not injectively $\mathsf{T_{\!z}S}$-closed by Theorem~\ref{t:Z}(4) below. On the other hand, $G$ is a subgroup of the absolutely $\mathsf{T_{\!1}S}$-closed group $X$.
\end{remark}

By Theorem~\ref{t:C} (and \ref{t:iCUS}) the (injective) $\mathsf{T_{\!1}S}$-closedness and the (injective) $\mathsf{T_{\!2}S}$-closedness are equivalent for (unipotent) commutative semigroups. The following  example shows that the injective $\mathsf{T_{\!1}S}$-closedness is not equivalent to the injective $\mathsf{T_{\!2}S}$-closedness for commutative semigroups.

\begin{example}\label{ex:main} Let $X$ be the set $\w$ endowed with the semigroup operation $*$ defined by
$$x*y=\begin{cases}
1&\mbox{if $x,y\in \{2n+3:n\in\w\}$ and $x\ne y$};\\
2n&\mbox{if $\{x,y\}\subseteq\{2n,2n+1\}$ for some $n\in\w$};\\
0&\mbox{otherwise}.
\end{cases}
$$The semigroup $X$ has the following properties:
\begin{enumerate}
\item $X$ is injectively $\mathsf{T_{\!2}S}$-closed;
\item $X$ is $\mathsf{T_{\!1}S}$-closed;
\item $X$ is not injectively $\mathsf{T_{\!1}S}$-closed;
\item the set $I=\{0,1\}$ is an ideal in $X$;
\item the quotient semigroup $X/I$ is $\mathsf{T_{\!1}S}$-closed but not injectively $\mathsf{T_{\!z}S}$-closed;
\item the set $J= I\cup E(X)$ is an ideal in $X$ and the quotient semigroup $X/J$ is not $\mathsf{T_{\!z}S}$-closed.
\end{enumerate}
\end{example}

The properties (1)--(6) of Example~\ref{ex:main} will be proved in Section~\ref{s:ex}.

Theorems~\ref{t:main-iC} and \ref{t:mainP} imply the following characterizations of absolutely $\C$-closed commutative semigroups that completes Theorem~\ref{t:aC}.

\begin{corollary}\label{c:aC} Let $\C$ be a class of topological semigroups such that $\mathsf{T_{\!z}S}\subseteq\C\subseteq\mathsf{T_{\!2}S}$. For a commutative semigroup $X$ the following conditions are equivalent:
 \begin{enumerate}
\item $X$ is absolutely $\C$-closed;
\item $X$ is projectively $\C$-closed and injectively $\C$-closed;
\item $X$ is bounded, chain-finite, group-finite and Clifford+finite.
\end{enumerate}
\end{corollary}

\begin{proof} The implication $(1)\Ra(2)$ is trivial, $(2)\Ra(3)$ follows from Theorems~\ref{t:mainP} and \ref{t:main-iC}, and the implication $(3)\Ra(1)$ follows from Theorem~\ref{t:aC}.
\end{proof}

For a semigroup $X$, let
$$Z(X)\defeq\{z\in X:\forall x\in X\;\;(xz=zx)\}$$
be the {\em center} of $X$.

The following theorem collects the results of the papers \cite{BB}, \cite{ACVS}, \cite{ICT1S} and describes some properties of the center of a semigroup possessing various closedness properties.

\begin{theorem}[Banakh--Bardyla]\label{t:Z}  Let $X$ be a semigroup.
\begin{enumerate}
\item If $X$ is $\mathsf{T_{\!z}S}$-closed, then $Z(X)$ is chain-finite, periodic and nonsingular.
\item If $X$ is $\mathsf{T_{\!z}S}$-closed and $X\cdot H(X)\subseteq Z(X)$, then $X$ is group-bounded and $Z(X)$ is $\mathsf{T_{\!1}S}$-closed.
\item If $X$ is viable and ideally $\mathsf{T_{\!z}S}$-closed, then $Z(X)$ is projectively $\mathsf{T_{\!1}S}$-closed.
\item If $X$ is injectively $\mathsf{T_{\!z}S}$-closed, then $Z(X)$ is group-finite and  $\mathsf{T_{\!1}S}$-closed.
\item If $X$ is injectively $\mathsf{T_{\!1}S}$-closed, then $Z(X)$ is injectively $\mathsf{T_{\!1}S}$-closed.
\item If $X$ is absolutely $\mathsf{T_{\!1}S}$-closed, then $Z(X)$ is  absolutely $\mathsf{T_{\!1}S}$-closed.
\end{enumerate}
\end{theorem}

Theorem~\ref{t:Z}(4--6) motivates the following open problem.

\begin{problem} Let $\C$ be a class of topological semigroups such that $\mathsf{T_{\!z}S}\subseteq\C\subseteq\mathsf{T_{\!2}S}$. Is the center of any (injectively) $\C$-closed semigroup (injectively) $\C$-closed?
\end{problem}

The paper is organized as follows. In Section~\ref{s:prelim} we recall some known information and prove some lemmas that will be used in the proof of Theorem~\ref{t:main-iC}. Section~\ref{s:if} contains the proof of the ``if'' part of Theorem~\ref{t:main-iC}. In Sections~\ref{s:b1}, \ref{s:b2} and \ref{s:Cs} we prove three ingredients of the ``only if'' part of Theorem~\ref{t:main-iC}. All pieces are composed together in Section~\ref{s:pf} containing the proof of Theorem~\ref{t:main-iC}. In Section~\ref{s:ex} we check the properties (1)--(6) of the semigroup $X$ from Example~\ref{ex:main}.

\section{Preliminaries}\label{s:prelim}

We denote by $\w$ the set of finite ordinals, by $\IN\defeq\w\setminus\{0\}$ the set of positive integer numbers, and by $\IR_+$ the set of positive real numbers. 
For a set $X$ we denote by $[X]^{<\w}$ the family of all finite subsets of $X$. 

A subset $A\subseteq X$ of a topological space $(X,\Tau)$ is defined to be
\begin{itemize}
\item {\em $\Tau$-open}  if $A\in\Tau$;
\item  {\em $\Tau$-closed} if $X\setminus A\in\Tau$;
\item {\em $\Tau$-clopen} if $A$ is $\Tau$-open and $\Tau$-closed.
\end{itemize}
A topological space $X$ is called $T_0$ if for any distinct points $x,y\in X$ there is an open set $U$ in $X$ such that $U\cap\{x,y\}$ is a singleton. It is easy to see that each $T_0$ zero-dimensional space is Hausdorff (and moreover, Tychonoff).

A {\em poset} is a set $X$ endowed with a partial order $\le$. For an element $a$ of a poset $X$, let  
$${\downarrow}a\defeq\{x\in X :x\le a\}\quad\mbox{and}\quad{\uparrow}a\defeq\{x\in X:a\le x\}$$
be the {\em lower} and {\em upper sets} of $a$ in $X$, respectively. 

For two elements $x,y$ of a poset $X$ we write $x<y$ if $x\le y$ and $x\ne y$.

A subset $A$ of a poset $X$ is called 
\begin{itemize}
\item a {\em chain} if for any $x,y\in A$ either $x\le y$ or $y\le x$;
\item an {\em antichain} if for any $x,y\in A$ with $x\le y$ we have $x=y$.
\end{itemize}

A poset $X$ is called
\begin{itemize}
\item {\em chain-finite} if each chain in $X$ is finite;
\item {\em well-founded} if every nonempty set $A\subseteq X$ contains an element $a\in A$ such that $A\cap{\downarrow}a=\{a\}$.
\end{itemize}
It is easy to see that each chain-finite poset is well-founded.


For a semigroup $X$, let $E(X)\defeq\{x\in X:xx=x\}$ be the set of idempotents of $X$. The set $E(X)$ carries the {\em natural partial order} $\le$ defined by $x\le y$ for $x,y\in E(X)$ iff $xy=yx=x$.

A semigroup $X$ is called 
\begin{itemize}
\item a {\em band} if $X=E(X)$;
\item a {\em semilattice} if $X$ is a commutative band.
\end{itemize}
Each semilattice is considered as a poset endowed with the natural partial order $\le$. It is easy to see that a semilattice is chain-finite as a semigroup if and only if it is chain-finite as a poset. 

For a semigroup $X$, the intersection
$$EZ(X)\defeq E(X)\cap Z(X)=E(Z(X))$$is a semilattice called the {\em central semilattice} of $X$. An element $z\in X$ of a semigroup $X$ is called {\em central} if $z\in Z(X)\defeq\{z\in X:\forall x\in X\;\;(zx=xz)\}$.

By $\two$ we denote the two-element semilattice $\{0,1\}$ endowed with the operation of minimum. This operation induces a natural partial order on $\two$, which coincides with the linear order on $\two$, inherited from the real line. 

For any semigroup $X$ the {\em binary quasiorder} $\lesseq$ on $X$ is the reflexive transitive relation on $X$ defined by $x\lesseq y$ iff $h(x)\le h(y)$ for any homomorphism $h:X\to \two$. 

The obvious properties of the natural partial order $\le$ on $\two$ imply the following properties of the binary quasiorder $\lesseq$.

\begin{lemma}\label{l:quasiorder} For any elements $x,y,z$ of a semigroup $X$, the following properties hold:
\begin{enumerate}
\item if $x\lesseq y$, then $xz\lesseq yz$ and $zx\lesseq zy$;
\item $xy\lesseq yx\lesseq xy$;
\item $x\lesseq x^2\lesseq x$;
\item $xy\lesseq x$ and $xy\lesseq y$.
\end{enumerate}
\end{lemma} 


For an element $a\in X$ and a subset $A\subseteq X$ of a semigroup $X$, consider the sets
$$
{\Uparrow}a\defeq\{x\in X:a\lesseq x\},\quad{\Downarrow}a\defeq\{x\in X:x\lesseq a\}
,\quad{\Updownarrow}a\defeq \{x\in X:a\lesseq x\lesseq a\}={\Uparrow}a\cap{\Downarrow}a
$$
and
$$
{\Uparrow}A\defeq\bigcup_{a\in A}{\Uparrow}a,\quad {\Downarrow}A\defeq\bigcup_{a\in A}{\Downarrow}a,\quad\quad  {\Updownarrow}A\defeq\bigcup_{a\in A}{\Updownarrow}a.
$$
Given two elements $x,y$ of a semigroup, we write $x\Updownarrow y$ if ${\Updownarrow}x={\Updownarrow}y$. By \cite{BH}, the equivalence relation $\Updownarrow$ coincides with the smallest semilattice congruence on $X$. The quotient semilattice $X/_{\Updownarrow}$ is called the {\em semilattice reflection} of $X$. By $q:X\to X/_{\Updownarrow}$ we shall denote the quotient homomorphism of $X$ onto its semilattice reflection. It follows that $${\Uparrow}a=\{x\in X:q(x)q(a)=q(a)\}=q^{-1}[{\uparrow}q(a)]\quad\mbox{and}\quad{\Downarrow}a=\{x\in X:q(x)q(a)=q(x)\}=q^{-1}[{\downarrow}q(a)]$$for every $a\in X$.


A semigroup $X$ is called
\begin{itemize}
\item {\em $E$-central} if $ex=xe$ for any $e\in E(X)$ and $x\in X$;
\item {\em $E_{\Uparrow}$-central} if $ex=xe$ for any $e\in E(X)$ and $x\in{\Uparrow}e$;
\item {\em $E$-hypercentral} if for any $x,y\in X$ with $xy=e\in E(X)$ we have $xe=ex$ and $ye=ey$;
\item {\em $E$-separated} if for any distinct idempotents $x,y\in E(X)$ there exists a  homomorphism $h:X\to \two$ such that $h(x)\ne h(y)$;
\item {\em viable} if for any elements $x,y\in X$ with  $\{xy,yx\}\subseteq E(X)$ we have $xy=yx$.
\end{itemize}

Viable semigroups were introduced and studied by Putcha and Weissglass \cite{PW} who proved that every viable semigroup is $E$-separated. The following characterization of viable semigroups was obtained by the author in \cite{BanE}.

\begin{theorem}[Banakh]\label{t:viable} For a semigroup $X$ the following conditions are equivalent:
\begin{enumerate}
\item $X$ is viable;
\item $X$ is $E_{\Uparrow}$-central;
\item $X$ is $E$-hypercentral;
\item $X$ is $E$-separated.
\end{enumerate}
\end{theorem}

Theorem~\ref{t:viable} implies that every $E$-central semigroup is viable. In particular, every commutative semigroup is viable.

For an element $a$ of a semigroup $X$ the set
$$H_a\defeq\{x\in X:(xX^1=aX^1)\;\wedge\;(X^1x=X^1a)\}$$
is called the {\em $\mathcal H$-class} of $a$.
Here $X^1\defeq X\cup\{1\}$ where $1$ is an element such that $1x=x=x1$ for all $x\in X^1$.

By Corollary 2.2.6 \cite{Howie}, for every idempotent $e\in E(X)$ its $\mathcal H$-class $H_e$ coincides with the maximal subgroup of $X$, containing the idempotent $e$.
The union
$$H(X)\defeq \bigcup_{e\in E(X)}H_e$$is called the {\em Clifford part} of $X$.
The Clifford part is not necessarily a subsemigroup of $X$.

On the other hand,
the {\em central Clifford part}
$$H_Z(X)\defeq \bigcup_{e\in EZ(X)}H_e$$is a subsemigroup of $X$, see Lemma~2.1 in \cite{ICT1S}.

The following lemma is proved in \cite[2.3]{ICT1S}.

\begin{lemma}\label{l:CH} For any semigroup $X$ and elements $x,y\in H(X)$ with $xy=yx$ we have $xy\in H(X)$.
\end{lemma}

The following important lemma is due to Tamura \cite{Tamura82} (see also \cite[4.4]{BH} and \cite[4.2]{BanE} for alternative proofs).

\begin{lemma}[Tamura]\label{l:Tamura} For every idempotent $e$ in a viable semigroup $X$, the maximal subgroup $H_e$ is an ideal in the subsemigroup ${\Uparrow}e$ of $X$ and $e\in Z({\Uparrow}e)$.
\end{lemma}

For a subset $A$ of a semigroup $X$, let
$$\korin{\IN}{A}\defeq\bigcup_{n\in\IN}\korin{n}{A}\quad\mbox{where}\quad\korin{n}{A}\defeq\{x\in X:x^n\in A\}.$$Observe that $\korin{\IN}{A}=\{x\in X:x^\IN\cap A\ne\emptyset\}$ where $$x^\IN\defeq\{x^k:k\in\IN\}$$is the {\em monogenic semigroup} generated by $x$.

 For a point $a\in X$, the set $\korin{\IN}{\{a\}}$ will be denoted by $\korin{\IN}{\,a}$. It is easy to see that $\korin{\IN}{\,a}\subseteq{\Updownarrow}a$ for every $a\in X$. The sets $\korin{\IN}{E(X)}$ and $\korin{\IN}{H(X)}$ are called {\em the periodic part} and {\em the eventually Clifford part} of $X$, respectively.

A semigroup $X$ is called {\em eventually Clifford} if $X=\korin{\IN}{H(X)}$.
It is clear that each periodic semigroup is eventually Clifford (but not vice versa). 

The following two lemmas are proved in \cite[3.1]{BB} and \cite[2.5]{ICT1S}, respectively.

\begin{lemma}\label{l:C-ideal} For any idempotent $e$ of a semigroup we have $(\korin{\IN}{H_e}\cdot H_{e})\cup(H_{e}\cdot \korin{\IN}{H_e}\,)\subseteq H_{e}.$
\end{lemma}

\begin{lemma}\label{l:pi-well-defined} Let $x$ be an element of a semigroup $X$ such that $x^n\in H_e$ for some $n\in\IN$ and $e\in E(X)$. Then $x^m\in H_e$ for all $m\ge n$.
\end{lemma}

For a semigroup $X$, let $\pi:\korin{\IN}{H(X)}\to E(X)$ be the map assigning to each $x\in\korin{\IN}{H(X)}$ the unique idempotent $\pi(x)\in E(X)$ such that $x^n\in H_{\pi(x)}$ for some $n\in\IN$. Lemma~\ref{l:pi-well-defined} ensures that the map $\pi:\korin{\IN}{H(X)}\to E(X)$ is well-defined.

\begin{lemma}\label{l:piZ} $\pi[Z(X)\cap\korin{\IN}{H(X)}]\subseteq EZ(X)$.
\end{lemma}

\begin{proof} Given any $z\in Z(X)\cap\korin{\IN}{H(X)}$, find $n\in\IN$ such that $z^n\in H_{\pi(x)}$. Since $z^n\in Z(X)$, Lemma 2.2 of \cite{ICT1S} implies that $H_{\pi(x)}\cap Z(X)$ is a subgroup of $H_{\pi(x)}$ and hence $\pi(x)\in H_{\pi(x)}\cap Z(X)\subseteq E(X)\cap Z(X)=EZ(X)$.
\end{proof}

\begin{lemma}\label{l:pi1} For any idempotent $e$ of a viable semigroup $X$ we have $$\pi^{-1}[{\uparrow}e]=\korin{\IN}{H(X)}\cap{\Uparrow}e\quad\mbox{and}\quad\pi^{-1}[{\downarrow}e]=\korin{\IN}{H(X)}\cap{\Downarrow}e.$$
\end{lemma}

\begin{proof} If $x\in\pi^{-1}[{\uparrow}e]$, then $e\pi(x)=e$ and hence $x\in\korin{\IN}{H(X)}\cap{\Updownarrow}\pi(x)\subseteq\korin{\IN}{H(X)}\cap{\Uparrow}e$ by Lemma~\ref{l:quasiorder}. If $x\in\korin{\IN}{H(X)}\cap{\Uparrow}e$, then $\pi(x)\in {\Updownarrow}x\subseteq{\Uparrow}e$. By Lemma~\ref{l:Tamura}, $\pi(x)e=e\pi(x)\in H_e$ and hence $\pi(x)e=\pi(x)e=e$ as $e$ is the unique idempotent of the group $H_e$. Therefore, $\pi(x)\in{\uparrow}e$.
\smallskip

If $x\in\pi^{-1}[{\downarrow}e]$, then $e\pi(x)=\pi(x)$ and hence $x\in\korin{\IN}{H(X)}\cap{\Updownarrow}\pi(x)\subseteq\korin{\IN}{H(X)}\cap{\Downarrow}e$, by Lemma~\ref{l:quasiorder}. If $x\in\korin{\IN}{H(X)}\cap{\Downarrow}e$, then $e\in{\Uparrow}x={\Uparrow}\pi(x)$. By Lemma~\ref{l:Tamura}, $\pi(x)e=e\pi(x)\in H_{\pi(x)}$ and hence $\pi(x)e=\pi(x)e=\pi(x)$ as $\pi(x)$ is the unique idempotent of the group $H_{\pi(x)}$. Therefore, $\pi(x)\le e$ and $x\in\pi^{-1}[{\downarrow}e]$.
\end{proof}

Lemma~\ref{l:pi1} implies 

\begin{lemma}\label{l:pi2} Let $X$ be a viable semigroup. For any $e\in E(X)$ we have $${\uparrow}e=E(X)\cap{\Uparrow}e\quad\mbox{and}\quad {\downarrow}e=E(X)\cap{\Downarrow}e.$$
\end{lemma}

Now we recall some information on the topologization of semigroups with the help of remote bases.

Let $X$ be a semigroup. Given two elements $e,b\in X$, consider the set
$$\tfrac{b}e\defeq\{x\in X:xe=b\}$$which can be thought as the set of all left shifts that  move $e$ to $b$. If the set $\frac be$ is not empty, then $\frac be\cdot e=\{b\}$ and for any subset $U\subseteq X$ containing $e$, the set $\frac{b}e\cdot U$ contains $b$.

The assignment $$U\mapsto \Lambdae(b;U)\defeq\{b\}\cup\big(\tfrac{b}e\cdot U\big)$$ will be referred to as the {\em  $e$-to-$b$ shift} of $U$. The following lemma taken from \cite[Lemma 3.1]{ICT1S} describes some properties of $e$-to-$b$-shifts.

\begin{lemma}\label{l:Lamb} Let $e$ be an idempotent of a semigroup $X$. For any elements $a,b\in X$ and subsets $U,V,W\subseteq X$ the following statements hold:
\begin{enumerate}
\item If $V\subseteq W$, then $\Lambdae(b;V)\subseteq \Lambdae(b;W)$.
\item $\Lambdae(b;\frac ee)\subseteq\tfrac{be}{e}$.
\item If $b\ne be$, then $\Lambdae(b;\frac ee)=\{b\}$.
\item If $a\in\Lambdae(b;\frac ee)\setminus\{b\}$, then $\Lambdae(a;\frac ee)=\{a\}$.
\item If $a\ne b$ and $\Lambdae(a;\frac ee)\cap\Lambdae(b;\frac ee)\ne \emptyset$, then either $\Lambdae(a;\frac ee)=\{a\}$ or $\Lambdae(b;\frac ee)=\{b\}$.
\item If $V\subseteq W$, then $a\cdot\Lambdae(b;V)\subseteq \Lambdae(ab;W)$.
\item If $Ub\subseteq bW$ and $be=eb$, then $\Lambdae(a;U)\cdot b\subseteq\Lambdae(ab;W)$;
\item If $e\in Z(X)$, $V\subseteq W$, $Ub\subseteq bW$ and $\forall y\in \frac{b}e\;(UyV\subseteq yW)$, then
 $\Lambdae(a;U)\cdot\Lambdae(b;V)\subseteq \Lambdae(ab;W)$.
\end{enumerate}
\end{lemma}

\begin{definition}\label{d:remote}  Let $X$ be a semigroup and $e$ be a central idempotent in $X$. An {\em $e$-base} on $X$ is a family $\Zeta$ of subsemigroups of $X$ such that for every $U,W\in\Zeta$ there exists $V\in\Zeta$ such that $VV\subseteq V\subseteq U\cap W\subseteq\frac ee\cap Z(X)$.
\smallskip

Given an $e$-base $\Zeta$, let  $\Tau_\Zeta$ be the topology on $X$, consisting of all sets $W\subseteq X$ such that for every $x\in W$ there exist a set $U\in\Zeta$ such that $x\in\Lambdae(x;U)\subseteq W$. The topology $\Tau_\Zeta$ will be referred to as {\em the topology generated by the $e$-base $\Zeta$}.
\end{definition}

Lemma~\ref{l:Lamb}(1,4) implies the following lemma.

\begin{lemma}\label{l:base-e} Let $e$ be a central idempotent in a semigroup $X$ and $\Zeta$ be an $e$-base. For every $x\in X$ the family $$\mathcal B_x\defeq\big\{\Lambdae(x;V):V\in\Zeta\big\}$$ is a neighborhood base of the topology $\Tau_\Zeta$ at $x$.
\end{lemma}

\begin{definition}\label{d:regular} Let $X$ be a semigroup and $e$ be a central idempotent in $X$. An $e$-base $\Zeta$  is defined to be {\em regular} if for any element $b\in X$ with $b\ne be$, there exists a set $V\in \Zeta$ such that $b\notin \frac{be}{e}\cdot V$.
\end{definition}

The statements (1)--(4) of the following theorem are proved in Theorem 4.4 of \cite{ICT1S}, the statements (5a) and (5b) are proved in Propositions 4.5 and 4.6 of \cite{ICT1S}, respectively.

\begin{theorem}\label{t:TS} Let $X$ be a  semigroup, $e$ be a central idempotent in $X$ and $\Zeta$ be an $e$-base. Then
\begin{enumerate}
\item $(X,\Tau_{\Zeta})$ is a topological semigroup;
\item $(X,\Tau_{\Zeta})$ is a $T_0$ topological space with discrete subspace of non-isolated points.
\item If the $e$-base $\Zeta$ is regular, then for every point $b\in X$, any subset $B\subseteq \Lambdae(b;\frac ee)$ containing $b$ is closed in the topology $\Tau_{\Zeta}$.
\item If the $e$-base $\Zeta$ is regular, then every topology $\tau$ on $X$ with $\Tau_\Zeta\subseteq\tau$ is Hausdorff and zero-dimensional.
\item The $e$-base $\Zeta$ is regular if one of the following conditions is satisfied:
\begin{enumerate}
\item the poset $EZ(X)$ is well-founded and for every finite set $F\subseteq  EZ(X)\setminus{\downarrow}e$, there exists a set $V\in\Zeta$ such that $V\subseteq H(X)\cap Z(X)\setminus\pi^{-1}[{\uparrow}F]$;
\item the semigroup $X$ is nonsingular and eventually Clifford, the poset $E(X)$ is well-founded, and for every $n\in\IN$ and finite set $F\subseteq E(X)\setminus{\downarrow}e$, there exists a set $V\in\Zeta$ such that $V\subseteq \{z^m:z\in  Z(X)\cap\frac ee \setminus\pi^{-1}[{\uparrow} F],\;m\ge n\}$.
\end{enumerate}
\end{enumerate}
\end{theorem}

Finally, we present two important examples of $e$-bases. Those examples are taken from Theorem 5.3 of \cite{ICT1S}.

\begin{example}\label{ex:He} For any central idempotent $e$ of a semigroup $X$, the family
$$\mathcal H[e)\defeq\big\{Z(X)\cap H(X)\cap\tfrac ee\setminus\pi^{-1}[{\uparrow}F]:F\in[EZ(X)\setminus{\downarrow}e]^{<\w}\big\}$$
is an $e$-base on $X$. If the poset $EZ(X)$ is well-founded, then the $e$-base $\mathcal H[e)$ is regular.
\end{example} 

\begin{example}\label{ex:Ze} For any central idempotent $e$ of a semigroup $X$ the family
$$\mathcal Z[e)\defeq\big\{\{z^n:z\in Z(X)\cap\tfrac ee\cap\pi^{-1}[E(X)\setminus {\uparrow}F]\}:n\in\IN,\;F\in[E(X)\setminus{\downarrow}e]^{<\w}\big\}$$
is an $e$-base on $X$. If the poset $E(X)$ is well-founded and the semigroup $X$ is nonsingular and eventualy Clifford, then the $e$-base $\mathcal Z[e)$ is regular.
\end{example}

Finally, we recall some basic information on filters. A nonempty family $\F$ of nonempty subsets of a set $X$ is called a {\em filter} on $X$ if $\F$ is closed under taking intersections and supersets. For a filter $\F$ on a set $X$, a subfamily $\mathcal B\subseteq\F$ is called a {\em base} for $\F$ if $\F=\{F\subseteq X:\exists B\in\mathcal B\;\;(B\subseteq F)\}$. A filter $\F$ on $X$ is called {\em free} if $\bigcap\F=\emptyset$.

For a set $X$ we denote by $\Fil(X)$ the set of all filters on $X$. The set $\Fil(X)$ is partially ordered by the inclusion relation $\subseteq$. Maximal elements of the poset $\Fil(X)$ are called {\em ultrafilters}. By the Kuratowski--Zorn Lemma, each filters can be enlarged to an ultrafilter. The set of ultrafilters on $X$ is denoted by $\beta(X)$. Each element $x\in X$ can be identified with the {\em principal ultrafilter} $\F_x=\{F\subseteq X:x\in F\}$. So, $X$ can be identified with the subset of principal ultrafilters in $\beta(X)$. Therefore, we have the chain of inclusions $X\subseteq\beta(X)\subseteq\Fil(X)$. We shall often use the following characteristic property of ultrafilters, see \cite[1.8]{Prot}, \cite[3.6]{HS}.

\begin{lemma}\label{l:ultra} A filter $\U$ on a set $X$ is an ultrafilter if and only if for any finite family $\F$ of sets with $\bigcup\F\in\U$, there exists a set $F\in\F\cap\U$.
\end{lemma}

We say that a filter $\F$ on a topological space $X$ is {\em convergent} if $\F$ {\em converges} to some point $x\in X$, which means that any neighborhood $O_x$ of $x$ in $X$ belongs to the filter $\F$. If the space $X$ is Hausdorff, then the point $x$ is unique. It is called the {\em limit} of the filter $\F$ and is denoted by $\lim\F$.

\section{The proof of the ``if'' part of Theorem~\ref{t:main-iC}}\label{s:if}

The ``if'' part of Theorem~\ref{t:main-iC} is established in the following (a bit more informative) proposition.

\begin{proposition}\label{p:if} Let $Y$ be a Hausdorff topological semigroup and $X$ be a commutative subsemigroup of\/ $Y$. If $X$ is chain-finite, group-finite, bounded, nonsinglar and not Clifford-singular, then $X$ is closed in $Y$ and $X\setminus H(X)$ is a closed discrete subspace of $X$.
\end{proposition}

\begin{proof} Replacing $Y$ by the closure of $X$ in $Y$, we can assume that $X$ is dense in $Y$. The commutativity of $X$ and the density of $X$ in the Hausdorff topological semigroup $Y$ imply that the semigroup $Y$ is commutative.

Since the semigroup $X$ is chain-finite, so is the semilattice $E(X)$. By Theorem~\ref{t:aC}, $E(X)$ is closed in $Y$. Since the semigroup $X$ is bounded, there exists a number $p\in\IN$ such that $x^p\in E(X)$ for all $x\in X$. The commutativity of $Y$ ensures that the map $\pi:Y\to Y$, $\pi:y\mapsto y^p$, is a continuous homomorphism. Taking into account that $\pi[X]\subseteq E(X)$, we conclude that $$E(X)\subseteq E(Y)\subseteq  \pi[Y]=\pi[\overline X]\subseteq \overline{\pi[X]}=\overline{E(X)}=E(X)$$and hence $E(X)=E(Y)=\pi[Y]$.

\begin{lemma}\label{l:H-closed} The Clifford part $H(X)$ of $X$ is closed in $Y$.
\end{lemma}

\begin{proof} By Lemma~\ref{l:CH}, the Clifford part $H(X)$ of $X$ is a subsemigroup of $X$. By Theorem~\ref{t:aC}, the bounded chain-finite group-finite Clifford semigroup $H(X)$ is absolutely $\mathsf{T_{\!2}S}$-closed and hence $H(X)$ is closed in $Y$.
\end{proof}

\begin{lemma}\label{cl:cl-disc} For every $e\in E(X)$ the set $X\cap \pi^{-1}(e)$ is closed and discrete in $Y$.
\end{lemma}

\begin{proof} Since $\pi:Y\to Y$ is a homomorphism, for every $e\in E(X)$ the preimage $\pi^{-1}(e)$ is a subsemigroup of $Y$ and hence $X\cap \pi^{-1}(e)$ is a subsemigroup of $X$. Since $X$ is bounded, nonsingular and group-finite, so is its unipotent subsemigroup $X\cap\pi^{-1}(e)$.  By Theorem~\ref{t:iCUS}, the  semigroup $X\cap \pi^{-1}(e)$ is injectively $\mathsf{T_{\!1}S}$-closed and by Theorem~\ref{t:iT1}, $X\cap \pi^{-1}(e)$ is closed and discrete in  $Y$.
\end{proof}

\begin{lemma}\label{l:up-clopen} For every idempotent $e\in E(X)$ the set ${\Uparrow}e=\{y\in Y:\pi(y)e=e\}$ is clopen in $Y$.
\end{lemma}

\begin{proof} First observe that the set ${\uparrow}e=\{x\in E(X):xe=e\}$ is closed in $E(X)$ being the preimage of the singleton $\{e\}$ under the continuous map $E(X)\to E(X)$, $x\mapsto xe$. Since the semilattice $E(X)$ is chain-finite, so is its subsemilattice $E(X)\setminus{\uparrow}e$. By Theorem~\ref{t:aC}, the chain-finite semilattice $E(X)\setminus{\uparrow}e$ is absolutely $\mathsf{T_{\!2}S}$-closed and hence $E(X)\setminus{\uparrow}e$ is closed in $E(X)\subseteq Y$ and the set ${\uparrow}e$ is open in $E(X)$. Then the set ${\Uparrow}e=\pi^{-1}[{\uparrow}e]$ is clopen in $Y$ being the preimage of the clopen set ${\uparrow}e$ under the continuous map $\pi$.
\end{proof}

For a number $n\in\IN$ and subset $A\subseteq X$, let
$$A^{{\uparrow}n}\defeq\{a^n:a\in A\}.$$

\begin{lemma}\label{l:up-k} For every  $k>p$ we have $Y^{{\uparrow}k}\subseteq H(X)$.
\end{lemma}

\begin{proof} First observe that for every $k>p$, $x\in X$ and the idempotent $e=x^p=\pi(x)$ we have $x^k=x^px^{k-x}\in e\korin{\IN}{H_e}\subseteq H_e\subseteq H(X)$ by Lemma~\ref{l:C-ideal}. The continuity of the map $Y\to Y$, $y\mapsto y^k$ and Lemma~\ref{l:H-closed} imply that $Y^{{\uparrow}k}=\overline X^{{\uparrow}k}\subseteq\overline{X^{{\uparrow}k}}\subseteq \overline{H(X)}=H(X)$.
\end{proof}

\begin{lemma}\label{l:improve} Let $S$ be a subsemigroup of $X$. If some free convergent ultrafilter $\U$ on $Y$ contains  the set $S\setminus H(X)$, then there exists a free convergent ultrafilter $\V$ on $Y$ such that $S\setminus H(X)\in\V$ and $V^{{\uparrow}2}\subseteq H(X)$ for some $V\in\V$.
\end{lemma}

\begin{proof}  Let $k\in\IN$ be the largest number such that for every $U\in\U$ the set $U^{{\uparrow}k}$ is not contained in $H(X)$. Lemma~\ref{l:up-k} ensures that the number $k$ is well-defined. By the maximality of $k$, there exists a set $U\in\U$ such that $U^{{\uparrow}(k+1)}\subseteq H(X)$.   The choice of $k$ ensures that the set $U'=\{x\in U:x^k\in H(K)\}$ does not belong to the ultrafilter $\U$. Then $U\setminus U'\in\U$. Replacing the set $U$ by  $S\cap U\setminus U'$ we can assume that $U^{{\uparrow}k}\subseteq S\setminus H(X)$.

We claim that for every $V\in\U$ the set $V^{{\uparrow}k}$ is infinite. In the opposite case we can apply Lemma~\ref{l:ultra} and find a set $V\in\U$ such that $V\subseteq U$ and $V^{{\uparrow}k}$ is a singleton. Then $V\subseteq X\cap \pi^{-1}(e)$ where $e$ is the unique element of the singleton $\pi[V^{{\uparrow}k}]$. By Lemma~\ref{cl:cl-disc}, the subset $X\cap\pi^{-1}(e)$ is closed and discrete in $Y$. So, it cannot belong to the free convergent ultrafilter $\U$. This contradiction shows the family $\{V^{{\uparrow}k}:V\in\U\}$ consists of infinite sets and hence is a base of some free ultrafilter $\V$ that converges to $(\lim\U)^k$.

It follows from $U^{{\uparrow} k}\subseteq S\setminus H(X)$ that $S\setminus H(X)\in\V$. On the other hand, for every $u\in U$ we have $u^{k+1}\in H(X)$ and hence $u^{2k}\in H(X)$ by Lemma~\ref{l:pi-well-defined}. Then for every element $v\in U^{{\uparrow} k}\in\V$ we have $v^2\in H(X)$.
\end{proof}

\begin{lemma}\label{l:key} No free convergent ultrafilter on $Y$ contains the set $X\setminus H(X)$.
\end{lemma}

\begin{proof} To derive a contradiction, assume that some free convergent ultrafilter on $Y$ contains the set $X\setminus H(X)$. For every $e\in E(X)$, let ${\Downarrow}e=\{x\in X:\pi(x)e=\pi(x)\}$. Let $M$ be the set of all idempotents $e\in E(X)$ for which there exists a free convergent ultrafilter on $Y$ that contains the set ${\Downarrow}e\setminus H(X)$. If $M=\emptyset$, then put $\mu=1$ and ${\Downarrow}\mu=Y$ where $1\notin Y$ is an element such that $1y=y=y1$ for all $y\in Y$. If $M$ is not empty, then there exists an element $\mu\in M$ such that $M\cap{\downarrow}\mu=\{\mu\}$ where ${\downarrow}\mu=\{x\in E(X):x\mu=x\}$. The element $\mu$ exists since the semilattice $E(X)$ is chain-finite. By Lemma~\ref{l:improve}, there exists a free convergent ultrafilter $\U$ on $Y$ such that  ${\Downarrow}\mu\setminus H(X)\in\U$ and $W^{{\uparrow}2}\subseteq H(X)$ for some set $W\in\U$ with $W\subseteq {\Downarrow}\mu\setminus H(X)$. Let $\lambda$ be the limit of the ultrafilter $\U$ in $Y$. Since the Clifford part $H(X)$ of $X$ is closed in $Y$, $$\lambda^2\in\overline{W^{{\uparrow}2}}\subseteq\overline{H(X)}=H(X).$$
Let $e=\pi(\lambda)$. By Lemma~\ref{l:up-clopen}, the set ${\Uparrow}e$ is clopen in $Y$. Since ${\Uparrow}e$ contains $\lambda=\lim\U$, we can replace $W$ by the subset $W\cap {\Uparrow}e$ and assume that $W\subseteq{\Uparrow}e$.

By Lemma~\ref{cl:cl-disc}, the set $\pi^{-1}(\mu)$ does not belong to the free convergent ultrafilter $\U$. So, we can replace $W$ by $W\setminus\pi^{-1}(\mu)$ and assume that $$W\subseteq{\Uparrow}e\cap{\Downarrow}\mu\setminus\pi^{-1}(\mu).$$

\begin{claim}\label{cl:dichotomy} For every $a\in X$ with $e\le\pi(a)<\mu$, there exists a set $U\in\U$ such that either  $aU\subseteq H(X)$ or $aU=\{a\lambda\}$.
\end{claim}

\begin{proof} If $H(X)\in a\U$, then $aU\subseteq H(X)$ for some $U\in\U$ and we are done. So, assume that $H(X)\notin a\U$. Observe that the set $aW\subseteq{\Downarrow}\pi(a)$ belongs to the ultrafilter $a\U$ and hence ${\Downarrow}\pi(a)\setminus H(X)\in a\U$. Since $\pi(a)<\mu$, the minimality of $\mu$ ensures that the ultrafilter $a\U$ is not free and so $a\U$ is principal. Then $\lim a\U=a\lim\U=a\lambda$ and there exists a set $U\in\U$ such that $aU=\{a\lambda\}$. 
\end{proof}

\begin{claim}\label{cl:PinU} The set $P\defeq\{a\in W:H(X)\notin a\U\}$ belongs to the ultrafilter $\U$.
\end{claim}

\begin{proof} Assuming that $P\notin \U$, we conclude that the set $W\setminus P=\{x\in W:H(X)\in a\U\}$ belongs to $\U$.

We shall inductively construct sequences of points $(a_n)_{n\in\w}$ and sets $(U_n)_{n\in\w}$ such that $a_0\in U_0=W\setminus P$ and for every $n\in\w$ the following conditions are satisfied:
\begin{itemize}
\item[{\rm(i)}] $a_n\in U_n\in\U$,
\item[{\rm(ii)}] $U_{n+1}\subseteq U_n\setminus\{a_n\}$;
\item[{\rm(iii)}] $a_nU_{n+1}\subseteq H(X)$.
\end{itemize}

We start the inductive construction letting $U_0=W\setminus P$ and $a_0$ be any point of $U_0$. Assume that for some $n\in\w$ we have constructed the sets $U_0,\dots,U_n$ and the point $a_0,\dots,a_n$. Since $a_n\in U_n\subseteq U_0=W\setminus P$, the set $H(X)$ belongs to the ultrafilter $a_n\U$. So, we can choose a set $U_{n+1}\in\U$ satisfying the inductive conditions (ii),(iii), and then choose any element $a_{n+1}\in U_{n+1}$. This completes the inductive step.
\smallskip

The inductive conditions (i) and (ii) ensure that $A=\{a_n\}_{n\in\w}$ is an infinite subset of $U_0\subseteq W\subseteq X\setminus H(X)$ such that $a_na_m\in H(X)$ for all $n<m$. Since the semigroup $X$ is commutative, we conclude that $a_na_m\in H(X)$ for any distinct $n,m\in\w$. On the other hand, the choice of the set $W$ with $W^{{\uparrow}2}\subseteq H(X)$ ensures that $a_n^2\in H(X)$ for all $n\in\w$. Therefore, $A$ is an infinite subset of $X\setminus H(X)$ such that $AA\subseteq H(X)$, which implies that the semigroup $X$ is Clifford-singular. But this contradicts our assumption.
\end{proof}

By Claim~\ref{cl:PinU}, $P\in\U$. We shall inductively construct sequences of points $(a_n)_{n\in\w}$ and sets $(U_n)_{n\in\w}$ such that $a_0\in U_0=P$ and for every $n\in\w$ the following conditions are satisfied:
\begin{itemize}
\item[{\rm(iv)}] $a_n\in U_n\in\U$;
\item[{\rm(v)}] $U_{n+1}\subseteq U_n$;
\item[{\rm(vi)}] $a_nU_{n+1}=\{a_n\lambda\}\in X\setminus H(X)$;
\item[{\rm(vii)}] $a_n\lambda\notin U_{n+1}U_{n+1}$.
\end{itemize}

We start the inductive construction letting $U_0=P$ and $a_0$ be any point of $U_0$. Assume that for some $n\in\w$ we have constructed the sets $U_0,\dots,U_n$ and the point $a_0,\dots,a_n$. Since $a_n\in U_n\subseteq U_0=P$, the ultrafilter $a_n\U$ is principal and $H(X)\not\subseteq a_n\U$, see Claim~\ref{cl:dichotomy}. Then we can find a set $U_{n+1}\in\U$ such that  $a_nU_{n+1}$ is a singleton. Since the ultrafilter $\U$ converges to $\lambda$, the principal ultrafilter $a_n\U$ converges to $a_n\lambda$ and hence $a_nU_{n+1}=\{a_n\lambda\}$. It follows from $H(X)\notin a_n\U$ that $a_n\lambda\notin H(X)$. Since $\lambda^2\in H(X)$, we can replace $U_{n+1}$ by a smaller set in $\U$ and assume that $U_{n+1}\subseteq U_n$ and $U_{n+1}U_{n+1}\cap\{a_n\lambda\}=\emptyset$. Choose any element $a_{n+1}\in U_{n+1}$ and observe that $a_{n+1}$ and $U_{n+1}$ satisfy the inductive assumptions (iv)--(vii). This completes the inductive step.

After completing the inductive construction, consider the set $A=\{a_n\lambda\}_{n\in\w}\subseteq X\setminus H(X)$. We claim that this set is infinite. Indeed, for any numbers $n<m$ we have $a_m\in U_m\subseteq U_{n-1}$ and hence $\{a_m\lambda\}=a_mU_{m+1}\subseteq  U_{n-1}U_{n-1}\subseteq X\setminus\{a_n\lambda\}$ by the inductive conditions (vi) and (vii). For every $n\in \w$ we have $a_n\in U_n\subseteq W\subseteq{\Uparrow}e$ and hence $\pi(a_n\lambda)=\pi(\lambda)=e$. For any $n<m$ we have
 $a_n a_{m}\in a_nU_{n+1}=\{a_n\lambda\}\subseteq \korin{\IN}{e}$ and finally $$a_n\lambda a_m\lambda= a_na_m\lambda^2\in \korin{\IN}{e}\cdot H_e\subseteq H_e\subseteq H(X),$$by Lemma~\ref{l:C-ideal}. The commutativity of $X$ implies that $a_n\lambda a_m\lambda\in H(X)$ for any distinct $n,m\in\w$. On the other hand, for every $n\in\w$ the choice of the set $W$ with $W^{{\uparrow}2}\subseteq H(X)$ ensures that
 $$a_n\lambda a_n\lambda=a_n^2\lambda^2\in H(X)\cdot H_e\subseteq H(X).$$
 Therefore, $A=\{a_n\lambda\}_{n\in\w}$ is an infinite subet of $X\setminus H(X)$ such that $AA\subseteq H(X)$, which implies that $X$ is Clifford-singular. But this contradicts our assumption.
 \end{proof}

Now we are ready to prove that the set $X$ is closed in $Y$ and the set $X\setminus H(X)$ is closed and discrete in $X$. Assuming that $X$ is not closed in $Y$, we can find a point $y\in \bar X\setminus X$. Let $\F$ be the  filter generated by the base $O_y\cap X$ where $O_y$ runs over neighborhoods of $y$ in $Y$. By Lemma~\ref{l:H-closed}, the Clifford part $H(X)$ is closed in $Y$. Then  the set $X\setminus H(X)$ belongs to the filter $\F$. By the Kuratowski--Zorn Lemma, the filter $\F$ is contained in some ultrafilter $\U$. The ultrafilter $\U$ is free and converges to $y$. Also $X\setminus H(X)\in\F\subseteq\U$. But the existence of such an ultrafilter $\U$ is forbidden by Lemma~\ref{l:key}.

Next, we show that the set $X\setminus H(X)$ is closed and discrete in $X$. In the opposite case, we could choose a free ultrafilter $\U$ with $X\setminus H(X)\in\U$ that converges to some point of the space $X$. But the existence of such an ultrafilter contradicts  Lemma~\ref{l:key}.
\end{proof}

\section{A boundedness property of  injectively $\mathsf{T_{\!z}S}$-closed semigroups}\label{s:b1}

In this section we prove one boundedness property of an injectively $\mathsf{T_{\!z}S}$-closed commutative semigroup.  First we prove that the injective $\C$-closedness is preserved by attaching an external unit. 

A semigroup $X$ is called {\em unital} if it contains an element $1$ such that $1x=x=x1$ for all $x\in  X$.

For a semigroup $X$ we denote by $X^1$ the unital semigroup $X^1=\{1\}\cup X$ where $1\notin X$ is an element such that $x1=x=1x$ for all $x\in X^1$.

\begin{lemma}\label{l:+1} Let $\C$ be a class of $T_1$ topological semigroups. For every injectively $\C$-closed semigroup $X$ the unital semigroup $X^1$ is injectively $\C$-closed.
\end{lemma}

\begin{proof} Let $h:X^1\to Y$ be any injective homomorphism to a topological semigroup $Y\in\C$. Since $X$ is injectively $\C$-closed, the image $h[X]$ is closed in $Y$. Since the singleton $\{h(1)\}$ is closed in the $T_1$ space $Y$, the image $h[X^1]=h[X]\cup\{h(1)\}$ is closed in $Y$, witnessing that the semigroup $X^1$ is injectively $\C$-closed.
\end{proof}

\begin{proposition}\label{p:vsim-grobikam-grobik} Let $X$ be an injectively $\mathsf{T_{\!z}S}$-closed semigroup. If the poset $E(X)$ is well-founded and the semigroup $X$ is viable, nonsingular and eventually Clifford, then $Z(X)\subseteq\sqrt[l]{H(X)}$ for some $l\in\IN$.
\end{proposition}


\begin{proof} By Lemma~\ref{l:+1}, we lose no generality assuming that the semigroup $X$ is unital and hence contains an element $1$ such that $x1=x=1x$ for all $x\in X$. Let $\pi:X\to E(X)$ be the function assigning to each $x\in X$ the unique idempotent $\pi(x)\in E(X)$ such that $x\in \korin{\IN}{H_{\pi(x)}}$. Since $X$ is eventually Clifford, the map $\pi$ is well-defined.

Let $q:X\to X/_{\Updownarrow}$ be the quotient homomorphism of $X$ onto its semilattice reflection. Since $X$ is viable and eventually Clifford, the restriction $q{\restriction}_{E(X)}:E(X)\to X/_{\Updownarrow}$ is bijective and for every idempotent $e\in E(X)$, the set ${\Updownarrow}e=q^{-1}(q(e))=\korin{\IN}{H_e}=\pi^{-1}(\pi(e))$ is a unipotent subsemigroup of $X$. It follows that $\pi(x)\in{\Updownarrow}x=\korin{\IN}{H_{\pi(x)}}$ for every $x\in X$.

By Theorem~\ref{t:Z}(1,4), the semigroup $Z(X)$ is chain-finite, periodic, nonsingular and group-finite.

\begin{claim}\label{cl:Z-korin-bounded} For every idempotent $e\in E(X)$, the commutative subsemigroup $Z(X)\cap\korin{\IN}{H_e}$ of $X$ is bounded.
\end{claim}

\begin{proof} Since $Z(X)$ is chain-finite, periodic, nonsingular and group-finite,  the unipotent subsemigroup $Z(X)\cap\korin{\IN}{H_e}$ of $Z(X)$ also is  chain-finite, periodic, nonsingular and group-finite. By Theorems~\ref{t:C} and \ref{t:CUS}, the unipotent semigroup $Z(X)\cap\korin{\IN}{H_e}$ is $\mathsf{T_{\!1}S}$-closed and bounded.
\end{proof}

To derive a contradiction, assume that $Z(X)\not\subseteq\korin{\ell}{H(X)}$ for all $\ell\in\IN$. The periodicity of $Z(X)$ implies that $Z(X)\subseteq \bigcup_{e\in EZ(X)}{\Updownarrow}e\subseteq\bigcup_{e\in EZ(X)}{\Uparrow}e$. 
Since the central semilattice $EZ(X)=E(Z(X))$ is chain-finite and $Z(X)\not\subseteq \korin{\ell}{H(X)}$ for all $\ell\in\IN$, there exists an idempotent $e\in EZ(X)$ such that 
\begin{itemize}
\item $Z(X)\cap {\Uparrow}e\not\subseteq\korin{\ell}{H(X)}$ for all $\ell\in\IN$, and
\item for every idempotent $u\in EZ(X)$ with $e<u$ there exists $\ell\in\IN$ such that $Z(X)\cap{\Uparrow}u\subseteq\korin{\ell}{H(X)}$. 
\end{itemize}

For every $n\in\IN$, let $$Z_{n!}\defeq\{z^{n!}:z\in Z(X)\cap\tfrac ee\}.$$

\begin{claim}\label{cl:antichain} There exists a sequence $(z_k)_{k\in\w}\in\prod_{k\in\w}Z_{k!}$ such that
\begin{enumerate}
\item $z_k\notin{\Uparrow}z_n$ and $\pi(z_k)\not\le\pi(z_n)$ for any distinct numbers $k,n\in\w$;
\item $z_k\notin H(X)$ and $z_k^2\in H(X)$ for every $k\in\w$.
\end{enumerate}
\end{claim}

\begin{proof} By the choice of $e$, for every $k\in\w$ there exists an element $\zeta_k\in Z(X)\cap {\Uparrow}e\setminus\korin{k!}{H(X)}$. Let $[\w]^2$ be the family of two-element subsets of $\w$. Consider the function $\chi:[\w]^2\to\{0,1,2\}$ defined by
$$\chi(\{n,m\})=\begin{cases}0&\mbox{if $\pi(\zeta_n)=\pi(\zeta_m)$};\\
1&\mbox{if $\pi(\zeta_n)<\pi(\zeta_m)$ or $\pi(\zeta_m)<\pi(\zeta_n)$};\\
2&\mbox{otherwise}
\end{cases}
$$
By the Ramsey Theorem 5 \cite{Ramsey}, there exists an infinite set $\Omega\subseteq \w$ such that $\chi[[\Omega]^2]=\{c\}$ for some $c\in\{0,1,2\}$. If $c=0$, then the set $\{\pi(z_n)\}_{n\in\Omega}$ contains a unique idempotent $u$ and hence $\{\zeta_k\}_{k\in\Omega}\subseteq Z(X)\cap \korin{\IN}{H_u}$. By Claim~\ref{cl:Z-korin-bounded}, the semigroup $Z\cap\korin{\IN}{H_u}$ is bounded and hence $Z(X)\cap\korin{\IN}{H_u}\subseteq\korin{\ell}{u}$ for some $\ell\in\IN$. Then for any $k\in\Omega$ with $k\ge \ell$ we have $\zeta_k^{k!}=(z_k^\ell)^{k!/\ell}=u^{k!/\ell}=u\in  E(X)\subseteq H(X)$, which contradicts the choice of $\zeta_k\notin \korin{k!}{H(X)}$. Therefore, $c\ne 0$. If $c=1$, then the set $\{\pi(\zeta_k)\}_{k\in\Omega}$ is an infinite chain in $EZ(X)$ which is not possible as $Z(X)$ is chain-finite. Therefore, $c=2$ and hence $\{\pi(\zeta_k)\}_{k\in\Omega}$ is an infinite antichain in the poset $EZ(X)$. Lemma~\ref{l:pi2} ensures that $\pi(z_k)\notin{\Uparrow}\pi(z_n)$ and hence $z_k\notin{\Uparrow}z_n$ for any distinct numbers $k,n\in\Omega$. 
\smallskip

By Claim~\ref{cl:Z-korin-bounded}, the semigroup $Z(X)\cap\korin{\IN}{H_e}$ is bounded and hence $Z(X)\cap\korin{\IN}{H_e}\subseteq\korin{\ell}{e}$ for some $\ell$.
Replacing $\Omega$ by a smaller infinite set, we can assume that $\ell\le\min\Omega$.
For every $k\in\Omega$, choose a unique number $p_k\in\IN$ such that $\zeta_k^{k!p_k}\notin H(X)$ but $\zeta_k^{k!(p_k+1)}\in H(X)$. The number $p_k$ exists since $\zeta_k^{k!}\notin H(X)$ and the semigroup $Z(X)$ is periodic. Lemma~\ref{l:pi-well-defined} implies that $(\zeta_k^{k!})^{2p_k}\in H(X)$.

By Lemma~\ref{l:Tamura}, $H_e$ is an ideal in ${\Uparrow}e$. This implies $\zeta_ke\in Z(X)\cap H_e$ and hence $\zeta_k^{k!p_k}e=\big((\zeta_ke)^\ell\big)^{k!p_k/\ell}=e$ and  $\zeta_k^{k!p_k}\in\frac ee$. 

 Write the infinite set $\Omega$ as $\{n_k\}_{k\in\w}$ for some strictly increasing sequence $(n_k)_{k\in\w}$. Put $z_k\defeq\zeta_{n_k}^{n_k!p_{n_k}}$ for every $k\in\w$ and observe that the sequence $(z_k)_{k\in\w}$ belongs to $\prod_{k\in\w}Z_{k!}$ and satisfies the conditions (1), (2) of Claim~\ref{cl:antichain}. 
\end{proof}

Fix any free ultrafilter $\U$ on the set $\w$. Consider the set $$E_\U\defeq\big\{u\in E(X):\{n\in\w:z_nu=z_n\}\in\U\big\}$$ and observe that $E_\U$ contains the element  $1\in E(X)$ and hence $E_\U$ is not empty. Since the poset $E(X)$ is well-founded, there exists an idempotent $s\in E_\U$ such that $E_\U\cap{\downarrow}s=\{s\}$. 

\begin{claim}\label{cl:s-min} $E_\U\subseteq{\uparrow}s$.
\end{claim}

\begin{proof}  Take any idempotent $u\in E_\U$. By the definition of $E_\U\ni u,s$, the set $$U\defeq\{n\in\w:z_nu=z_n\}\cap\{n\in\w:z_ns=z_n\}$$ belongs to the ultrafilter $\U$. Observe that for every $n\in U$ we have $z_nus=z_ns=z_n$ and hence $z_n(us)^k=z_n$ for every $k\in\IN$. Since $X$ is eventually Clifford, there exists $k\in\IN$ such that $(us)^k\in H_{\pi(us)}$ and hence $(us)^k=(us)^k\pi(us)$.
Then for every $n\in U$ we have
$$z_n=z_n(us)^k=z_n(us)^k\pi(us)=z_n\pi(us)$$and hence $\pi(us)\in E_\U$. It follows from $\pi(us)\Updownarrow us\lesseq s$ and Lemma~\ref{l:pi2} that $\pi(us)\le s$ and hence $\pi(us)=s$ by the minimality of $s$. Also $s=\pi(us)\Updownarrow us\lesseq u$ and Lemma~\ref{l:pi2} imply $s\le u$ and $u\in{\uparrow}s$.
\end{proof}     

By the definition of the set $E_\U$, the set $$U_s\defeq\{n\in\w:z_ns=z_n\}$$ belongs to the free ultrafilter $\U$ and hence is infinite.

\begin{claim}\label{cl:e<s} $e<\pi(z_n)<s$ for every $n\in U_s$.
\end{claim}

\begin{proof} Observe that $z_n\in \frac ee$ implies $z_ne=e$ and $e\lesseq z_n= z_ns\lesseq s$, see Lemma~\ref{l:quasiorder}. Then $e\lesseq z_n\Updownarrow \pi(z_n)\lesseq s$ and hence $e\le \pi(z_n)\le s$ by Lemma~\ref{l:pi2}. Assuming that $\pi(z_n)=e$ (or $\pi(z_n)=s$), we can take any $m\in U_s\setminus\{n\}$ and conclude that $\pi(z_n)=e\le\pi(z_m)$ (or $\pi(z_m)\le s=\pi(z_n)$), which contradicts Claim~\ref{cl:antichain}(1).
\end{proof}

Let $$\mathcal Z[e)\defeq\big\{\{z^{n}:z\in Z(X)\cap \tfrac ee\setminus\pi^{-1}[{\uparrow}F]\}:n\in\IN,\;F\in[E(X)\setminus{\downarrow}e]^{<\w}\big\}$$be the $e$-base from Example~\ref{ex:Ze}. Lemma~\ref{l:pi1} ensures that each set $V\in\mathcal Z[e)$ is a subsemigroup of $X$. Since the poset $E(X)$ is well-founded and the semigroup $X$ is nonsingular and eventually Clifford, the $e$-base $\mathcal Z[e)$ is regular and hence any topology $\Tau\supseteq \Tau_{\mathcal Z[e)}$ on $X$ is Hausdorff and zero-dimensional, according to Example~\ref{ex:Ze} and Theorem~\ref{t:TS}. 

\begin{claim}\label{cl:zn-converge} For every $V\in \mathcal Z[e)$ the set $\{n\in U_s:z_n\notin V\}$ is finite.
\end{claim}

\begin{proof} By the definition of the family $\mathcal Z[e)\ni V$, there exist $l\in\IN$ and $F\in[E(X)\setminus{\downarrow}e]^{<\w}$ such that $V=\{z^l:z\in Z(X)\cap\frac ee\setminus\pi^{-1}[{\uparrow}F]\}$. 

We claim that for every $f\in F$, there exists $\ell(f)\in\IN$ such that $Z(X)\cap{\Uparrow}e\cap {\Uparrow}f\subseteq\korin{\ell(f)}{H(X)}$. Since the poset $E(X)$ is well-founded, the nonempty subsemilattice $L\defeq EZ(X)\cap {\Uparrow}e\cap{\Uparrow}f\ni 1$ of $EZ(X)$ contains an idempotent $e_f\in L$ such that $L\subseteq {\uparrow}e_f\subseteq {\Uparrow}e_f$. It follows from $e_f\in L\subseteq {\Uparrow}e$ that $e\lesseq e_f$ and hence $e\le e_f$, see Lemma~\ref{l:pi2}. Assuming that $e=e_f$, we conclude that $f\in{\Downarrow}e_f={\Downarrow}e$, which contradicts the choice of $f\in F\subseteq E(X)\setminus{\downarrow}e=E(X)\setminus{\Downarrow}e$, see Lemma~\ref{l:pi2}. Therefore, $e<e_f$. Now the choice of the idempotent $e$ ensures that   $Z(X)\cap{\Uparrow}e_f\subseteq\korin{\ell(f)}{H(X)}$ for some $\ell(f)\in\IN$. We claim that $Z(X)\cap {\Uparrow}e\cap{\Uparrow}f\subseteq{\Uparrow}e_f$. Indeed, take any $x\in Z(X)\cap{\Uparrow}e\cap{\Uparrow}f$. By Lemma~\ref{l:piZ}, $\pi(x)\in EZ(X)\cap {\Uparrow}e\cap{\Uparrow}f=L\subseteq {\Uparrow}e_f$ and hence $x\in{\Updownarrow}\pi(x)\subseteq {\Uparrow}e_f$.

Choose $m\in\IN$ such that $m\ge \ell(f)$ for every $f\in F$. We claim that for every $k\ge m$ we have $\pi(z_k)\notin {\uparrow}F$. In the opposite case we could find  $k\ge m$ and $f\in F$ such that $\pi(z_k)\in{\uparrow}f$ and hence $z_k\in {\Updownarrow}\pi(z_k)\subseteq{\Uparrow}f$. Then $z_k\in Z(X)\cap\frac ee\cap{\Uparrow}f\subseteq Z(X)\cap{\Uparrow}e\cap{\Uparrow}f\subseteq Z(X)\cap {\Uparrow}e_f\subseteq\korin{\ell(f)}{H(X)}$ and hence $z_k^{\ell(f)}\in H(X)$. Since $\ell(f)\le m\le k$, we conclude that $z_k^{k!}=(z_k^{\ell(f)})^{k!/\ell(f)}\in H_{\pi(z_k)}^{k!/\ell(f)}\subseteq H_{\pi(z_k)}\subseteq H(X)$, which contradicts the choice of $z_k$ in Claim~\ref{cl:antichain}(2). This contradiction shows that $z_k\notin\pi^{-1}[{\uparrow}F]$ and hence for every $k\ge \max\{m,l\}$ we have $$z_k\in Z_{k!}\setminus\pi^{-1}[{\uparrow}F]=\{z^{k!}:z\in Z(X)\cap\tfrac ee\setminus\pi^{-1}[{\uparrow}F]\}\subseteq
\{z^l:z\in Z(X)\cap\tfrac ee\setminus\pi^{-1}[{\uparrow}F]\}=V.$$Then the set $\{n\in U_s:z_n\notin V\}\subseteq\{n\in U_s:n\le \max\{m,l\}\}$ is finite.
\end{proof}


For any subset $U\subseteq \w$ consider the set $$z_U\defeq\{z_n:n\in U\}\subseteq Z(X)\cap\tfrac ee.$$
Let $z_\U$ be the free ultrafilter on $X$ generated by the base $\{z_U:U\in\U\}$. The choice of $U_s$ ensures that $sz_n=z_n$ for all $n\in U_s\in\U$ and hence $sz_\U=z_\U$. The ultrafilter $z_\U$ is {\em central} in the sense that $xz_\U=z_\U x$ for all $x\in X$.

Observe that for every $x\in H_s$ we have
$$z_\U=sz_\U=x^{-1}xz_\U,$$
which implies that the ultrafilter $xz_\U$ is free.

Lemma~\ref{l:Tamura} implies the following useful fact.

\begin{claim}\label{cl:xs} For every $x\in{\Uparrow}s$ we have $xs\in H_s$ and $xz_\U=(xs)z_\U\in\{yz_\U:y\in H_s\}$.
\end{claim}

In the maximal group $H_s$ consider the subgroup
$$G\defeq\{x\in H_s:xz_\U=z_\U\}.$$
Since the ultrafilter $z_\U$ is central, for any $x\in H_s$ and any $g\in G$ we have $xgx^{-1}z_\U=xgz_\U x^{-1}=xz_\U x^{-1}=xx^{-1}z_\U=sz_\U=z_\U$ and hence $xgx^{-1}\in G$, which means that $G$ is a normal subgroup in $H_s$.

\begin{claim}\label{cl:well-defined} Let $x,y\in H_s$ be two elements such that $xz_\U=yz_\U$. Then
\begin{enumerate}
\item $xG=yG$;
\item $xe=ye$;
\item $xtz_\U=ytz_\U$ for every $t\in{\Uparrow}s$.
\end{enumerate}
\end{claim}

\begin{proof} 1. Let $y^{-1}$ be the inverse element to the element $y$ in the group $H_s$. Multiplying the equality $xz_\U=yz_\U$ by $y^{-1}$, we obtain $y^{-1}xz_\U=y^{-1}yz_\U=s z_\U=z_\U$ and hence $y^{-1}x\in G$ and finally, $xG=yG$.
\smallskip

2. Since $xz_\U=yz_\U$ and $U_s\in\U$, there exist numbers $n,m\in U_s$ such that $xz_n=yz_m$. The choice of the elements $z_n,z_m\in\frac ee$ ensures that $z_ne=e=z_me$. Then $xe=xz_ne=yz_me=ye$.
\smallskip

3. Choose any $t\in{\Uparrow}s$. By Claim~\ref{cl:xs}, $ts\in H_s$. Since $xG=yG$, there exists $g\in G$ such that $x=yg$. Since $G$ is a normal subgroup in $H_s$, there exists $g'\in G$ such that $g(ts)=(ts)g'=tg'$. Then
$$xtz_\U=xtsz_\U=ygtsz_\U=ytg' z_\U=ytz_\U.$$
\end{proof}

Consider the set
$$Y\defeq X\cup\{xz_\U:x\in H_s\}.$$
 Extend the semigroup operation of the semigroup $X$ to a binary operation $*:Y\times Y\to Y$ defined by
$$a*b=\begin{cases}
ab&\mbox{if $a,b\in X$};\\
ayz_\U&\mbox{if $a\in {\Uparrow}s$ and $b=yz_\U$ for some $y\in H_s$;}\\
aye&\mbox{if $a\in X\setminus{\Uparrow}s$ and $b=y z_\U$ for some $y\in H_s$;}\\
xb z_\U&\mbox{if $a=x z_\U$ for some $x\in H_s$ and $b\in {\Uparrow}s$;}\\
xeb&\mbox{if $a=xz_\U$ for some $x\in H_s$ and $b\in X\setminus{\Uparrow}s$};\\
xeye&\mbox{if $a=xz_\U$ and $b=yz_\U$ for some $x,y\in H_s$.}
\end{cases}
$$
Claim~\ref{cl:well-defined} ensures that the operation $*$ is well-defined.
Now we define a Hausdorff zero-dimensional topology $\Tau$ on $Y$ making the binary operation $*$ continuous. 

Consider the ideal $$I\defeq X\setminus\bigcup_{n\in U_s}{\Uparrow}z_n$$in $X$.

\begin{claim}\label{cl:nmI} For every distinct numbers $n,m\in\w$ we have $z_nz_m\in I$.
\end{claim}

\begin{proof} Assuming that $z_nz_m\notin I$, we conclude that $z_nz_m\in{\Uparrow}z_k$ for some $k\in U_s$. Then $z_n,z_m\in{\Uparrow}z_nz_m\subseteq{\Uparrow}z_k$ and hence $n=k=m$ by Claim~\ref{cl:antichain}(1), which contradicts the choice of the numbers $n,m$.
\end{proof}

For every $x\in X$ let
$$l_x\defeq
\begin{cases}
1&\mbox{if $x\notin{\Uparrow}s$};\\
2&\mbox{if $x\in{\Uparrow}s$};
\end{cases}
$$
and
$$U_x=\begin{cases}
U_s&\mbox{if $x\in{\Uparrow}s$};\\
\{n\in U_s:\pi(x)z_n\ne z_n\}&\mbox{if $x\notin{\Uparrow}s$}.
\end{cases}
$$

\begin{claim}\label{cl:UxU} For every $x\in X$ the set $U_x$ belongs to the ultrafilter $\U$.
\end{claim}

\begin{proof} If  $x\in{\Uparrow}s$, then the inclusion $U_x=U_s\in\U$ is evident. So, assume that $x\notin{\Uparrow}s$. Then $\pi(x)\notin{\Uparrow}s$ and by Claim~\ref{cl:s-min}, $\pi(x)\notin E_\U$ and $\{n\in\w:z_n\pi(x)=z_n\}\notin\U$. Since $\U$ is an ultrafilter, $\{n\in\w:z_n\pi(x)=z_n\}\notin\U$ implies $\{n\in \w:z_n\pi(x)\ne z_n\}\in\U$ and hence $U_x=\{n\in U_s:z_n\pi(x)\ne z_n\}\in\U$.
\end{proof}

For every $U\in\U$, consider the set
$$\nabla(U)\defeq I\cup \{xz_n^k:x\in X,\;n\in U\cap U_x,\;k\ge l_x\}.$$

\begin{claim}\label{cl:nabla} For every $x\in H_s$ we have
$xz_{U_s}\cap\nabla(\w)=\emptyset$.
\end{claim}

\begin{proof} Let $q:X\to X_{\Updownarrow}$ be the quotient homomorphism to the semilattice reflection of $X$. Observe that for every $n\in U_s$ we have
$$q(xz_n)=q(x)q(z_n)=q(s)q(z_n)=q(sz_n)=q(z_n)$$and hence $xz_n\in{\Updownarrow}z_n$ and $xz_n\notin I$. Assuming that $xz_{U_s}\cap\nabla(\w)\ne\emptyset$, we could find $y\in X$ and numbers $n\in U_s$ and $m\in U_y$ such that $xz_n=yz_m^k$ for some $k\ge l_y$. Applying the homomorphism $q$ to the equality $xz_n=yz_m^k$, we conclude that $$q(z_n)=q(sz_n)=q(s)q(z_n)=q(x)q(z_n)=q(xz_n)=q(yz_m^k)=q(y)q(z_m)\le q(z_m)$$and hence $n=m$ by Claim~\ref{cl:antichain}(1).
Then $xz_n=yz_m^k=yz_n^k$ and hence $$z_n=sz_n=x^{-1}xz_n=x^{-1}yz^k_n.$$

Assuming that $y\in {\Uparrow}s$, we conclude that $k\ge l_y=2$ and $z_n^k\in H_{\pi(z_n)}$ by Claim~\ref{cl:antichain}(2) and Lemma~\ref{l:C-ideal}. Then  $z_n=x^{-1}yz_n^k\in H_sH_{\pi(z_n)}\subseteq H_{\pi(z_n)}\subseteq H(X)$ by Claim~\ref{cl:e<s} and Lemma~\ref{l:Tamura}, which  contradicts the choice of $z_n$ in Claim~\ref{cl:antichain}(2).

This contradiction shows that $y\notin{\Uparrow}s$. Then $n=m\in U_s\cap U_y$ and $\pi(y)z_n\ne z_n$, by the definition of $U_y$.  On the other hand, $z_n=x^{-1}yz_n^k$ implies
$$z_n=x^{-1}yz_n^k=x^{-1}yz_nz_n^{k-1}=(x^{-1}y)^2z_n^kz_n^{k-1}=(x^{-1}y)^2z_nz_n^{2(k-1)}.$$Proceeding by induction, we can show that $z_n=(x^{-1}y)^jz_nz_n^{j(k-1)}$ for every $j\in\IN$. Since the semigroup $X$ is eventually Clifford, there exists $j\in\IN$ such that $(x^{-1}y)^j\in H_{\pi(x^{-1}y)}$. Lemma~\ref{l:Tamura} implies that $\pi(y)\pi(x^{-1}y)=\pi(x^{-1}y)$. Then
$$
z_n=(x^{-1}y)^jz_nz_n^{j(k-1)}=\pi(x^{-1}y)(x^{-1}y)^jz_nz_n^{j(k-1)}=\pi(x^{-1}y)z_n=\pi(y)\pi(x^{-1}y)z_n=\pi(y)z_n,
$$
which contradicts the inclusion $n\in U_y$.
\end{proof}

\begin{definition}\label{d:Tau}
Let $\Tau$ be the topology on $Y$ consisting of the sets $W\subseteq Y$ that satisfy the following conditions:
\begin{enumerate}
\item for any $x\in W\cap X$ with $x=xe\notin H_e$, there exists $V\in\mathcal Z[e)$ such that $\Lambdae(x;V)\setminus{\Uparrow}e\subseteq W$;
\item for any $x\in W\cap H_e$ there exists $V\in\mathcal Z[e)$ and $U\in\U$ such that $\Lambdae(x;V)\cap \nabla(U)\subseteq W$;
\item for any $x\in H_s$ with $xz_\U\in W$ there exists $U\in\U$ such that $xz_U\subseteq W$.
\end{enumerate}
\end{definition}


It is clear that $X$ is a nonclosed dense subset in $(Y,\Tau)$. It remains to prove that the topological space $(Y,\Tau)$ is Hausdorff and zero-dimensional, and the binary operation $*:Y\times Y\to Y$ is continuous with respect to the topology $\Tau$. This will be done in the following two lemmas.

\begin{lemma}\label{l:T-zerodim} The topological space $(Y,\Tau)$ is Hausdorff and zero-dimensional.
\end{lemma}

\begin{proof} We divide the proof of this lemma into a series of claims. 

\begin{claim}\label{cl:singleton-T} For every $a\in X$ with $a\ne ae$ the singleton $\{a\}$ is $\Tau$-clopen.
\end{claim}

\begin{proof} The singleton $\{a\}$ is  $\Tau$-open by the definition of the topology $\Tau$. To see that $\{a\}$ is $\Tau$-closed, consider the complement $W\defeq Y\setminus\{a\}$. Since the topology $\Tau_{\mathcal Z[e)}$ is Hausdorff, for every $x\in X\setminus\{a\}$ there exists a set $V\in \mathcal Z[e)$ such that $\Lambdae(x;V)\subseteq X\setminus\{a\}\subseteq W$. This witnesses that the set $W$ satisfies the conditions (1), (2) of Definition~\ref{d:Tau}. For every $x\in H_s$, the ultrafilter $xz_\U$ is free and hence there exists a set $U\in\U$ such that $a\notin xz_U$ and hence $xz_U\subseteq W$, so the condition (3) of Definition~\ref{d:Tau} is satisfied too and hence $\{a\}$ is $\Tau$-clopen.
\end{proof}

For every $x\in H_s$ and $U\in\U$, consider the subset
$$\Lambda(xz_\U;U)\defeq \{x z_\U\}\cup \{xz_n:n\in U\cap U_s\}$$
of $Y$.

\begin{claim}\label{cl:ultra-T} For every $a\in H_s$, the family $$\mathcal B_{az_\U}\defeq \{\Lambda(az_\U;U):U\in\U\}$$ consists of $\Tau$-clopen sets and is a neighborhood base of the topology $\Tau$ at $az _\U$.
\end{claim}

\begin{proof} First we show that for every $U\in\U$, the set $\Lambda(az_\U;U)$ is $\Tau$-open. Observe that for every $x\in X\cap\Lambda(az_\U;U)$, there exists $n\in U\cap U_s$ such that $x=az_n$. By Claim~\ref{cl:e<s} and Lemma~\ref{l:pi2}, $z_n\notin{\Updownarrow}e$ and hence $q(z_n)\ne q(e)$. It follows from $z_n\in\frac ee$ and $e<\pi(z_n)<s$ that
\begin{multline*}
q(xe)=q(az_ne)=q(ae)=q(a)q(e)=q(s)q(e)=q(se)=q(e)\\
\ne q(z_n)=q(sz_n)=q(s)q(z_n)=q(a)q(z_n)=q(az_n)=q(x)
\end{multline*}
 and hence $xe\ne x$. Then the conditions (1),(2) of Definition~\ref{d:Tau} are satisfied. To check the condition (3), take any $x\in H_s$ such that $xz_\U\in \Lambda(az_\U;U)$. Then $xz_\U=az_\U$ and hence there exists a set $U'\in\U$ such that $xz_{U'}\subseteq az_{U\cap U_s}\subseteq\Lambda(az_\U;U)$. Therefore, the set $\Lambda(sz_\U;U)$ satisfies the conditions (1)--(3) of Definition~\ref{d:Tau} and hence it is $\Tau$-open.

To see that $\Lambda(az_\U;U)$ is $\Tau$-closed, consider the set $W\defeq Y \setminus\Lambda(az_\U;U)$. We should check that the set $W$ satisfies the conditions (1)--(3) of Definition~\ref{d:Tau}.
\smallskip

1) Given any point $x\in X\cap W$ with $xe=x\notin H_e$, choose any set $V\in\mathcal Z[e)$ and observe that $\Lambdae(x;V)\setminus{\Uparrow}e\subseteq X\setminus{\Uparrow}e\subseteq X\setminus\Lambda(az_\U;U)\subseteq W$.
\smallskip

2) Given any point $x\in H_e\cap W$, take any $V\in\mathcal Z[e)$. By Claim~\ref{cl:nabla}, $\nabla(\w)\cap \Lambda(az_\U;U)=\emptyset$ and hence $\Lambdae(x;V)\cap \nabla(\w)\subseteq W$.
\smallskip

3) Finally take any point $x\in H_s$ with $xz_\U\in W$. Then $xz_\U\ne az_\U$. We claim that the set $U'=\{n\in U_s:xz_n\notin \Lambda(az_\U;U)\}$ belongs to the ultrafilter $\U$. In the opposite case, the set $U''=\{n\in U_s:xz_n\in \Lambda(az_\U;U)\}$ belongs to the ultrafilter $\U$. For every $n\in U''$, we can find $m\in U\cap U_s$ such that $xz_n=az_m$. Applying to this equality the homomorphism $q:X\to X/_{\Updownarrow}$, we conclude that $$q(z_n)=q(sz_n)=q(s)q(z_n)=q(x)q(z_n)=
q(xz_n)=q(az_m)=q(a)q(z_m)=q(s)q(z_m)=q(sz_m)=q(z_m)
$$ and hence $n=m$ by Claim~\ref{cl:antichain}(1). Then $xz_n=az_n$ for all $n\in U''\in\U$, which implies $xz_\U=az_\U$. But this contradicts the choice of $x$. This contradiction shows that the set $U'$ belongs to the ultrafilter $\U$ and hence $xz_{U'}\subseteq W$, witnessing that the condition (3) of Definition~\ref{d:Tau} is satisfied.
\smallskip

Therefore, the set $\Lambda(az_\U;U)$ is $\Tau$-clopen. Finally, we show that the family $\mathcal B_{az_\U}$ is a neighborhood base of the topology $\Tau$ at $a z_\U$. Given any neighborhood $W\in\Tau$ of $az_\U$, by Definition~\ref{d:Tau}, there exists a set $U\in\U$ such that $az_U\subseteq W$ and hence $\Lambda(az_\U;U)\subseteq\{az_\U\}\cup az_U\subseteq W$.
\end{proof}

The definition of the topology $\Tau$, Lemmas~\ref{l:Tamura}, \ref{l:base-e} and  Theorem~\ref{t:TS} imply the following

\begin{claim}\label{cl:notHe-T} For every $x\in X$ with $x=xe\notin H_e$ the family
$$\mathcal B_x=\{\Lambda^e(x;V)\setminus{\Uparrow}e:V\in\mathcal Z[e)\}$$ consists of $\Tau$-clopen sets and is a neighborhood base of the topology $\Tau$ at $x$.
\end{claim}

\begin{claim}\label{cl:He-T} For every $x\in H_e$ the family
$$\mathcal B_x=\{\Lambda^e(x;V)\cap\nabla(U):V\in\mathcal Z[e),\;U\in\U\}$$ consists of $\Tau$-clopen sets and is a neighborhood base of the topology $\Tau$ at $x$.
\end{claim}

\begin{proof} Fix any $V\in\mathcal Z[e)$ and $U\in\U$. By Definition~\ref{d:remote} and Lemma~\ref{l:Lamb}(2), $\Lambdae(x;V)\subseteq \Lambdae(x;\frac ee)\subseteq \frac{xe}{e}\subseteq{\Uparrow}e$. Observe that for every $y\in H_e\cap \frac{xe}e$, we have $y=ye=xe=x$. This observation and the definition of the topology $\Tau$ imply that the set $ \Lambdae(x;V)\cap\nabla(U)$ is $\Tau$-open. To show that this set is $\Tau$-closed, consider its complement $W\defeq Y\setminus (\Lambdae(x;V)\cap\nabla(U))$. By Theorem~\ref{t:TS}(3) and Example~\ref{ex:Ze}, the set $W\cap X$ is $\Tau_{\mathcal Z[e)}$-open, which implies that $W$ satisfies the conditions (1), (2) of Definition~\ref{d:Tau}.

To check the condition (3), fix any $y\in H_s$ with $xz_\U\in W$. By Claim~\ref{cl:nabla}, $\nabla(U)\cap\Lambda(yz_\U;U_s)\subseteq \nabla(\w)\cap \Lambda(yz_\U;U_s)=\emptyset$ and hence
$\Lambda(yz_\U;U_s)\subseteq Y\setminus \nabla(U)\subseteq W$.
\smallskip

Therefore, the set $\Lambdae(x;V)\cap\nabla(U)$ is $\Tau$-clopen. The definition of the topology $\Tau$ ensures that $\mathcal B_x$ is a neighborhood base of the topology $\Tau$ at $x$.
\end{proof}

Claims~\ref{cl:singleton-T}--\ref{cl:He-T} imply that the topology $\Tau$ satisfies the separation axiom $T_0$ and has a base consisting of $\Tau$-clopen sets. This implies that the topological space $(Y,\Tau)$ is Hausdorff and zero-dimensional.
\end{proof}

\begin{lemma}\label{l:conti} The binary operation $*:Y\times Y\to Y$ is continuous.
\end{lemma}

\begin{proof} Fix any elements $a,b\in Y$ and a neighborhood $O_{ab}\in\Tau$ of the element $a*b$. We should find neighborhoods $O_a,O_b\in\Tau$ of $a,b$ such that $O_a*O_b\subseteq O_{ab}$. According to the definition of the operation $*$, we will consider six cases.
\smallskip

1. First assume that $a,b\in X$. Then $a*b=ab$. This case has three subcases.
\smallskip

1.1. If $ab\ne abe$, then $ae\ne a$ and $be\ne b$ and hence $a,b$ are isolated points of $(X,\Tau)$ according to Claim~\ref{cl:singleton-T}. Then  the neighborhoods $O_a=\{a\}$ and $O_n=\{b\}$ have the desired property: $O_a*O_b=\{a*b\}\subseteq O_{ab}$.
\smallskip

1.2. Next, assume that $ab=abe$ and $ab\notin H_e$. By Claim~\ref{cl:notHe-T}, there exists a set $V\in\mathcal Z[e)$ such that $\Lambdae(ab;V)\setminus{\Uparrow}e\subseteq O_{ab}$. If $a,b\in{\Uparrow}e$, then $ab\in{\Uparrow}e$ and $ab=abe\in H_e$ by Lemma~\ref{l:Tamura}, which contradicts our assumption. So, $a\notin{\Uparrow}e$ or $b\notin{\Uparrow}e$. 

1.2.1. If $a\notin{\Uparrow}e$, then by Claim~\ref{cl:notHe-T} and Lemma~\ref{l:Lamb}(8), for the neighborhoods $O_a\defeq\Lambdae(a;V)\setminus{\Uparrow}e\in\Tau$ and $O_b\defeq\Lambdae(b;V)\in\Tau$ we have $$O_a*O_b\subseteq\Lambdae(ab;V)\setminus{\Uparrow}e
\subseteq O_{ab}.$$

1.2.2. If $b\notin{\Uparrow}e$, then by Claim~\ref{cl:notHe-T} and Lemma~\ref{l:Lamb}(8), for the neighborhoods $O_a\defeq\Lambdae(a;V)\in\Tau$ and $O_b\defeq\Lambdae(b;V)\setminus{\Uparrow}e\in\Tau$ we have $$O_a*O_b\subseteq\Lambdae(ab;V)\setminus{\Uparrow}e
\subseteq O_{ab}.$$
\smallskip

1.3. Finally, assume that $ab\in H_e$. In this case, by Claim~\ref{cl:He-T}, there exists $V\in\mathcal Z[e)$ and $U\in\U$ such that $U\subseteq U_a\cap U_b$ and  $\Lambdae(ab;V)\cap\nabla(U)\subseteq O_{ab}$.

If $a\ne ae$, then $O_a\defeq\{a\}$ is a $\Tau$-open neighborhood of $a$, by Claim~\ref{cl:singleton-T}. If $a=ae$, then $ab\in H_e$ implies $a\in{\Uparrow}e$ and hence $a=ae\in H_e$ by Lemma~\ref{l:Tamura}. In this case, Claim~\ref{cl:He-T} ensures that $$O_a\defeq \Lambdae(a;V)\cap\nabla(U)$$is a $\Tau$-open neighborhood of $a$.

By analogy, define the $\Tau$-open neighborhood
$$
O_b\defeq\begin{cases}\{b\},&\mbox{if $b\ne be$};\\
\Lambdae(b;V)\cap\nabla(U),&\mbox{otherwise};\\
\end{cases}
$$
of $b$.

We claim that $O_a*O_b\subseteq O_{ab}$. Given any elements $x\in O_a$ and $y\in O_b$, we should prove that $xy\in O_{ab}$. By Lemma~\ref{l:Lamb}(8) and Example~\ref{ex:Ze}, $$xy\in O_aO_b\subseteq \Lambdae(a;V)\cdot\Lambdae(b;V)\subseteq\Lambdae(ab;VV)\subseteq  \Lambdae(ab;V).$$
If $xy\in I$, then
$$xy\in \Lambdae(ab;V)\cap I\subseteq \Lambdae(ab;V)\cap\nabla(U)\subseteq O_{ab}$$ and we are done. So, we assume that $xy\notin I$. In this case $x,y\notin I$.

Next, we consider four subcases.
\smallskip

1.3.1. If $a\ne ae$ and $b\ne be$, then $O_a=\{a\}$ and $O_b=\{b\}$ and $xy=ab\in O_{ab}$.
\smallskip

1.3.2. Assume that $a\ne ae$ and $b=be$. Then $O_a=\{a\}$, $x=a\ne ae$ and hence $x=a\notin H_e$. Also $y\in O_b=\Lambdae(b;V)\cap\nabla(U)$. It follows from $y\in \nabla(U)\setminus I$ that $y=tz_n^k$ for some $t\in X$, $n\in  U\cap U_t$, and $k\ge l_t$.
\smallskip

1.3.2.1. If $t\in {\Uparrow}s$, then $U_t=U_s$ and $l_t=2$.
\smallskip

1.3.2.1.1. If $a\in{\Uparrow}s$, then $xt=at\in {\Uparrow}s$, $n\in  U\cap U_t=U\cap U_s=U\cap U_{xt}$, $l_{xt}=2=l_t\le k$ and hence
$$xy=xtz_n^k\in \Lambdae(ab;V)\cap \nabla(U)\subseteq O_{ab}.$$
\smallskip

1.3.2.1.2. If $a\notin{\Uparrow}s$, then $xt=at\notin{\Uparrow}s$ and $l_{at}=l_{xt}=1$.  
We claim that $\pi(at)z_n\ne z_n$. Indeed, assuming that $\pi(at)z_n=z_n$ and applying Lemma~\ref{l:pi2}, we conclude that $\pi(at)=\pi(a)\pi(at)$ and hence $z_n=\pi(at)z_n=\pi(a)\pi(at)z_n=\pi(a)z_n$, which contradicts the inclusion $n\in U\subseteq U_a$. This contradiction shows that $\pi(at)z_n\ne z_n$ and hence $n\in U\cap U_{at}$ and $xy=atz_n^k\in \Lambdae(ab;V)\cap\nabla(U)\subseteq O_{ab}$.
\smallskip

1.3.2.2. If $t\notin{\Uparrow}s$, then $l_t=1$ and $n\in U_t$ implies $\pi(t)z_n\ne z_n$. Repeating the argument from the case 1.3.2.1.2, we can show that $\pi(at)z_n\ne z_n$. Then $at\notin{\Uparrow}s$, $l_{at}=1$, and $n\in U\cap U_{at}$. Now the definition of $\nabla(U)$ ensures that $xy=atz_n^k\in \Lambdae(ab;V)\cap\nabla(U)\subseteq O_{ab}$.
\smallskip

1.3.3. The case $a=ae$ and $b\ne be$ can be considered by analogy with the case 1.3.2.
\smallskip

1.3.4. Finally, assume that $a=ae$ and $b=be$. In this case the assumpton $ab\in H_e$ of the case 1.3 implies $a,b\in{\Uparrow}e$ and hence $a=ae\in H_e$ and $b=be\in H_e$, by Lemma~\ref{l:Tamura}. The inclusions $x,y\in\nabla(U)\setminus I$ imply that $x=tz_n^k$ for some $t\in X$, $n\in U\cap U_t$ and $k\ge l_t$, and $y=t'z_{n'}^{k'}$ for some $t'\in X$, $n'\in U\cap U_{t'}$ and $k'\ge l_{t'}$. If $n\ne n'$, then by Claim~\ref{cl:nmI},  $xy=tz_n^kt'z_{n'}^{k'}=tt'z_n^kz_{n'}^{k'}\in I$, which contradicts our assumption. So, $n=n'$ and $k+k'\ge 2$.
\smallskip

1.3.4.1. If $t,t'\in{\Uparrow}s$, then $tt'\in {\Uparrow}s$, $l_{tt'}=2=l_{t}=l_{t'}$ and $n=n'\in U\cap U_t\cap U_{t'}=U\cap U_s=U\cap U_{tt'}$. Then
$xy=tz_n^kt'z_n^{k'}=tt'z_n^{k+k'}\in \Lambdae(ab;V)\cap\nabla(U)\subseteq O_{ab}.$
\smallskip

1.3.4.2. If $t\notin{\Uparrow}s$, then $tt'\notin{\Uparrow}s$ and $n\in U_t$ imply $\pi(t)z_n\ne z_n$ and $\pi(tt')z_n\ne z_n$ (as $\pi(tt')=\pi(t)\pi(tt')$ by Lemma~\ref{l:pi2}). Then $n\in U\cap U_{tt'}$ and $xy=tz_n^kt'z_n^{k'}=tt'z_n^{k+k'}\in\Lambdae(ab;V)\cap\nabla(U)\subseteq O_{ab}.$
\smallskip

1.3.4.3. The case $t'\notin{\Uparrow}s$ can be treated by analogy with the case 1.3.4.2.

This completes the analysis of the case 1.
\smallskip

2. Assume that $a\in{\Uparrow}s$ and $b=yz_\U$ for some $y\in H_s$. In this case $a*b=ayz_\U$ and by Claim~\ref{cl:ultra-T}, there exist a set $U\in\U$ such that $\Lambda(ayz_\U;U)\subseteq O_{ab}$. It follows from $a\in{\Uparrow}s$ and $e<s$ that $ae\ne a$ and hence $O_a\defeq\{a\}$ is a $\Tau$-open neighborhood of $a$.  Consider the neighborhood $O_b\defeq \Lambda(yz_\U;U)$ of $b$ and observe that $O_a*O_b=\Lambda(ayz_\U,U)\subseteq O_{ab}$.
\smallskip

3. Assume that $a\in X\setminus{\Uparrow}s$ and $b=yz_\U$ for some $y\in H_s$. In this case $a*b=aye\in X$. 
By Claims~\ref{cl:notHe-T} and \ref{cl:He-T}, there exist $V\in\mathcal Z[e)$ and $U\in\U$ satisfying the following conditions:
\begin{itemize}
\item $U\subseteq U_a$;
\item if $aye\notin H_e$, then $\Lambdae(aye;V)\setminus{\Uparrow}e\subseteq O_{ab}$;
\item if $aye\in H_e$, then $\Lambdae(aye;V)\cap\nabla(U)\subseteq O_{ab}$.
\end{itemize}
By Claim~\ref{cl:zn-converge}, the set $U'\defeq\{n\in U_s:z_n\in V\}$ belongs to the ultrafilter $\U$.
Replacing  the set $U$ by $U\cap U'$, we can assume that $U\subseteq U'$. Then for every $n\in U\subseteq U',$ the element $yz_n\in\Lambda(yz_\U;U)$ belongs to the set $\frac{ye}e\cdot V\subseteq \Lambdae(ye;V)$.

Let $O_b\defeq\Lambda(yz_\U;U)$ and $$O_a\defeq \begin{cases}
\{a\}&\mbox{if $a\ne ae$};\\
\Lambdae(a;V)\setminus{\Uparrow}e&\mbox{if $a=ae\notin H_e$};\\
\Lambdae(a;V)\cap \nabla(U)&\mbox{if $a\in H_e$}.
\end{cases}
$$
By Claims~\ref{cl:singleton-T}-- 
\ref{cl:He-T}, the sets $O_a,O_b$ are $\Tau$-open. 
To show that $O_a*O_b\subseteq O_{ab}$, take any elements $a'\in O_a$ and $b'\in O_b$, and consider two cases.
\smallskip

3.1. Assume that $b'=yz_\U$ and consider two subcases.

3.1.1. If $aye\notin H_e$, then $aye\notin{\Uparrow}e$, by Lemma~\ref{l:Tamura}. It follows from $ye\in H_se\subseteq H_e$ and $aye\notin{\Uparrow}e$ that $a\notin{\Uparrow}e$. Then $a'*b'\in O_aye\subseteq(\Lambda(a;V)\setminus{\Uparrow}e)\cdot ye\subseteq \Lambdae(aye;V)\setminus {\Uparrow}e\subseteq O_{ab}$, see Lemma~\ref{l:Lamb}(7). 
\smallskip

3.1.2. If $aye\in H_e$, then $a'*b'=a'ye\in (\Lambdae(a;V)ye)\cap{\Downarrow}e\subseteq\Lambdae(aye;V)\cap I\subseteq\Lambdae(aye;V)\cap\nabla(U)\subseteq O_{ab}$.

3.2. Next, assume that $b'\ne yz_\U$ and hence $b'=yz_n$ for some $n\in U\subseteq U_a\cap U'$.

3.2.1. If $aye\notin H_e$, then $aye\notin{\Uparrow}e$ and  $a\notin{\Uparrow}e$. Then $a'*b'\in O_ayz_n\subseteq(\Lambdae(a;V) \setminus{\Uparrow}e)\cdot\Lambdae(ye;V)\subseteq \Lambdae(aye;V)\setminus {\Uparrow}e\subseteq O_{ab}$. 
\smallskip

3.2.2. If $aye\in H_e$ and $a'*b'\in I$, then $a'*b'=a'yz_n\in \big(\Lambdae(a;V)\cdot\Lambdae(ye;V)\big)\cap I\subseteq\Lambdae(aye;V)\cap I\subseteq\Lambdae(aye;V)\cap\nabla(U)\subseteq O_{ab}$.
\smallskip

3.2.3. Finally, assume that $aye\in H_e$ and $a'*b'\notin I$. Then $a'\notin I$. It follows from $aye\in H_e$ that $a\in{\Uparrow}e$. 

3.2.3.1. If $a\ne ae$, then $a'\in O_a=\{a\}$ and hence $a'*b'=ayz_n$. It follows from $a\notin{\Uparrow}s$ that $ay\notin{\Uparrow}s$ and $l_{ay}=1$. Also $a\notin{\Uparrow}s$ and $n\in U\subseteq U_a$ imply $\pi(a)z_n\ne z_n$. Using Lemma~\ref{l:pi2}, we can show that $\pi(ay)\le\pi(a)$ and hence $\pi(a)\pi(ay)=\pi(ay)$. Assuming that $\pi(ay)z_n=z_n$, we conclude that $z_n=\pi(ay)z_n=\pi(a)\pi(ay)z_n=\pi(a)z_n$, which contradicts $n\in U_a$. This contradiction shows that $\pi(ay)z_n\ne z_n$ and hence $n\in U_{ay}$. Then $a'*b'=ayz_n\in(a\Lambdae(ye;V))\cap\nabla(U)\subseteq \Lambdae(aye;V)\cap\nabla(U)\subseteq O_{ab}$ as $n\in U\cap U_{ay}$ and $1=l_{ay}$.
\smallskip

3.2.3.2. If $a=ae$, then $a=ae\in({\Uparrow}e)e\subseteq H_e$ by Lemma~\ref{l:Tamura} and hence $a'\in O_a=\Lambdae(a;V)\cap\nabla(U)$. It follows from $a'\in \nabla(U)\setminus I$ that $a'=x'z_{n'}^{k'}$ for some $x'\in X$, $n'\in U\cap U_{x'}$ and $k'\ge l_{x'}$. If $n\ne n'$, then $a'*b'=x'z_{n'}^{k'}yz_n=x'yz_{n'}^{k'}z_n\in I$ by Claim~\ref{cl:nmI}, which contradicts or assumption. This contradiction shows that $n=n'$.
\smallskip

We claim that $n\in U_{x'y}$. Indeed, if $x'\in{\Uparrow}s$, then $x'y\in{\Uparrow}s$ and hence $n=n'\in U_{x'}=U_s=U_{yx'}$.
If $x'\notin{\Uparrow}s$, then $x'y\notin{\Uparrow}s$ and hence $n=n'\in U_{x'}$ implies $\pi(x')z_n\ne z_n$. Lemma~\ref{l:pi2} ensures that $\pi(x'y)=\pi(x')\pi(x'y)$. Assuming that $\pi(x'y)z_n=z_n$, we conclude that $z_n=\pi(x'y)z_n=\pi(x')\pi(x'y)z_n=\pi(x')z_n$, which contradicts the inclusion $n\in U_{x'}$. This contradiction shows that $\pi(x'y)z_n\ne z_n$ and hence $n\in U_{x'y}$. 

Then $a'*b'=x'z_n^{k'}yz_n=x'yz_n^{k'+1}\in\nabla(U)$ as $n\in U\cap U_{yx'}$ and $k'+1\ge 2\ge l_{x'y}$. Finally, $a'*b'=a'yz_n\in (\Lambdae(a;V)\cdot\Lambdae(ye;V))\cap\nabla(U)\subseteq \Lambdae(aye;V)\cap\nabla(U)\subseteq O_{ab}.$
\smallskip

4. Assume that $a=xz_\U$ for some $x\in H_s$ and $b\in {\Uparrow}s$. This case can be considered by analogy with case 2.
\smallskip

5. Assume that $a=x z_\U$ for some $x\in H_s$ and $b\in X\setminus{\Uparrow}s$.
 This case can be considered by analogy with case 3.
\smallskip

6. Finally, assume that $a=xz_\U$ and $b=yz_\U$ for some $x,y\in H_s$. In this case $a*b=xeye=xye\in H_e$. By Claim~\ref{cl:He-T}, there exist $V\in\mathcal Z[e)$ and $U\in\U$ such that $\Lambdae(xye;V)\cap\nabla(U)\subseteq O_{ab}$. By Claim~\ref{cl:zn-converge}, there exists a set $U'\in\U$ such that $U'\subseteq U\cap U_s$ and for every $n\in U'$ we have $z_n\in V$.  Then for every $n\in U'$ and $g\in H_s$ we have
$$gz_n\in\tfrac{ge}e\cdot V\subseteq\Lambdae(ge;V).$$

By Claim~\ref{cl:ultra-T}, the sets $O_a\defeq\Lambda(x z_\U;U')$ and $O_b\defeq\Lambda(y z_\U;U')$ are $\Tau$-open. We claim that $O_a*O_b\subseteq O_{ab}$. Given any elements $a'\in O_a$ and $b'\in O_b$, we should check that $a'*b'\in O_{ab}$. Consider four cases.

6.1. If $a'=x z_\U$ and $b'=y z_\U$, then $a'*b'=a*b\in O_{ab}$.

6.2. If $a'=xz_\U$ and $b'\ne yz_\U$, then $b'=yz_n$ for some $n\in U'\subseteq U_s$. By the definition of the binary operation $*$ and Lemma~\ref{l:Lamb}(8),
$$a'*b'=xeyz_n\in (xe\Lambda(ye;V))\cap I\subseteq \Lambda(xeye;V)\cap\nabla(U)\subseteq O_{ab}.$$

6.3. The case of $a'\ne xz_\U$ and $b'=yz_\U$ can be considered by analogy with the case 6.2.

6.4. If $a'\ne xz_\U$ and $b'\ne yz_\U$, then $a'=xz_n$ and $b'=yz_m$ for some $n,m\in U'\subseteq U_s$. If $n\ne m$, then $a'*b'=xz_nyz_m\in I$ by Claim~\ref{cl:nmI} and hence
$$a'*b'=xz_nyz_m\in \big(\Lambda(xe;V)\cdot\Lambda(ye;V)\big)\cap I\subseteq\Lambdae(xeye;V)\cap\nabla(U)\subseteq O_{ab}.$$
If $n=m$, then $a'*b'=xz_nyz_m=xyz_n^2$. Since $n\in U'\subseteq U_s=U_{xy}$, we have  
$$a'*b'=xz_nyz_m=xyz_n^2\in \big(\Lambdae(xe;V)\cdot\Lambdae(ye;V)\big)\cap\nabla(U)\subseteq \Lambdae(xeye;V)\cap\nabla(U)\subseteq O_{ab}.$$
\end{proof}

\begin{lemma}\label{l:non-IC} $(Y,\Tau)$ is a  Hausdorff zero-dimensional topological semigroup containing $X$ as a nonclosed dense subsemigroup.
\end{lemma}

\begin{proof} By Lemma~\ref{l:T-zerodim}, the space $(Y,\Tau)$ is Hausdorff and zero-dimensional and by Lemma~\ref{l:conti}, the binary operation $*:Y\times Y\to Y$ is continuous. The density of $X$ in the Hausdorff topological space $(Y,\Tau)$ and the associativity of the semigroup operation $*{\restriction}_{X\times X}$ of $X$ imply that the binary operation $*$ is associative and hence $(Y,\Tau)$ is a topological semigroup containing $X$ as a dense subsemigroup.
\end{proof}

Lemma~\ref{l:non-IC} implies that the semigroup $X$ is not injectively $\mathsf{T_{\!z}S}$-closed, which contradicts our assumption. This contradiction shows that $Z(X)\subseteq\korin{\ell}{H(X)}$ for some $\ell\in\IN$.
\end{proof}

\section{Another boundedness property of injectively $\mathsf{T_{\!z}S}$-closed semigroups}\label{s:b2}

In this section we prove that the Cliford part of an injectively $\mathsf{T_{\!z}S}$-closed commutative semigroup $X$ is bounded. In fact, we shall derive this result from a more general result treating injectively $\mathsf{T_{\!z}S}$-closed viable semigroups with some extra properties.

A semigroup $X$ is defined to be {\em projectively $Z$-bounded} if there exists $n\in\IN$ such that for every surjective homomorphism $h:X\to Y$ to a semigroup $Y$ we have $Z(Y)\subseteq \korin{n}{E(Y)}$. It is easy to see that each bounded semigroup is projectively $Z$-bounded.

\begin{proposition}\label{p2:vsim-grobikam-grobik} Let $X$ be an injectively $\mathsf{T_{\!z}S}$-closed semigroup. If the poset $E(X)$ is well-founded, the semigroup $X$ is viable, eventually Clifford, and all maximal subgroups of $X$ are projectively $Z$-bounded, then the semigroup $Z(X)\cap H(X)$ is bounded.
\end{proposition}

\begin{proof} By Lemma~\ref{l:+1}, we lose no generality assuming that the semigroup $X$ is unital and hence contains an element $1$ such that $x1=x=1x$ for all $x\in X$. By Lemma~\ref{l:CH}, the set $Z(X)\cap H(X)$ is a subsemigroup of $X$.  

Let $q:X\to X/_{\Updownarrow}$ be the quotient homomorphism of $X$ to its semilattice reflection. Since the semigroup $X$ is viable and eventually Clifford, the restriction $q{\restriction}_{E(X)}:E(X)\to X/_{\Updownarrow}$ is bijective. 

Let $\pi:X\to E(X)$ be the function assigning to each $x\in X$ a unique idempotent $\pi(x)$ such that $x^\IN\cap H_{\pi(x)}$. It follows that 
$$\textstyle{\Updownarrow}x={\Updownarrow}\pi(x)=\korin{\IN}{H_{\pi(x)}}\quad\mbox{for every $x\in X$}.$$

By Theorem~\ref{t:Z}, the semigroup $Z(X)$ is chain-finite, periodic, nonsingular, and group-finite. 

To derive a contradiction, assume that the semigroup $Z(X)\cap H(X)$ is unbounded. 
Since the semilattice $EZ(X)=E(Z(X))$ is chain-finite and the semigroup $Z(X)\cap H(X)$ is not bounded, there exists an idempotent $e\in EZ(X)$ such that the semigroup $Z(X)\cap H(X)\cap {\Uparrow}e$ is not bounded, but for every idempotent $u\in EZ(X)$ with $e<u$, the semigroup $Z(X)\cap H(X)\cap{\Uparrow}u$ is bounded. 

For every $n\in\IN$, let $$Z_{n!}\defeq\{z^{n!}:z\in Z(X)\cap H(X)\cap\tfrac ee\}.$$

\begin{claim}\label{cl2:antichain} There exists a sequence $(z_k)_{k\in\w}\in\prod_{k\in\w}Z_{k!}$ such that
\begin{enumerate}
\item $z_k\notin{\Uparrow}z_n$ and $\pi(z_k)\not\le \pi(z_n)$ for any distinct numbers $k,n\in\w$;
\item $z_k^i\ne z_k^j$ for any $k\in\w$ and distinct numbers $i,j\in\{1,\dots,k\}$.
\end{enumerate}
\end{claim}

\begin{proof} Since the semigroup $Z(X)$ is periodic, the semigroup $Z(X)\cap H_e$ is a periodic subgroup of $Z(X)$. Since $Z(X)$ is group-finite, the subgroup $Z(X)\cap H_e$ is finite and hence for the number $p=|Z(X)\cap H_e|$ we have $z^p=e$ for all $z\in Z(X)\cap H_e$. Since the semigroup $Z(X)\cap H(X)\cap{\Uparrow}e$ is unbounded, for every $k\in\w$ there exists an element $g_k\in Z(X)\cap H(X)\cap {\Uparrow}e$ such that for any distinct positive numbers $i,j\le p(k+1)!$ we have $g_k^i\ne g_k^j$. Let $[\w]^2$ be the family of two-element subsets of $\w$. Consider the function $\chi:[\w]^2\to\{0,1,2\}$ defined by
$$\chi(\{n,m\})=\begin{cases}0&\mbox{if $\pi(g_n)=\pi(g_m)$};\\
1&\mbox{if $\pi(g_n)<\pi(g_m)$ or $\pi(g_m)<\pi(g_n)$};\\
2&\mbox{otherwise}
\end{cases}
$$
By the Ramsey Theorem 5 \cite{Ramsey}, there exists an infinite set $\Omega\subseteq \w$ such that $\chi[[\Omega]^2]=\{c\}$ for some $c\in\{0,1,2\}$. If $c=0$, then the set $\{\pi(g_n)\}_{n\in\Omega}$ contains a unique idempotent $u$ and hence the set $\{g_k\}_{k\in\Omega}\subseteq Z(X)\cap H_u$ is finite (since $Z(X)$ is periodic and group-finite). By the Pigeonhole Principle, for any $k>|Z(X)\cap H_u|$ there are two numbers $i<j\le k$ such that $g_k^i=g_k^j$, which contradicts the choice of $g_k$. Therefore, $c\ne 0$. If $c=1$, then the set $\{\pi(g_k)\}_{k\in\Omega}$ is an infinite chain in $EZ(X)$ which is not possible as $Z(X)$ is chain-finite. Therefore, $c=2$ and hence $\{\pi(g_k)\}_{k\in\Omega}$ is an infinite antichain in $EZ(X)$. Write the infinite set $\Omega$ as $\{n_k\}_{k\in\w}$ for some strictly increasing sequence $(n_k)_{k\in\w}$. For every $k\in\w$ put $z_k=g_{n_k}^{k!p}\in Z_{k!}$. Then $z_ke=g_{n_k}^{k!p}e=(g^{k!}_{n_k}e)^{p}=e$, because $g^{k!}_{n_k}e\in Z(X)\cap H_e$ by Lemma~\ref{l:Tamura}. Hence condition (1) of Claim~\ref{cl:antichain} is satisfied. It is easy to see that the sequence $(z_k)_{k\in\w}$ also satisfies condition (2).
\end{proof}

Fix any free ultrafilter $\U$ on the set $\w$. Consider the set $$E_\U\defeq\big\{u\in E(X):\{n\in\w:z_nu=z_n\}\in\U\big\}$$ and observe that $E_\U$ contains the element  $1\in E(X)$ and hence $E_\U$ is not empty. Since the poset $E(X)$ is well-founded, there exists an idempotent $s\in E_\U$ such that $E_\U\cap{\downarrow}s=\{s\}$. By analogy with Claim~\ref{cl:s-min}, we can prove that $E_\U\subseteq{\uparrow}s$.

By the definition of the set $E_\U$, the set $$U_s\defeq\{n\in\w:z_ns=z_n\}$$ belongs to the free ultrafilter $\U$ and hence is infinite. By analogy with Claim~\ref{cl:e<s} we can prove 

\begin{claim}\label{cl2:e<s} $e<\pi(z_n)<s$ for every $n\in U_s$.
\end{claim}

Let $$\mathcal H[e)\defeq\big\{Z(X)\cap H(X)\cap \tfrac ee\setminus\pi^{-1}[{\uparrow}F]:F\in[ZE(X)\setminus{\downarrow}e]^{<\w}\big\}$$be the $e$-base from Example~\ref{ex:He}. Definition~\ref{d:remote} ensures that each set $V\in\mathcal Z[e)$ is a subsemigroup of $X$. Since the poset $EZ(X)$ is well-founded, the $e$-base $\mathcal H[e)$ is regular and hence any topology $\Tau\supseteq \Tau_{\mathcal H[e)}$ on $X$ is Hausdorff and zero-dimensional, according to Example~\ref{ex:He} and Theorem~\ref{t:TS}. 

\begin{claim}\label{cl2:zn-converge} For every $V\in \mathcal H[e)$ the set $\{n\in U_s:z_n\notin V\}$ is finite.
\end{claim}

\begin{proof} By the definition of the family $\mathcal H[e)\ni V$, there exists a finite set $F\subseteq EZ(X)\setminus{\downarrow}e$ such that $V=Z(X)\cap H(X)\cap\frac ee\setminus\pi^{-1}[{\uparrow}F]=Z(X)\cap H(X)\cap\frac ee\setminus{\Uparrow}F$, see Lemma~\ref{l:pi1}. 

We claim that for every $f\in F$, the semigroup $Z(X)\cap H(X)\cap \frac ee\cap {\Uparrow}f$ is bounded. Since the semilattice $EZ(X)$ is chain-finite, the nonempty subsemilattice $L\defeq EZ(X)\cap \frac ee\cap{\Uparrow}f\ni 1$ of $EZ(X)$ contains an idempotent $e_f\in L$ such that $L\subseteq {\uparrow}e_f\subseteq {\Uparrow}e_f$. It follows from $e_f\in L\subseteq \frac ee\subseteq{\Uparrow}e$ that $e\lesseq e_f$ and hence $e\le e_f$, see Lemma~\ref{l:pi2}. Assuming that $e=e_f$, we conclude that $f\in{\Downarrow}e_f={\Downarrow}e$, which contradicts the choice of $f\in F\subseteq E(X)\setminus{\downarrow}e=E(X)\setminus{\Downarrow}e$, see Lemma~\ref{l:pi2}. Therefore, $e<e_f$. Now the choice of the idempotent $e$ ensures that the semigroup $Z(X)\cap H(Z)\cap{\Uparrow}e_f$ is bounded. Then the semigroup $$Z(X)\cap H(X)\cap\tfrac ee\cap{\Uparrow}f\subseteq Z(X)\cap H(X)\cap{\Uparrow}e_f$$is bounded and hence there exists $n_f\in\IN$ such that $z^{n_f}=z^{2n_f}$ for all $z\in Z(X)\cap H(X)\cap\frac ee\cap{\Uparrow}f$. Choose $m\in\IN$ such that $m\ge 2n_f$ for all $f\in F$. For every $k\ge m$ and $f\in F$ we have $z_k^{n_f}\ne z_k^{2n_f}$ by Claim~\ref{cl2:antichain}(2) and hence $z_k\notin Z(X)\cap H(X)\cap\frac ee\cap {\Uparrow}f$. Consequently, for every $k\ge m$ we have $$z_k\in Z(X)\cap H(X)\cap \tfrac ee\setminus{\Uparrow}F=V$$and hence the set $\{n\in U_s:z_n\notin V\}\subseteq\{n\in U_s:n\le m\}$ is finite.
\end{proof}

For any subset $U\subseteq \w$ consider the set $$z_U\defeq\{z_n:n\in U\}\subseteq Z(X)\cap H(X)\cap\tfrac ee.$$
Let $z_\U$ be the free ultrafilter on $X$ generated by the base $\{z_U:U\in\U\}$. The choice of $U_s$ ensures that $sz_n=z_n$ for all $n\in U_s\in\U$ and hence $sz_\U=z_\U$. The ultrafilter $z_\U$ is {\em central} in the sense that $xz_\U=z_\U x$ for all $x\in X$.

Observe that for every $x\in H_s$ we have
$$z_\U=sz_\U=x^{-1}xz_\U,$$
which implies that the ultrafilter $xz_\U$ is free.

Lemma~\ref{l:Tamura} implies the following fact.

\begin{claim}\label{cl2:xs} For every $x\in{\Uparrow}s$ we have $xs\in H_s$ and $xz_\U=(xs)z_\U\in\{yz_\U:y\in H_s\}$.
\end{claim}

In the maximal group $H_e\subseteq X^1$ consider the subgroup
$$G\defeq\{x\in H_s:xz_\U=z_\U\}.$$
Since the ultrafilter $z_\U$ is central, for any $x\in H_s$ and any $g\in G$ we have $xgx^{-1}z_\U=xgz_\U x^{-1}=xz_\U x^{-1}=xx^{-1}z_\U=sz_\U=z_\U$ and hence $xgx^{-1}\in G$, which means that $G$ is a normal subgroup in $H_s$.

By analogy with Claim~\ref{cl:well-defined} we can prove the following 

\begin{claim}\label{cl2:well-defined} Let $x,y\in H_s$ be two elements such that $xz_\U=yz_\U$. Then
\begin{enumerate}
\item $xG=yG$;
\item $xe=ye$;
\item $xtz_\U=ytz_\U$ for every $t\in{\Uparrow}s$.
\end{enumerate}
\end{claim}

Consider the set $$Y\defeq X\cup\{x z_\U:x\in H_s\}.$$ Extend the semigroup operation of the semigroup $X$ to a binary operation $*:Y\times Y\to Y$ defined by
$$a*b=\begin{cases}
ab&\mbox{if $a,b\in X$};\\
ayz_\U&\mbox{if $a\in {\Uparrow}s$ and $b=yz_\U$ for some $y\in H_s$;}\\
aye&\mbox{if $a\in X\setminus{\Uparrow}s$ and $b=y z_\U$ for some $y\in H_s$;}\\
xb z_\U&\mbox{if $a=x z_\U$ for some $x\in H_s$ and $b\in {\Uparrow}s$;}\\
xeb&\mbox{if $a=xz_\U$ for some $x\in H_s$ and $b\in X\setminus{\Uparrow}s$};\\
xeye&\mbox{if $a=xz_\U$ and $b=yz_\U$ for some $x,y\in H_s$.}
\end{cases}
$$
Claim~\ref{cl2:well-defined} ensures that the operation $*$ is well-defined.
Now we define a zero-dimensional topology $\Tau$ on $Y$ making the binary operation $*$ continuous.

Consider the ideal $$I=X\setminus\bigcup_{n\in U_s}{\Uparrow}z_n$$in $X$. By analogy with Claim~\ref{cl:nmI}, we can prove the following

\begin{claim}\label{cl2:nmI} For every distinct numbers $n,m\in\w$ we have $z_nz_m\in I$.
\end{claim}

For every $x\in X$ let
$$l_x\defeq
\begin{cases}
1&\mbox{if $x\notin{\Uparrow}s$};\\
2&\mbox{if $x\in{\Uparrow}s$};
\end{cases}
$$
and
$$U_x=\begin{cases}
U_s&\mbox{if $x\in{\Uparrow}s$};\\
\{n\in U_s:\pi(x)z_n\ne z_n\}&\mbox{if $x\notin{\Uparrow}s$}.
\end{cases}
$$
By analogy with Claim~\ref{cl:UxU} we can prove that for every $x\in X$ the set $U_x$ belongs to the ultrafilter $\U$.

For every $U\in\U$ and $\e\in\IR_+$, consider the set
$$\nabla(U,\e)\defeq I\cup \{xz_n^k:x\in X,\;n\in U\cap U_x,\;l_x\le k<\e n\}.$$

By our assumption, maximal subgroups of $X$ are projectively $Z$-bounded. In particular, the maximal subgroup $H_s$ is projecively $Z$-bounded. Then there exists a number $p_s\in \IN$ such that for every surjective homomorphism $h:H_s\to S$ to a semigroup $S$ we have $Z(S)\subseteq\korin{p_s}{E(S)}$. 

\begin{claim}\label{cl2:nabla} 
$H_sz_{U_s}\cap\nabla(\w,\tfrac1{p_s})=\emptyset$. 
\end{claim}

\begin{proof} Assuming that $H_sz_{U_s}\cap\nabla(\w,\tfrac1{p_s})\ne\emptyset$, we could find $x\in H_s$, $y\in X$ and numbers $n\in U_s$ and $m\in U_y$ such that $xz_n=yz_m^k$ for some $k\in[l_y,\tfrac1{p_s}m)$. Applying the homomorphism $q$ to the equality $xz_n=yz_m^k$, we conclude that $$q(z_n)=q(sz_n)=q(s)q(z_n)=q(x)q(z_n)=q(xz_n)=q(yz_m^k)=q(y)q(z_m)\le q(z_m)$$and hence $n=m$ by Claim~\ref{cl2:antichain}(1).
Then $xz_n=yz_m^k=yz_n^k$ and hence $$z_n=sz_n=x^{-1}xz_n=x^{-1}yz^k_n.$$
Consider the homomorphism $h:H_s\to H_{\pi(z_n)}$, $h:g\mapsto g\pi(z_n)$. By Claim~\ref{cl2:e<s} and Lemma~\ref{l:Tamura}, the homomorphism $h$ is well-defined. 

Assuming that $y\in {\Uparrow}s$, we conclude that $x^{-1}y=x^{-1}sy\in H_s$ and 
$h(x^{-1}y)=x^{-1}y\pi(z_k)=z_n^{1-k}\in Z(X)\cap H_{\pi(z_n)}\subseteq Z(H_{\pi(z_n)})$.
The choice of $p_s$ guarantees that $(x^{-1}y\pi(z_n))^{p_s}=\pi(z_n)$. Taking the equality $z_n=xy^{-1}\pi(z_n)z_n^k$ to the power $p_s$, we obtain
$$z_n^{p_s}=(xy^{-1}\pi(z_n))^{p_s}z_n^{kp_s}=\pi(z_n)z_n^{kp_s}=z_n^{kp_s}.$$
Since $p_s<2p_s=l_yp_s\le kp_s$, Claim~\ref{cl2:antichain}(2) implies that $kp_s>n$. On the other hand, we have $k<\frac1{p_s}m=\frac1{p_s} n$ and hence $p_sk<n$.

This contradiction shows that $y\notin{\Uparrow}s$. Then $n=m\in U_s\cap U_y$ and $\pi(y)z_n\ne z_n$, by the definition of $U_y$. On the other hand, $z_n=x^{-1}yz_n^k$ implies 
$$z_n=x^{-1}yz_n^k=x^{-1}yz_nz_n^{k-1}=(x^{-1}y)^2z_n^kz_n^{k-1}=(x^{-1}y)^2z_nz_n^{2(k-1)}.$$Proceeding by induction, we can show that $z_n=(x^{-1}y)^jz_nz_n^{j(k-1)}$ for every $j\in\IN$. Since the semigroup $X$ is eventually Clifford, there exists $j\in\IN$ such that $(x^{-1}y)^j\in H_{\pi(x^{-1}y)}$ and hence $(x^{-1}y)^j=\pi(x^{-1}y)(x^{-1}y)^j$. Since $\pi(x^{-1}y)\lesseq x^{-1}y\lesseq y\lesseq\pi(y)$, we can apply Lemma~\ref{l:pi2} and conclude that $\pi(x^{-1}y)=\pi(y)\pi(x^{-1}y)$. Then 
$$
z_n=(x^{-1}y)^jz_nz_n^{j(k-1)}=\pi(x^{-1}y)(x^{-1}y)^jz_nz_n^{j(k-1)}=\pi(x^{-1}y)z_n=\pi(y)\pi(x^{-1}y)z_n=\pi(y)z_n,
$$
which contradicts the inclusion $n\in U_y$.
\end{proof}

\begin{definition}\label{d2:Tau}
Let $\Tau$ be the topology on the set $Y=X\cup\{x z_\U:x\in H_s\}$ consisting of sets $W\subseteq Y$ that satisfy the following conditions:
\begin{enumerate}
\item for every $x\in W\cap X$ with $xe=x\notin H_e$ there exist $V\in\mathcal H[e)$ and such that $\Lambdae(x;V)\setminus{\Uparrow}e\subseteq W$;
\item for every $x\in W\cap H_e$ there exist $V\in\mathcal H[e)$, $U\in\U$ and $\e\in\IR_+$ such that $\Lambda^e(x;V)\cap\nabla(U,\e)\subseteq W$;
\item for any $x\in H_s$ with $xz_\U\in W$ there exists $U\in\U$ such that $xz_U\subseteq W$.
\end{enumerate}
\end{definition}

It is clear that $X$ is a nonclosed dense subset in $(Y,\Tau)$. It remains to prove that the topological space $(Y,\Tau)$ is Hausdorff and zero-dimensional, and the binary operation $*:Y\times Y\to Y$ is continuous with respect to the topology $\Tau$. This will be done in the following two lemmas.

\begin{lemma}\label{l2:T-zerodim} The topological space $(Y,\Tau)$ is Hausdorff and zero-dimensional.
\end{lemma}

\begin{proof} The following claim can be proved by analogy with Claim~\ref{cl:singleton-T}.

\begin{claim}\label{cl2:singleton-T} For every $a\in X$ with $a\ne ae$ the singleton $\{a\}$ is $\Tau$-clopen.
\end{claim}

For every $x\in H_s$ and $U\in\U$, consider the subset
$$\Lambda(xz_\U;U)\defeq \{x z_\U\}\cup \{xz_n:n\in U\cap U_s\}$$
of $Y$. The following claim can be proved by analogy with Claim~\ref{cl:ultra-T}.

\begin{claim}\label{cl2:ultra-T} For every $a\in H_s$, the family $$\mathcal B_{az_\U}\defeq \{\Lambda(az_\U;U):U\in\U\}$$ consists of $\Tau$-clopen sets and is a neighborhood base of the topology $\Tau$ at $az _\U$.
\end{claim}

The definition of the topology $\Tau$ and Theorem~\ref{t:TS} imply the following  

\begin{claim}\label{cl2:notHe-T} For every $x\in X$ with $x=xe\notin H_e$ the family
$$\mathcal B_x=\{\Lambda^e(x;V)\setminus{\Uparrow}e:V\in\mathcal H[e)\}$$ consists of $\Tau$-clopen sets and is a neighborhood base of the topology $\Tau$ at $x$.
\end{claim}

\begin{claim}\label{cl2:He-T} For every $x\in H_e$ the family
$$\mathcal B_x=\{\Lambda^e(x;V)\cap\nabla(U,\e):V\in\mathcal H[e),\;U\in\U,\;\e\in (0,\tfrac1{p_s}]\}$$ consists of $\Tau$-clopen sets and is a neighborhood base of the topology $\Tau$ at $x$.
\end{claim}

\begin{proof} Fix any $V\in\Phi$, $U\in\U$ and $\e\in \big(0,\tfrac1{p_s}\big]$. By Definition~\ref{d:remote} and Lemma~\ref{l:Lamb}(2), $\Lambdae(x;V)\cap\nabla(U,\e)\subseteq \Lambdae(x;\frac ee)\subseteq \frac{xe}{e}\subseteq{\Uparrow}e$. Observe that for every $y\in H_e\cap \frac{xe}e$, we have $y=ye=xe=x$. This observation and the definition of the topology $\Tau$ imply that the set $ \Lambdae(x;V)\cap\nabla(U,\e)$ is $\Tau$-open. To show that this set is $\Tau$-closed, consider its complement $W\defeq Y\setminus ( \Lambdae(x;V)\cap\nabla(U,\e))$ in $Y$. By Theorem~\ref{t:TS}(3) and Example~\ref{ex:He}, the set $W\cap X$ is $\Tau_{\mathcal H[e)}$-open, which implies that $W$ satisfies the conditions (1), (2) of Definition~\ref{d2:Tau}. 

To check the condition (3), fix any $y\in H_s$ with $yz_\U\in W$. By Claim~\ref{cl2:nabla}, $\nabla(\w,\e)\cap \Lambda(yz_\U;U_s)=\emptyset$ and hence 
$\Lambda(yz_\U;U_s)\subseteq Y\setminus \nabla(\w,\e)\subseteq W$.
\smallskip

Therefore, the set $\Lambdae(x;V,F)\cap\nabla(U,\e)$ is $\Tau$-clopen. The definition of the topology $\Tau$ ensures that $\mathcal B_x$ is a neighborhood base of the topology $\Tau$ at $x$.
\end{proof}

Claims~\ref{cl2:singleton-T}--\ref{cl2:He-T} imply that the topology $\Tau$ satisfies the separation axiom $T_0$ and has a base consisting of $\Tau$-clopen sets. This implies that the topological space $(Y,\Tau)$ is Hausdorff and zero-dimensional.
\end{proof}

\begin{lemma}\label{l2:conti} The binary operation $*:Y\times Y\to Y$ is continuous.
\end{lemma}

\begin{proof} Fix any elements $a,b\in Y$ and a neighborhood $O_{ab}\in\Tau$ of the element $a*b$. We should find neighborhoods $O_a,O_b\in\Tau$ of $a,b$ such that $O_a*O_b\subseteq O_{ab}$. According to the definition of the operation $*$, we will consider six cases.
\smallskip

1. First assume that $a,b\in X$. Then $a*b=ab$. This case has three subcases.
\smallskip

1.1. If $ab\ne abe$, then $ae\ne a$ and $be\ne b$ and hence $a,b$ are isolated points of $(X,\Tau)$ according to Claim~\ref{cl2:singleton-T}. Then  the neighborhoods $O_a=\{a\}$ and $O_n=\{b\}$ have the desired property: $O_a*O_b=\{a*b\}\subseteq O_{ab}$. 
\smallskip

1.2. Next, assume that $ab=abe$ and $ab\notin H_e$. Lemma~\ref{l:Tamura} implies that $ab\notin{\Uparrow}e$ and hence $a\notin{\Uparrow}e$ or $b\notin{\Uparrow}e$.
By Claim~\ref{cl2:notHe-T}, there exists a subsemigroup $V\in\mathcal H[e)$ such that $\Lambdae(ab;V)\setminus{\Uparrow}e\subseteq O_{ab}$. 

1.2.1. If $a\notin{\Uparrow}e$, then by Claim~\ref{cl2:notHe-T} and Lemma~\ref{l:Lamb}(8), for the neighborhoods $O_a\defeq\Lambdae(a;V)\setminus{\Uparrow}e\in\Tau$ and $O_b\defeq\Lambdae(b;V)\in\Tau$ we have $O_a*O_b\subseteq\Lambdae(ab;V)\setminus {\Uparrow}e\subseteq O_{ab}.$ 

1.2.2. If $b\notin{\Uparrow}e$, then by Claim~\ref{cl2:notHe-T} and Lemma~\ref{l:Lamb}(8), for the neighborhoods $O_a\defeq\Lambdae(a;V)\in\Tau$ and $O_b\defeq\Lambdae(b;V)\setminus{\Uparrow}e\in\Tau$ we have $O_a*O_b\subseteq\Lambdae(ab;V)\setminus {\Uparrow}e\subseteq O_{ab}.$ 
\smallskip

1.3. Finally, assume that $ab\in H_e$. In this case, by Claim~\ref{cl2:He-T}, there exist $V\in\mathcal H[e)$, $U\in\U$ and $\e\in\big(0,\frac1{2p_s}]$ such that $\Lambdae(ab;V)\cap\nabla(U,2\e)\subseteq O_{ab}$ and $U\subseteq U_a\cap U_b\subseteq U_s$.

If $a\ne ae$, then $O_a\defeq\{a\}$ is a $\Tau$-open neighborhood of $a$, by Claim~\ref{cl2:singleton-T}. If $a=ae$, then $ab\in H_e$ implies $a\in{\Uparrow}e$ and hence $a=ae\in H_e$ by Lemma~\ref{l:Tamura}. In this case, Claim~\ref{cl2:He-T} ensures that $$O_a\defeq \Lambdae(a;V)\cap\nabla(U,\e)$$is a $\Tau$-open neighborhood of $a$.

By analogy, define the $\Tau$-open neighborhood
$$
O_b\defeq\begin{cases}\{b\},&\mbox{if $b\ne be$};\\
\Lambdae(b;V)\cap\nabla(U,\e),&\mbox{otherwise},\\
\end{cases}
$$
of $b$.

We claim that $O_a*O_b\subseteq O_{ab}$. Given any elements $x\in O_a$ and $y\in O_a$, we should prove that $xy\in O_{ab}$. By Lemma~\ref{l:Lamb}(8), $$xy\in O_aO_b\subseteq \Lambdae(a;V)\cdot\Lambdae(b;V)\subseteq \Lambdae(ab;V).$$ 
If $xy\in I$, then 
$$xy\in \Lambdae(ab;V)\cap I\subseteq \Lambdae(ab;V)\cap\nabla(U,2\e)\subseteq O_{ab}$$ and we are done. So, we assume that $xy\notin I$. In this case $x,y\notin I$.

Next, we consider four subcases.
\smallskip

1.3.1. If $a\ne ae$ and $b\ne be$, then $O_a=\{a\}$, $O_b=\{b\}$, and $xy=ab\in O_{ab}$.
\smallskip

1.3.2. Assume that $a\ne ae$ and $b=be$. Then $O_a=\{a\}$ and $x=a\notin H_e$. Also $y\in O_b=\Lambdae(b;V)\cap\nabla(U,\e)$. It follows from $y\in \nabla(U,\e)\setminus I$ that $y=tz_n^k$ for some $t\in X$, $n\in U\cap U_t$, and $k\in[l_t,\e n)$. 
\smallskip

1.3.2.1. If $t\in {\Uparrow}s$, then $U_t=U_s$ and $l_t=2$. 
\smallskip

1.3.2.1.1. If $a\in{\Uparrow}s$, then $xt=at\in {\Uparrow}s$, $n\in U\cap U_t=U\cap U_s=U\cap U_{xt}$, $l_{xt}=2=l_t\le k<\e n<2\e n$ and hence 
$$xy=xtz_n^k\in \Lambdae(ab;V)\cap \nabla(U,2\e)\subseteq O_{ab}.$$
\smallskip

1.3.2.1.2. If $a\notin{\Uparrow}s$, then $xt=at\notin{\Uparrow}s$, $l_{at}=1$ and $n\in U\subseteq U_a$. The definition of the set $U_a$ implies $\pi(a)z_n\ne z_n$. 
We claim that $\pi(at)z_n\ne z_n$. Indeed, assuming that $\pi(at)z_n=z_n$ and applying Lemma~\ref{l:pi2}, we conclude that $z_n=\pi(at)z_n=\pi(a)\pi(at)z_n=\pi(a)z_n$, which contradicts $n\in U_a$. This contradiction shows that $\pi(at)z_n\ne z_n$ and hence $n\in U\cap U_{at}$ and $xy=atz_n^k\in \Lambdae(ab;V)\cap\nabla(U,2\e)\subseteq O_{ab}$ as $l_{at}=1\le k<\e n<2\e n$.
\smallskip

1.3.2.2. If $t\notin{\Uparrow}s$, then $l_t=1$ and $n\in U_t$ implies $\pi(t)z_n\ne z_n$. Repeating the argument from the case 1.3.2.1.2, we can show that $\pi(at)z_n\ne z_n$. Then $at\notin{\Uparrow}s$, $n\in U\cap U_{at}$ and $l_{at}=1=l_t\le k<\e n<2\e n$. Now the definition of $\nabla(U,2\e)$ ensures that $xy=atz_n^k\in \Lambdae(ab;V)\cap\nabla(U,2\e)\subseteq O_{ab}$.
\smallskip

1.3.3. The case $a=ae$ and $b\ne be$ can be considered by analogy with the case 1.3.2.
\smallskip

1.3.4. Finally, assume that $a=ae$ and $b=be$. In this case the assumption $ab\in H_e$ of case 1.3 implies that $a,b\in{\Uparrow}e$ and hence $a=ae\in H_e$ and $b=be\in H_e$, by Lemma~\ref{l:Tamura}. The choice of the neighborhoods $O_a,O_b$ ensures that $x\in O_a=\Lambdae(a;V)\cap\nabla(U,\e)$ and $y\in O_b=\Lambdae(b;V)\cap\nabla(U,\e)$. The inclusions $x,y\in\nabla(U,\e)\setminus I$ imply that $x=tz_n^k$ for some $t\in X$, $n\in U\cap U_t$ and $k\in[l_t,\e n)$, and $y=t'z_{n'}^{k'}$ for some $t'\in X$, $n'\in U\cap U_{t'}$ and $k'\in[l_{t'},\e n')$. If $n\ne n'$, then by Claim~\ref{cl2:nmI},  $xy=tz_n^kt'z_{n'}^{k'}=tt'z_n^kz_{n'}^{k'}\in I$, which contradicts our assumption. So, $n=n'$ and $2\le k+k'<2\e n$.
\smallskip

1.3.4.1. If $t,t'\in{\Uparrow}s$, then $tt'\in {\Uparrow}s$ and $l_{tt'}=2=l_{t}=l_{t'}$ and $n=n'\in U\cap U_t\cap U_{t'}=U\cap U_s=U\cap U_{tt'}$. Then 
$xy=tz_n^kt'z_n^{k'}=tt'z_n^{k+k'}\in \Lambdae(ab;V)\cap\nabla(U,2\e)\subseteq O_{ab}.$
\smallskip

1.3.4.2. If $t\notin{\Uparrow}s$, then $tt'\notin{\Uparrow}s$ and $n\in U_t$ imply $\pi(t)z_n\ne z_n$ and $\pi(tt')z_n\ne z_n$. Then $n\in U\cap U_{tt'}$ and $xy=tz_n^kt'z_n^{k'}=tt'z_n^{k+k'}\in\Lambdae(ab;V)\cap\nabla(U,2\e)\subseteq O_{ab}.$
\smallskip

1.3.4.3. The case $t'\notin{\Uparrow}s$ can be treated by analogy with the case 1.3.4.2.

This completes the analysis of the case 1.
\smallskip

2. Assume that $a\in{\Uparrow}s$ and $b=yz_\U$ for some $y\in H_s$. In this case $a*b=ayz_\U$ and by Claim~\ref{cl2:ultra-T}, there exist a set $U\in\U$ such that $\Lambda(ayz_\U;U)\subseteq O_{ab}$. It follows from $a\in{\Uparrow}s$ and $e<s$ that $ae\ne a$ and hence $O_a\defeq\{a\}$ is a $\Tau$-open neighborhood of $a$.  Consider the neighborhood $O_b\defeq \Lambda(yz_\U;U)$ of $b$ and observe that $O_a*O_b=\Lambda(ayz_\U,U)\subseteq O_{ab}$.
\smallskip

3. Assume that $a\in X\setminus{\Uparrow}s$ and $b=yz_\U$ for some $y\in H_s$. In this case $a*b=aye\in X$. By Claims~\ref{cl2:notHe-T} and \ref{cl2:He-T}, there exist $V\in\Phi$, $U\in\U$ and $\e\in(0,\frac1{2p_s}]$ satisfying the following conditions:
\begin{itemize}
\item $U\subseteq U_a$;
\item if $aye\notin H_e$, then $\Lambdae(aye;V)\setminus{\Uparrow}e\subseteq O_{ab}$;
\item if $aye\in H_e$, then $\Lambdae(aye;V)\cap\nabla(U,2\e)\subseteq O_{ab}$.
\end{itemize}
By Claim~\ref{cl2:zn-converge}, the set $U'\defeq\{n\in U:n\e>1,\;z_n\in V\}$ belongs to the ultrafilter $\U$.
Replacing  the set $U$ by $U'$, we can assume that $U=U'$. Then for every $n\in U=U',$ the element $yz_n\in\Lambda(yz_\U;U)$ belongs to the set $\frac{ye}e\cdot V\subseteq \Lambdae(ye;V)$. 

Let $O_b\defeq\Lambda(yz_\U;U)$ and $$O_a\defeq \begin{cases}\{a\}&\mbox{if $a\ne ae$};\\
\Lambdae(a;V)\setminus{\Uparrow}e&\mbox{if $a=ae\notin H_e$};\\
\Lambdae(a;V)\cap \nabla(U,\e)&\mbox{if $a\in H_e$}.
\end{cases}
$$
By Claims~\ref{cl2:singleton-T}--\ref{cl2:He-T}, the sets $O_a,O_b$ are $\Tau$-open. To show that $O_a*O_b\subseteq O_{ab}$, take any elements $a'\in O_a$ and $b'\in O_b$, and consider two cases.
\smallskip

3.1. Assume that $b'=yz_\U$ and consider two subcases.

3.1.1. If $aye\notin H_e$, then $aye\notin{\Uparrow}e$, by Lemma~\ref{l:Tamura}. It follows from $ye\in H_se\subseteq H_e$ and $aye\notin{\Uparrow}e$ that $a\notin{\Uparrow}e$. Then $a'*b'\in O_aye\subseteq(\Lambda(a;V)\setminus{\Uparrow}e)ye\subseteq \Lambdae(aye;V)\setminus {\Uparrow}e\subseteq O_{ab}$, by Lemma~\ref{l:Lamb}(7). 
\smallskip

3.1.2. If $aye\in H_e$, then $a'*b'=a'ye\in (\Lambdae(a;V)ye)\cap{\Downarrow}e\subseteq\Lambdae(aye;V)\cap I\subseteq\Lambdae(aye;V)\cap\nabla(U,2\e)\subseteq O_{ab}$.

3.2. Next, assume that $b'\ne yz_\U$ and hence $b'=yz_n$ for some $n\in U$. It follows from $n\in U=U'$ that $1<n\e$ and $z_n\in V$.

3.2.1. If $aye\notin H_e$, then $aye\notin{\Uparrow}e$ and  $a\notin{\Uparrow}e$. Then $a'*b'\in O_ayz_n\subseteq(\Lambdae(a;V) \setminus{\Uparrow}e)\cdot\Lambdae(ye;V)\subseteq \Lambdae(aye;V)\setminus {\Uparrow}e\subseteq O_{ab}$. 
\smallskip

3.2.2. If $aye\in H_e$ and $a'*b'\in I$, then $a'*b'=a'yz_n\in \big(\Lambdae(a;V)\cdot\Lambdae(ye;V)\big)\cap I\subseteq\Lambdae(aye;V)\cap I\subseteq\Lambdae(aye;V)\cap\nabla(U,2\e)\subseteq O_{ab}$.
\smallskip

3.2.3. Finally, assume that $aye\in H_e$ and $a'*b'\notin I$. Then $a'\notin I$. It follows from $aye\in H_e$ that $a\in{\Uparrow}e$. 

3.2.3.1. If $a\ne ae$, then $a'\in O_a=\{a\}$ and hence $a'*b'=ayz_n$. It follows from $a\notin{\Uparrow}s$ that $ay\notin{\Uparrow}s$ and $l_{ay}=1$. Also $a\notin{\Uparrow}s$ and $n\in U\subseteq U_a$ imply $\pi(a)z_n\ne z_n$. Using Lemma~\ref{l:pi2}, we can show that $\pi(ay)\le\pi(a)$ and hence $\pi(a)\pi(ay)=\pi(ay)$. Assuming that $\pi(ay)z_n=z_n$, we conclude that $z_n=\pi(ay)z_n=\pi(a)\pi(ay)z_n=\pi(a)z_n$, which contradicts $n\in U_a$. This contradiction shows that $\pi(ay)z_n\ne z_n$ and hence $n\in U_{ay}$. Then $a'*b'=ayz_n\in(a\Lambdae(ye;V))\cap\nabla(U,\e)\subseteq \Lambdae(aye;V)\cap\nabla(U,2\e)\subseteq O_{ab}$ as $n\in U\cap U_{ay}$ and $l_{ay}=1<n\e<2n\e$.
\smallskip

3.2.3.2. If $a=ae$, then $a=ae\in({\Uparrow}e)e\subseteq H_e$ by Lemma~\ref{l:Tamura} and hence $a'\in O_a=\Lambdae(a;V)\cap\nabla(U,\e)$. It follows from $a'\in \nabla(U,\e)\setminus I$ that $a'=x'z_{n'}^{k'}$ for some $x'\in X$, $n'\in U\cap U_{x'}$ and $k'\in [l_{x'},\e n')$. If $n\ne n'$, then $a'*b'=x'z_{n'}^{k'}yz_n=x'yz_{n'}^{k'}z_n\in I$ by Claim~\ref{cl2:nmI}, which contradicts or assumption. This contradiction shows that $n=n'$.
\smallskip

We claim that $n\in U_{x'y}$. Indeed, if $x'\in{\Uparrow}s$, then $x'y\in{\Uparrow}s$ and hence $n=n'\in U_{x'}=U_s=U_{yx'}$.
If $x'\notin{\Uparrow}s$, then $x'y\notin{\Uparrow}s$ and hence $n=n'\in U_{x'}$ implies $\pi(x')z_n\ne z_n$. Lemma~\ref{l:pi2} ensures that $\pi(x'y)=\pi(x')\pi(x'y)$. Assuming that $\pi(x'y)z_n=z_n$, we conclude that $z_n=\pi(x'y)z_n=\pi(x')\pi(x'y)z_n=\pi(x')z_n$, which contradicts the inclusion $n\in U_{x'}$. This contradiction shows that $\pi(x'y)z_n\ne z_n$ and hence $n\in U_{x'y}$. 

Then $a'*b'=x'z_{n'}^{k'}yz_n=x'yz_n^{k'+1}\in\nabla(U,2\e)$ as $n\in U\cap U_{x'y}$ and $l_{x'y}\le 2\le k'+1<\e n'+1=\e n+1<\e n+\e n=2\e n$. Finally, $a'*b'=a'yz_n\in (\Lambdae(a;V)\cdot\Lambdae(ye;V))\cap\nabla(U,2\e)\subseteq \Lambdae(aye;V)\cap\nabla(U,2\e)\subseteq O_{ab}.$
\smallskip

4. Assume that $a=xz_\U$ for some $x\in H_s$ and $b\in {\Uparrow}s$. This case can be considered by analogy with case 2.
\smallskip

5. Assume that $a=x z_\U$ for some $x\in H_s$ and $b\in X\setminus{\Uparrow}s$.
 This case can be considered by analogy with case 3.
\smallskip

6. Finally, assume that $a=xz_\U$ and $b=yz_\U$ for some $x,y\in H_s$. In this case $a*b=xeye=xye\in H_e$. By Claim~\ref{cl2:He-T}, there exist $V\in\mathcal H[e)$, $U\in\U$ and $\e\in\big(0,\frac1{2p_s}\big]$ such that $\Lambdae(xye;V)\cap\nabla(U,2\e)\subseteq O_{ab}$.  By Claim~\ref{cl2:zn-converge}, there exists a set $U'\in\U$ such that for every $n\in U'$ we have $1<\e n$ and $z_n\in V$.  Replacing the set $U$ by $U\cap U'\cap U_s$, we can assume that $U\subseteq U'\cap U_s$. Then for every $n\in U\subseteq U'$ and $g\in H_s$ we have
$gz_n\in\tfrac{ge}e\cdot V\subseteq\Lambdae(ge;V)$.

By Claim~\ref{cl2:ultra-T}, the sets $O_a\defeq\Lambda(x z_\U;U)$ and $O_b\defeq\Lambda(y z_\U;U)$ are $\Tau$-open. We claim that $O_a*O_b\subseteq O_{ab}$. Given any elements $a'\in O_a$ and $b'\in O_b$, we should show that $a'*b'\in O_{ab}$. Consider four cases.

6.1. If $a'=x z_\U$ and $b'=y z_\U$, then $a'*b'=a*b\in O_{ab}$.

6.2. If $a'=xz_\U$ and $b'\ne yz_\U$, then $b'=yz_n$ for some $n\in U\subseteq  U'\cap U_s$. By the definition of the binary operation $*$ and Lemma~\ref{l:Lamb}(6), 
$$a'*b'=xeyz_n\in (xe\Lambdae(ye;V))\cap I\subseteq \Lambdae(xeye;V)\cap\nabla(U,2\e)\subseteq O_{ab}.$$

6.3. The case of $a'\ne xz_\U$ and $b'=yz_\U$ can be considered by analogy with the case 6.2.

6.4. If $a'\ne xz_\U$ and $b'\ne yz_\U$, then $a'=xz_n$ and $b'=yz_m$ for some $n,m\in U\subseteq U_s\cap U'$. If $n\ne m$, then $a'*b'=xz_nyz_m\in I$ by Claim~\ref{cl2:nmI} and hence
$$a'*b'=xz_nyz_m\in \big(\Lambdae(xe;V)\cdot\Lambdae(ye;V)\big)\cap I\subseteq\Lambdae(xeye;V)\cap\nabla(U,2\e)\subseteq O_{ab}.$$
If $n=m$, then $a'*b'=xz_nyz_m=xyz_n^2$. Since $n\in U=U\cap U_s=U\cap U_{xy}\subseteq U'$, we have  $l_{xy}=2<2\e n$ and hence 
$$a'*b'=xz_nyz_m=xyz_n^2\in \big(\Lambdae(xe;V)\cdot\Lambdae(ye;V)\big)\cap\nabla(U,2\e)\subseteq \Lambdae(xeye;V)\cap\nabla(U,2\e)\subseteq O_{ab}.$$
\end{proof} 

\begin{lemma}\label{l2:non-IC} $(Y,\Tau)$ is a  Hausdorff zero-dimensional topological semigroup containing $X$ as a nonclosed dense subsemigroup.
\end{lemma}

\begin{proof} By Lemma~\ref{l2:T-zerodim}, the space $(Y,\Tau)$ is Hausdorff and zero-dimensional and by Lemma~\ref{l2:conti}, the binary operation $*:Y\times Y\to Y$ is continuous. The density of $X$ in the Hausdorff topological space $(Y,\Tau)$ and the associativity of the semigroup operation $*{\restriction}_{X\times X}$ of $X$ imply that the binary operation $*$ is associative and hence $(Y,\Tau)$ is a topological semigroup containing $X$ as a dense subsemigroup.
\end{proof}

Lemma~\ref{l2:non-IC} implies that the semigroup $X$ is not injectively $\mathsf{T_{\!z}S}$-closed, which contradicts our assumption. This contradiction shows that the semigroup $Z(X)\cap H(X)$ is bounded and completes the long proof of Proposition~\ref{p2:vsim-grobikam-grobik}.
\end{proof}

Propositions~\ref{p:vsim-grobikam-grobik} and \ref{p2:vsim-grobikam-grobik} imply the following 

\begin{corollary}\label{c:bounded} Let $X$ be an injectively $\mathsf{T_{\!z}S}$-closed semigroup. If the poset $E(X)$ is well-founded and the semigroup $X$ is viable, nonsingular, eventually Clifford and group-bounded, then its center $Z(X)$ is bounded.
\end{corollary}

\begin{proof} Assume that the poset $E(X)$ is well-founded and the semigroup $X$ is viable, nonsingular, eventually Clifford and group-bounded. Then each maximal subgroup $H_s$ of $X$ is bounded and hence $H_e$ is projectively $Z$-bounded. By Proposition~\ref{p2:vsim-grobikam-grobik}, the semigroup $Z(X)\cap H(X)$ is bounded and by Proposition~\ref{p:vsim-grobikam-grobik}, there exists $\ell\in \IN$ such that $Z(X)\subseteq\korin{\ell}{H(X)}$. Since $Z(X)\cap H(X)$ is bounded, there exists $b\in\IN$ such that $Z(X)\cap H(X)\subseteq\korin{b}{E(X)}$. Then for every $z\in Z(X)$ we have $z^\ell\in Z(X)\cap H(X)$ and $z^{\ell b}=(z^\ell)^b\in E(X)$, which means that the semigroup $Z(X)$ is bounded.
\end{proof}

\section{The Clifford-nonsingularity of an injectively $\mathsf{T_{\!z}S}$-closed commutative semigroup}\label{s:Cs}

In this section we find conditions ensuring that an injectively $\mathsf{T_{\!z}S}$-closed   semigroup $X$ has Clifford-nonsingular {\em ideal center}
$$I\!Z(X)\defeq\{z\in Z(X):zX\subseteq Z(X)\}.$$
Observe that $I\!Z(X)$ is the largest ideal of $X$, which is contained in $Z(X)$. If a semigroup $X$ is commutative, then $I\!Z(X)=Z(X)=X$.

\begin{proposition}\label{p:vsim-grobikam-grobik2} Let $X$ be an injectively $\mathsf{T_{\!z}S}$-closed semigroup. If the poset $E(X)$ is well-founded and the semigroup $X$ is viable and eventually Clifford, then the semigroup $I\!Z(X)$ is not Clifford-singular.
\end{proposition}

\begin{proof}By Lemma~\ref{l:+1}, we lose no generality assuming that the injectively $\mathsf{T_{\!z}S}$-closed semigroup $X$ is unital. By Theorem~\ref{t:Z}, the center $Z(X)$ of $X$ is periodic, nonsingular, chain-finite, and group-finite.

Let $\pi:X\to E(X)$ be the map assigning to each $x\in X$ the unique idempotent $\pi(x)$ such that $x^\IN\cap H_{\pi(x)}\ne\emptyset$. Since $X$ is eventually Clifford, the map $\pi$ is well-defined.

Let $q:X\to X/_{\Updownarrow}$ be the quotient homomorphism of $X$ to its semilattice reflection. Since the semigroup $X$ is viable and eventually Clifford, the restriction $q{\restriction}_{E(X)}:E(X)\to X_{\Updownarrow}$ is bijective and for every $x\in X$ we have
$$\textstyle{\Updownarrow}x=q^{-1}(q(x))=\pi^{-1}(\pi(x))=\korin{\IN}{H_{\pi(x)}}.$$

To derive a contradiction assume the ideal center $I\!Z(X)$ is Clifford-singular and hence there exists an infinite subset $A\subseteq I\!Z(X)\setminus H(I\!Z(X))$ such that $AA\subseteq H(I\!Z(X))$. Since the semigroup $Z(X)$ is periodic, $H(Z(X))=H(X)\cap Z(X)$ and $H(I\!Z(X))=H(X)\cap I\!Z(X)$.

Let $M$ be the set of idempotents $u\in E(X)$ for which there exists an infinite set $A\subseteq {\Downarrow}u\cap I\!Z(X)\setminus H(X)$ such that $AA\subseteq H(X)$. The set $M$ contain the unit of $X$ and hence is not empty. Since the poset $E(X)$ is well-founded, there exists an idempotent $\mu\in M$ such that $M\cap {\downarrow}\mu=\{\mu\}$. Fix an infinite subset $A\subseteq {\Downarrow}\mu\cap I\!Z(X)\setminus H(X)$ such that $AA\subseteq H(X)$.

\begin{lemma}\label{l:mu-fin} For every $b\in X\setminus{\Uparrow}\mu$, the set $bA\setminus H(X)$ is finite.
\end{lemma}

\begin{proof} To derive a contradiction, assume that for some $b\in X\setminus{\Uparrow}\mu$ the set $bA\setminus H(X)$ is infinite. Choose an infinite set $A'\subseteq A$ such that $bA'\subseteq bA\setminus H(X)$ and the map $A'\to bA'$, $a\mapsto ba$, is injective.

\begin{claim}\label{cl:bn-infty} For every $n\in\IN$ and infinite subset $J\subseteq A'$ the set $b^nJ$ is infinite.
\end{claim}

\begin{proof} To derive a contradiction, assume that for some $n\in\IN$ and infinite set $J\subseteq A'$ the set $b^nI$ is finite. We can assume that $n$ is the smallest number for which there exists an infinite set $J\subseteq A'$ such that $b^nJ$ is finite.  The injectivity of the map $A'\to bA'$, $a\mapsto ba$, ensures that $n>1$.  Let $J\subseteq A'$ be an infinite set such that $b^{n}J$ is finite. By the Pigeonhole Principle, there exists an infinite set $J'\subseteq J$ such that $b^{n}J'$ is a singleton. The minimality of $n$ ensures that the set $b^{n-1}J'$ is infinite. The inclusion $J'\subseteq I\!Z(X)$ implies $b^{n-1}J'\subseteq I\!Z(X)\subseteq Z(X)$. Fix any element $j\in J'$ and observe that  for any elements $a,a'\in J'\subseteq A\subseteq Z(X)$ we have
$$b^{n-1}ab^{n-1}a'=b^{n}ab^{n-2}a'=b^{n}jb^{n-2}a'=b^{n-2}jb^{n}a'=b^{n-2}jb^{n}j=b^{2n-2}j^2.$$ Therefore, $(b^{n-1}J')^2=\{b^{2n-2}j^2\}$, which contradicts the nonsingularity of $Z(X)$. This contradiction completes the proof of the claim.
\end{proof}

Observe that $bA'\subseteq bA\subseteq {\Downarrow}b\cdot{\Downarrow}\mu\subseteq{\Downarrow}b\mu={\Downarrow}\pi(b\mu)$. It follows from $b\notin{\Uparrow}\mu$ that  $\pi(b\mu)\in{\Downarrow}\mu\setminus{\Uparrow}\mu$ and hence $\pi(b\mu)<\mu$, by Lemma~\ref{l:pi2}. The minimality of $\mu$ in $M$ ensures that $\pi(b\mu)\notin M$.

Since the semigroup $X$ is eventually Clifford, there exists $n\in \IN$ such that $b^n\in H_{\pi(b)}$. Then $(b^nA')^2\subseteq (b^n A)^2=b^{2n}A^2\subseteq H(X)$, see Lemma~\ref{l:CH}. Observe that $b^nA'\subseteq X\cdot I\!Z(X)\subseteq  I\!Z(X)$, and hence $b^nA'\subseteq{\Downarrow}\pi(b\mu)\cap I\!Z(X)\setminus H(X)$. Since $\pi(b\mu)\notin M$, the set $b^nA'\setminus H(X)$ is finite. 
By Claim~\ref{cl:bn-infty},  the set $A''=\{a\in A':b^na\notin H(X)\}$ is finite and hence the set $A'\setminus A''$ is infinite. 

Let $k\in\IN$ be the smallest number for which there exists an infinite set $J\subseteq A'\setminus A''$ such that $b^{k}J\subseteq H(X)$. Since $b^n(A'\setminus A'')\subseteq H(X)$, the number $k$ is well-defined and $k\le n$. Taking into account that  $bA'$ is an infinite subset of $X\setminus H(X)$, we conclude that $k>1$. Let $J$ be an infinite subset of $A'\setminus A''$ such that $b^{k}J\subseteq H(X)$. By the minimality of $k$,  the set $\{a\in J:b^{k-1}a\in H(X)\}$ is finite. Replacing $J$ by its infinite subset $\{a\in J:b^{k-1}a\notin H(X)\}$, we can additionally assume that $b^{k-1}J\subseteq X\setminus H(X)$. By Claim~\ref{cl:bn-infty}, the set $b^{k-1}J$ is infinite.  Observe that for every $a\in J\subseteq A\subseteq I\!Z(X)$ we have $b^{k}a\in H(X)\cap Z(X)$ and $\pi(b^ka)\Updownarrow b^ka\lesseq b\Updownarrow\pi(b)$ and hence $\pi(b^ka)\le \pi(b)$ by Lemma~\ref{l:pi2}. Then $$b^{k}a=b^{k}a\,\pi(b^{k}a)=b^{k}a\,\pi(b^{k}a)\,\pi(b)=b^{k}a\,\pi(b)=\pi(b)b^ka.$$ 
By Lemma~\ref{l:pi-well-defined}, the inclusion $b^n\in H_{\pi(b)}$ implies that $b^{n+1},b^{n+2}\in H_{\pi(b)}$. 
Let $g$ be the inverse element to $b^{n+2}$ in the group $H_{\pi(b)}$.
Then for any $a,a'\in J$ we have
\begin{multline*}
b^{k-1}ab^{k-1}a'=b^{k}ab^{k-2}a'=\pi(b)b^{k}ab^{k-2}a'=\pi(b)b^{k-2}ab^{k}a'=
\pi(b)b^{k-2}a\pi(b)b^{k}a'=\\
gb^{n+2}b^{k-2}agb^{n+2}b^{k-1}a'=gb^{n}b^{k}agb^{n+1}b^{k}a'.
\end{multline*}
It follows from $gb^n\in H(X)$ and $b^ka\in H(X)\cap Z(X)$ that  $gb^nb^ka\in H(X)$, see Lemma~\ref{l:CH}. By analogy we can show that $gb^{n+1}b^ka'\in H(X)$. Taking into account that $a\in J\subseteq A\subseteq I\!Z(X)$, we conclude that $gb^nb^ka\in Z(X)$ and hence  $b^{k-1}ab^{k-1}a'=gb^nb^kagb^nb^ka'\in H(X)$, by Lemma~\ref{l:CH}.  
 
Therefore, $b^{k-1}J$ is an infinite subset of ${\Downarrow}\pi(b\mu)\cap I\!Z(X)\setminus H(X)$ such that $(b^{k-1}J)^2\subseteq H(X)$. But the existence of such a set contradicts $\pi(b\mu)\notin M$. This completes the proof of Lemma~\ref{l:mu-fin}.
\end{proof}

Since the semilattice $EZ(X)$ is chain-finite and $A\subseteq Z(X)\subseteq\bigcup_{e\in EZ(X)}{\Uparrow}e$, there exists an idempotent $e\in EZ(X)$ such that $A\cap{\Uparrow}e$ is infinite but for any  idempotent $f\in EZ(X)$ with $e<f$, the set $A\cap{\Uparrow}f$ is finite.

\begin{claim}\label{cl:Af-fin} For every idempotent $f\in E(X)$ the set $A\cap{\Updownarrow}f=A\cap\korin{\IN}{H_f}$ is finite.
\end{claim}

\begin{proof} 
%
To derive a contradiction, assume that $A\cap{\Updownarrow}f$ is infinite. By Lemma~\ref{l:Tamura}, the subgroup $H_f$ is an ideal in ${\Uparrow}f$. Then $af\in H_f\cap Z(X)$ for all $a\in A\cap{\Updownarrow}f$. Since $Z(X)$ is group-finite and periodic, the intersection $H_f\cap Z(X)$ is a finite subgroup of $H_f$. By the Pigeonhole Principle, there exists $g\in H_f\cap Z(X)$ such that the set $B=\{a\in  A\cap{\Updownarrow}f:af=g\}\subseteq A$ is infinite.  We claim that $B^2=\{g^2\}$. Indeed,  for every $a,b\in B$ we have $ab\in AA\cap{\Updownarrow}f\subseteq H(X)\cap{\Updownarrow}f=H_f$ and hence $ ab=abff=afbf=gg$. So, $BB=\{gg\}$ is a singleton, which contradicts the nonsingularity of the semigroup $Z(X)$.
\end{proof}

Since the set $A\cap{\Uparrow} e$ is infinite and the sets $A\cap{\Updownarrow}e$ and $A\cap{\Updownarrow}\mu$ are finite, we can replace the infinite set $A$ by the set $A\cap {\Uparrow}e\setminus({\Updownarrow}e\cup{\Updownarrow}\mu)$ and assume that $$A\subseteq ({\Uparrow}e\cap{\Downarrow}\mu)\setminus({\Updownarrow}e\cup{\Updownarrow}\mu).$$ Such a choice of $A$ and Lemma~\ref{l:pi2} imply

\begin{claim}\label{cl3:e<mu} $e<\pi(a)<\mu$ for every $a\in A$.
\end{claim}

\begin{claim}\label{cl:AU-fin} For every idempotent $f\in E(X)\setminus{\Downarrow}e$ the set $A\cap {\Uparrow}f$ is finite.
\end{claim}

\begin{proof} Since the semigroup $Z(X)$ is chain-finite, the nonempty subsemilattice $EZ(X)\cap{\uparrow}e\cap{\uparrow}f\ni 1$ of $Z(X)$ contains the smallest element $e_f$. Then $e_f\in{\uparrow}e$ and $e\le e_f$. Assuming that $e=e_f$, we conclude that $f\in{\downarrow}e_f={\downarrow}e\subseteq{\Downarrow}e$, which contradicts the choice of $f$. Therefore, $e<e_f$. Now the maximality of $e$ and the inclusion $A\subseteq Z(X)\cap {\Uparrow}e$  ensure that the set
$$A\cap{\Uparrow}f=A\cap{\Uparrow}e\cap{\Uparrow}f\subseteq A\cap\pi^{-1}[EZ(X)\cap{\uparrow}e\cap{\uparrow}f]\subseteq A\cap\pi^{-1}[{\uparrow}e_f]=A\cap{\Uparrow}e_f$$is finite.
\end{proof}

\begin{claim}\label{cl:antichain2} There exists a sequence $\{a_k\}_{k\in\w}\subseteq A$ such that $\pi(a_k)\not\le \pi(a_n)$ for any distinct numbers $k,n\in\w$.
\end{claim}

\begin{proof} Choose any sequence $(x_k)_{k\in\w}$ of pairwise distinct elements of the infinite set $A$. Let $[\w]^2$ be the family of two-element subsets of $\w$. Consider the function $\chi:[\w]^2\to\{0,1,2\}$ defined by
$$\chi(\{n,m\})=\begin{cases}0&\mbox{if $\pi(x_n)=\pi(x_m)$};\\
1&\mbox{if $\pi(x_n)<\pi(x_m)$ or $\pi(x_m)<\pi(x_n)$};\\
2&\mbox{otherwise}.
\end{cases}
$$
By the Ramsey Theorem 5 \cite{Ramsey}, there exists an infinite set $\Omega\subseteq \w$ such that $\chi[[\Omega]^2]=\{c\}$ for some $c\in\{0,1,2\}$. If $c=0$, then the set $\{\pi(x_n)\}_{n\in\Omega}$ contains a unique idempotent $u$ and hence $\{x_k\}_{k\in\Omega}$ is an infinite subset of $A\cap {\Updownarrow}u$, which contradicts Claim~\ref{cl:Af-fin}.  Therefore, $c\ne 0$. If $c=1$, then the set $\{\pi(x_k)\}_{k\in\Omega}$ is an infinite chain in $EZ(X)$ which is not possible as $Z(X)$ is chain-finite. Therefore, $c=2$ and hence $\{\pi(x_k)\}_{k\in\Omega}$ is an infinite antichain in $E(X)$. Write the infinite set $\Omega$ as $\{n_k\}_{k\in\w}$ for some strictly increasing sequence $(n_k)_{k\in\w}$. For every $k\in\w$ put $a_k=x_{n_k}$. Then the sequence $(a_k)_{k\in\w}$ has the required property.
\end{proof}

Let $(a_n)_{n\in\w}$ be the sequence given by Claim~\ref{cl:antichain2}. Replacing the set $A$ by $\{a_n\}_{n\in\w}$, we can assume that $A=\{a_n\}_{n\in\w}$ and hence $\pi(a)\not\le\pi(b)$ for any distinct elements $a,b\in A$.

Consider the ideal
$$J\defeq X\setminus{\Uparrow} A$$
in the semigroup $X$.

\begin{claim}\label{cl:J} For any distinct elements $x,y\in A$ we have $xy\in J$ and $\pi(x)\pi(y)\in J$.
\end{claim}

\begin{proof} Assuming that $xy\notin J$, we can find $a\in A$ such that $xy\in{\Uparrow}a$. Then $\pi(a)\Updownarrow a\lesseq xy\lesseq x\Updownarrow\pi(x)$ and hence $\pi(a)\le\pi(x)$, by Lemma~\ref{l:pi2}. By Clam~\ref{cl:antichain2}, $x=a$. By analogy we can prove that $y=a$ and hence $x=y$, which contradicts the choice of $x,y$.
This contradiction shows that $xy\in J$. By analogy we can prove that $\pi(x)\pi(y)\in J$.
\end{proof}

Fix any free ultrafilter $\U$ on $X$ such that $A\in\U$.
For every $U\in\U$ and $F\in[EZ(X)\setminus{\downarrow}e]^{<\w}$, consider the set
$$\nabla(U,F)\defeq \big((J\cap E(I\!Z(X))\cap\tfrac ee)\cup\pi[U\cap A]\big)\setminus {\uparrow} F.$$ The periodicity of $Z(X)$ and Claim~\ref{cl:J} imply that $\nabla(U,F)$ is a subsemilattice of the semilattice $E(I\!Z(X))\cap{\uparrow}e$.

It follows that the family
$$\Phi\defeq\big\{\nabla(U,F):U\in\U,\;F\in[EZ(X)\setminus{\downarrow}e]^{<\w}\big\}$$is an $e$-base for $X$. Let $\Tau_\Phi$ be the topology on $X$, generated by the $e$-base $\Phi$. By Lemma~\ref{l:base-e} and Theorem~\ref{t:TS}(5a), $\Phi$ is a regular $e$-base, $(X,\Tau_\Phi)$ is a Hausdorff zero-dimensional topological semigroup and for every $x\in X$, the family
$$\mathcal B_x\defeq\{\Lambdae(x;V):V\in\Phi\}$$is a neighborhood base of the topology $\Tau_{\Phi}$ at $x$.

Let $\Tau_\Phi'$ be the topology on $X$, generated by the subbase $\Tau_\Phi\cup\{X\setminus{\Uparrow}e\}\cup\{\{x\}:x\in X\setminus I\!Z(X)\}$. Observe that for every $x\in I\!Z(X)\cap{\Uparrow}e$, the family $\mathcal B_x=\{\Lambdae(x;V):V\in\Phi\}$ remains a neighborhood base at $x$ in the topological space $(X,\Tau'_\Phi)$. By Theorem~\ref{t:TS}(4), the topological space $(X,\Tau'_\Phi)$ is Hausdorff and zero-dimensional.

\begin{claim} $(X,\Tau'_\Phi)$ is a topological semigroup.
\end{claim}

\begin{proof} The continuity of the semigroup operation will follow as soon as for any $a,b\in X$ and a neighborhood $O_{ab}\in\Tau_\Phi\cup\{X\setminus{\Uparrow}e\}\cup\{\{x\}:x\in X\setminus I\!Z(X)\}$ of $ab$, we find open sets $O_a,O_b\in\Tau'_\Phi$ such that $a\in O_a$, $b\in O_b$ and $O_aO_b\subseteq O_{ab}$. If $O_{ab}\in\Tau_\Phi$, then such open sets exist by the continuity of the semigroup operation in the topology $\Tau_\Phi$. If $O_{ab}=X\setminus{\Uparrow}e$, then $ab\notin {\Uparrow}e$ and hence $a\notin{\Uparrow}e$ or $b\notin{\Uparrow}e$. 

If $a\notin{\Uparrow}e$, then put $O_a\defeq X\setminus{\Uparrow}e$ and $O_b\defeq X$. If $b\notin{\Uparrow}e$, then put $O_a\defeq X$ and $O_b\defeq X\setminus{\Uparrow}e$. In both cases we will have $a\in O_a\in\Tau'_\Phi$, $b\in O_b\in\Tau'_\Phi$ and $O_aO_b\subseteq X\setminus{\Uparrow}e=O_{ab}$.

If $O_{ab}=\{ab\}$ and $ab\notin I\!Z(X)$, then $a,b\notin I\!Z(X)$ because $I\!Z(X)$ is an ideal and then the $\Tau'_\Phi$-open sets $O_a=\{a\}$ and $O_b=\{b\}$ has the required property: $O_aO_b\subseteq O_{ab}$.
\end{proof}


We use the topology $\Tau'_{\Phi}$ to define the notions of convergent and divergent filters on $X$.

A filter $\F$ on $X$ will be called {\em convergent} if there exists a  point $\lim\F\in X$ such that every neighborhood $U\in\Tau'_{\Phi}$ of $\lim\F$ belongs to the filter $\F$. Since the space $(X,\Tau'_\Phi)$ is Hausdorff, the point $\lim\F$ is unique. It is clear that for every $x\in X$ the principal ultrafilter $\F_x\defeq\{F\subseteq X:x\in F\}$ converges to $x$. A filter $\F$ on $X$ is called {\em divergent} if it is not convergent.

Let $\varphi(X)$ be the semigroup of filters on $X$. For two filters $\F,\mathcal E\in\varphi(X)$ their product in $\varphi(X)$ is the filter $\F\mathcal E$ generated by the base $\{FE:F\in\F,\; E\in\mathcal E\}$. Since $A\in\U$ and $A\subseteq Z(X)$, the ultrafilter $\U$ is {\em central} in the sense that $x\U=\U x$ for all $x\in X$. 

Let $\U'$ be the ultrafilter generated by the base $\big\{\{a\pi(a):a\in U\}:U\in \U\big\}$.

\begin{lemma}\label{l:cdiv} \begin{enumerate}
\item The ultrafilter $\U'$ is convergent and $\lambda\defeq \lim\U'\in  I\!Z(X)\cap H_e$.
\item The filter $\U\U$ is convergent to $\lambda^2$.
\item For any $x,y\in X$ the filter $x\U y\U=xy\U\U$ is convergent to $xy\lambda^2$.
\item For every $x\in X$ with $H(X)\in x\U$, the ultrafilter $x\U$ is convergent.
\item For any $x\in X\setminus{\Uparrow}\mu$ the ultrafilter $x\U$ is principal.
\item If for some $x\in{\Uparrow}\mu$ the ultrafilter $x\U$ is convergent, then $\lim x\U=x\lambda\in I\!Z(X)\cap H_e$ and $x\U=x\U'$.
\item The ultrafilter $\U$ is divergent.
\item For two elements $x,y\in{\Uparrow}\mu$ the ultrafilters $x\U$ and $y\U$ are equal if and only if $\{a\in A:xa=ya\}\in \U$ if and only if $xA\in y\U$.
\end{enumerate}
\end{lemma}

\begin{proof} 1. Since $Z(X)$ is periodic and group-finite, the set $I\!Z\!H_e\defeq I\!Z(X)\cap H_e$ is a finite subgroup of $I\!Z(X)$. We claim that some point of  $I\!Z\!H_e$ is the limit of the ultrafilter $\U'$. Assuming the opposite, we conclude that for every $g\in I\!Z\!H_e$ there exists a set $V_g\in\Phi$ such that  $\Lambdae(g;V_g)\notin\U'$ and hence $\bigcup_{g\in I\!Z\!H_e}\Lambdae(g;V_g)\notin\U'$. Consider the semilattice $V=\bigcap_{g\in I\!Z\!H_e}V_g\in\Phi$ and find sets $U\in\U$ and $F\in[EZ(X)\setminus{\downarrow}e]^{<\w}$ such that $V=\nabla(U,F)$. By Claim~\ref{cl:AU-fin} and Lemma~\ref{l:pi2}, the set $A\cap{\Uparrow}F$ is finite. Since  $\bigcup_{g\in I\!Z\!H_e}\Lambdae(g;V)\notin\U'$ and $\{a\pi(a):a\in A\}\in\U'$, there exists $a\in A\cap U$ such that $a\notin{\Uparrow}F$ and $a\pi(a)\notin \bigcup_{g\in I\!Z\!H_e}\Lambdae(g;V_g)$. Then $\pi(a)\in\pi[A\cap U]\setminus{\Uparrow}F\subseteq \nabla(U,F)=V$. On the other hand, Lemma~\ref{l:Tamura} and Claim~\ref{cl3:e<mu} imply that $ae=a\pi(a)e\in I\!Z(X)\cap(H_{\pi(a)}H_e)\subseteq I\!Z(X)\cap H_e=I\!Z\!H_e$ and $a\pi(a)\in \frac{ae}{e}\cdot \pi(a)\in\frac{ae}{e}\cdot V\subseteq \Lambdae(ae;V)\subseteq\Lambdae(ae;V_{ae})$ which contradicts the choice of $a$.
\smallskip

2.  By Lemma~\ref{l:cdiv}(1), the ultrafilter $\U'$ converges to some point $\lambda\defeq\lim\U'\in H_e$. Since $(X,\Tau'_\Phi)$ is a topological semigroup, for every neighborhood $W\in\Tau'_{\Phi}$ of $\lambda^2$, there exists a  neighborhood $V\in\Tau'_{\Phi}$ of $\lambda$ such that $VV\subseteq W$. It follows that $V\in\U'$ and hence there exists a set $B\in\U$ such that $B\subseteq A$ and $\{\pi(a)a:a\in B\}\subseteq V$.  The choice of the set $A$ ensures that $BB\subseteq AA\subseteq H(X)$. By the periodicity of $Z(X)$, for every $a,b\in B\subseteq Z(X)$ we have $\pi(a)\in a^\IN\subseteq Z(X)$ and $\pi(b)\in b^\IN\subseteq Z(X)$. Since $\pi(a)$ and $\pi(b)$ are commuting idempotents, their product $\pi(a)\pi(b)$ is an idempotent. Observe that $$q(\pi(a)\pi(b))=q(\pi(a))q(\pi(b))=q(a)q(b)=q(ab)=q(\pi(ab))$$ and hence $\pi(a)\pi(b)=\pi(ab)$ by the injectivity of the restriction $q{\restriction}_{E(X)}$.
Then for every $a,b\in B$ we have $ab=ab\pi(ab)=ab\pi(a)\pi(b)=a\pi(a)b\pi(b)\in VV$. It follows from $BB\in\U\U$ and $BB\subseteq VV\subseteq W$ that $W\in\U\U$, which means that the filter $\U\U$ converges to $\lambda^2$.
\smallskip

3. By the preceding statement, the filter $\U\U$ converges to $\lambda^2$ in the topological space $(X,\Tau'_\Phi)$. Since the ultrafilter $\U$ is central, for every $x,y\in X$, the filter $x\U y\U=xy\U\U$ converges to $xy\lambda^2$ in the topological semigroup $(X,\Tau'_{\Phi})$.
\smallskip


4. Let $x\in X$ be such that $H(X)\in x\U$. Since $(X,\Tau'_\Phi)$ is a topological semigroup, for every neighborhood $W\in\Tau'_{\Phi}$ of $x\lambda$ there exists a neighborhood $W'\in\Tau'_\Phi$ of $\lambda$ such that $xW'\subseteq W$. Since $\lambda=\lim\U'\in W'$, there exists a set $U\in\U$ such that $\{\pi(a)a:a\in U\}\subseteq W'$ and hence $\{x\pi(a)a:a\in U\}\subseteq W$. Since $H(X)\in x\U$, we can replace $U$ by a smaller set in $\U$ and assume that $U\subseteq A$ and  $xU\subseteq H(X)$. Then for every $a\in U$ we have $xa\in H(X)$. Since $a\in A\subseteq Z(X)$, we can apply Lemma~\ref{l:pi2} and conclude that $\pi(xa)=\pi(xa)\pi(a)$. Then $xa=xa\pi(xa)=xa\pi(xa)\pi(a)=xa\pi(a)\in W$, which implies that $W\in x\U$ and means that the ultrafilter $x\U$ converges to $x\lambda$.
\smallskip

5. By Lemma~\ref{l:mu-fin}, for every $x\in X\setminus{\Uparrow}\mu$ the set $xA$ is finite. Since $A\in\U$, for some $g\in xA$ the set $A_g=\{a\in A:xa=g\}$ belongs to the ultrafilter $\U$. Then $\{g\}=xA_g\in x\U$ and the ultrafilter $x\U$ is principal.
\smallskip

6. Let $x\in{\Uparrow}\mu$ and assume that the ultrafilter $x\U$ is convergent to some element $y\in X$. For every $a\in A$ we have $e\le \pi(a)\le\mu\le \pi(x)$ and hence $\pi(xa)=\pi(a)\in{\Uparrow}e$, by Lemma~\ref{l:pi2}. For distinct elements $a,b\in A$, Claim~\ref{cl:antichain2} ensures that $\pi(xa)=\pi(a)\ne\pi(b)=\pi(xb)$ and hence $xa\ne xb$. This implies that the ultrafilter $x\U$ is free. If $y\ne ye$ or $y\notin I\!Z(X)$, then $y$ is an isolated point of $(X,\Tau'_\Phi)$ and hence $y\ne \lim x\U$. Therefore, $y=ye\in I\!Z(X)$. If $y\notin{\Uparrow}e$, then the neighborhood $X\setminus{\Uparrow}e\in\Tau'_\Phi$ of $y$ is disjoint with the set $xA\subseteq{\Uparrow}e$ and hence $y\ne\lim x\U$. Therefore, $y=ye\in{\Uparrow}e$ and $y=ye\in({\Uparrow}e)e\in H_e$ by Lemma~\ref{l:Tamura}.

To prove that $x\U=x\U'$, take any set $U'\in \U'$ and find a set $U\in\U$ such that $U\subseteq A$ and $\{a\pi(a):a\in U\}\subseteq U'$.
Since $\lim x\U=y\in H_e$, the neighborhood $\Lambdae(y;\nabla(U,\emptyset))\in\Tau_\Phi\subseteq\Tau'_\Phi$ of $y$ belongs to the ultrafilter $x\U$. Consequently, there exists a set $U''\in\U$ such that $U''\subseteq A$ and 
$xU''\subseteq \Lambdae(y;\nabla(U,\emptyset))$. Then for every $a\in U''$, there exists $y'\in\frac{y}e$ and $v\in\nabla(U,\emptyset)$ such that $xa=y'v$. It follows from $x\in{\Uparrow}\mu$ that $\pi(a)\le\mu\le\pi(x)$ and hence $\pi(a)=\pi(xa)=\pi(y'v)\le \pi(v)$. It follows from $\pi(y'v)=\pi(xa)=\pi(a)\notin J={\Updownarrow}J$ that $y'v\notin J$ and $v\notin J$. Then $v\in\nabla(U,\emptyset)\setminus J\subseteq \pi[U]$ and hence $v=\pi(a')$ for some $a'\in U\subseteq A$. It follows from $\pi(a)\le\pi(v)=\pi(a')$ that $a=a'$, see Claim~\ref{cl:antichain2}. Then $$xa=y'v=y'\pi(a')=y'\pi(a')\pi(a')=y'v\pi(a')=xa\pi(a')=xa'\pi(a')\in xU'$$ and hence $xU''\subseteq xU'$, which implies that $x\U'\subseteq x\U$ and $x\U'=x\U$ by the maximality of the ultrafilter $x\U'$. 

Since the ultrafilter $\U'$ converges to $\lambda\in I\!Z(X)\cap H_e$ and $x\lambda\in({\Uparrow}\mu)H_e\subseteq({\Uparrow}e)H_e\subseteq H_e$, the ultrafilter $x\U=x\U'$ converges to $x\lambda\in I\!Z(X)\cap H_e$.
\smallskip

7. Assuming that $\U$ converges and applying the preceding statement, we would conclude that $\U=\U'$, which is not true (as $H(X)\in\U'$ and $H(X)\notin\U$).
\smallskip

8. Given any elements $x,y\in{\Uparrow}\mu$, we should prove the equivalence of the following conditions:
\begin{itemize}
\item[{\rm(i)}] $x\U=y\U$;
\item[{\rm(ii)}] $xA\in y\U$;
\item[{\rm(iii)}] $\{a\in A:xa=ya\}\in \U$.
\end{itemize}

The implication (i)$\Ra$(ii) follows from $xA\in x\U=y\U$.

To prove that (ii)$\Ra$(iii), assume that $aA\in y\U$ and hence there exists a set $U\in\U$ such that $U\subseteq A$ and $yU\subseteq xA$. Then for every $u\in U$ there exists $a\in A$ such that $yu=xa$. It follows from $x,y\in{\Uparrow}\mu$  and $u,a\in{\Downarrow}\mu$ that $\pi(u)=\pi(yu)=\pi(xa)=\pi(a)$. By Claim~\ref{cl:antichain2}, $a=u$ and hence $U\subseteq \{z\in A:xz=yz\}$.

To prove that (iii)$\Ra$(i), assume  that the set $U=\{a\in A:xa=ya\}$ belongs to the ultrafilter $\U$. To derive a contradiction, assume that $x\U\ne y\U$. Then there exists a set $V\in\U$ such that $xV\cap yV=\emptyset$. Since $U\cap V\in\U$, there exists a point $a\in V\cap U$, for which $xa=ya\in xV\cap yV=\emptyset$, which is a contradiction completing the proof.
\end{proof}

Let $D$ be the set of elements $x\in X$ for which the ultrafilter $x\U$ diverges in the topological semigroup $(X,\Tau'_\Phi)$. By Lemma~\ref{l:cdiv}(6,7),
$$D=\{x\in{\Uparrow}\mu:x\U\ne x\U'\}\subseteq {\Uparrow}\mu.$$

\begin{lemma}\label{l:Ux} For every $x\in D$ there exists a set $U_x\in\U$ such that
$U_x\subseteq A$ and $xU_x\cap\Lambdae(g;V)=\emptyset$ for every $g\in I\!Z(X)\cap H_e$ and $V\in\Phi$.
\end{lemma}

\begin{proof} Since the filter $x\U$ is divergent in the topological semigroup $(X,\Tau'_\Phi)$, for every $g\in I\!Z\!H_e\defeq I\!Z(X)\cap H_e$ there exist a set $V_g\in\Phi$ such that $\Lambdae(g;V_g)\notin x\U$. Since the group $I\!Z\!H_e$ is finite, there exists an element $V_x\in\Phi$ such that $V_x\subseteq\bigcap_{g\in I\!Z\!H_e}V_g$. By the definition of the $e$-base $\Phi$, for the set $V_x\in\Phi$, there exist sets $U_x\in\U$ and $F_x\in[EZ(X)\setminus{\downarrow}e]^{<\w}$ such that $\nabla(U_x,F_x)\subseteq V_x$.  By Claim~\ref{cl:AU-fin}, the set $A\cap {\Uparrow} F_x$ is finite. Replacing the set $U_x$ by $U_x\cap A\setminus{\Uparrow}F_x$, we can assume that $U_x\subseteq A\setminus{\Uparrow}F_x$. Since $\Lambdae(I\!Z\!H_e;V_x)\defeq\bigcup_{g\in I\!Z\!H_e}\Lambdae(g;V_x)\notin x\U$, we can replace $U_x$ by a smaller set in $\U$ and assume that $xU_x\cap\Lambdae(I\!Z\!H_e;V_x)=\emptyset$.

It remains to check that $xU_x\cap \Lambdae(g;V)=\emptyset$ for any $g\in I\!Z\!H_e$ and $V\in\Phi$. To derive a contradiction, assume that for some $a\in U_x$, $g\in I\!Z\!H_e$ and $V\in\Phi$ we have $xa\in \Lambdae(g;V)$. It follows from $x\in D\subseteq{\Uparrow}\mu$ and $A\subseteq({\Downarrow}\mu\cup{\Uparrow}e)\setminus{\Updownarrow}e$ that $e<\pi(a)\le \mu\le\pi(x)$ and hence
$\pi(xa)=\pi(a)\notin {\Downarrow}e$ and $xa\ne g$. Then $xa=yv$ for some $y\in\frac{g}e$ and  $v\in V\subseteq \nabla(A,\emptyset)$. It follows from $\pi(xa)=\pi(a)\notin J={\Updownarrow}J$ that $yv=xa\notin J$ and hence $v\in\nabla(A,\emptyset)\setminus J=\pi[A]$. Then $v=\pi(a_v)$ for some $a_v$. Since $\pi(a)=\pi(xa)=\pi(yv)\le\pi(v)=v=\pi(a_v)$,  Claim~\ref{cl:antichain2} implies the equality $a=a_v$. The choice of $U_x\subseteq A\setminus{\Uparrow}F_x$ ensures that $a_v=a\in U_x\subseteq A\setminus {\Uparrow}F_x$ and hence $v=\pi(a_v)\in\nabla(U_x,F_x)\subseteq V_x$ and $xa=yv\in \frac ge\cdot V_x\subseteq \Lambdae(g;V_x)$. On the other hand, $a\in U_x$ and hence $xa=yv\in xU_x\cap\Lambdae(g;V_x)=\emptyset$, by the choice of $U_x$. This contradiction completes the proof of the lemma.
\end{proof}

In the semigroup of filters $\varphi(X)$, consider the subset
$$Y\defeq X\cup\{x\U:x\in D\}.$$
Here we identify $X$ with the subset of $\varphi(X)$ consisting of principal ultrafilters.

Lemma~\ref{l:cdiv}(3) implies that the binary operation $*:Y\times Y \to Y$ assigning to any filters $\F,\mathcal E\in Y$ the filter $$
\F*\mathcal E\defeq\begin{cases}
\lim \F\mathcal E&\mbox{if $\F\mathcal E$ is convergent};\\
\F\mathcal E&\mbox{if $\F\mathcal E$ is divergent};
\end{cases}
$$
is well-defined.

\begin{definition}\label{d3:Tau}
Let $\Tau$ be the topology on $Y$ consisting of the sets $W\subseteq Y$ satisfying two conditions:
\begin{enumerate}
\item $X\cap W\in\Tau'_\Phi$;
\item for any $x\in D$ with $x\U\in W$ there exists a set $U\in\U$ such that $xU\subseteq W$.
\end{enumerate}
\end{definition}

The definition of the topology $\Tau$ implies that  $(X,\Tau'_\Phi)$ is an open dense nonclosed subspace in $(Y,\Tau)$.

The following four claims describe neighborhood bases of the topology $\Tau$ at points of the set $Y$.

\begin{claim}\label{cl3:isolated} For every $x\in X$ with $x\ne xe$ or $x\notin I\!Z(X)$ the singleton $\{x\}$ is $\Tau$-clopen.
\end{claim}

\begin{proof} If $x\ne xe$, then $\{x\}=\Lambdae(x,\frac ee)$ belongs to the topology $\Tau_\Phi$ by Lemmas~\ref{l:Lamb}(3) and \ref{l:base-e}. Since $\Tau_\Phi\subseteq\Tau_\Phi'\subseteq\Tau$, the singleton $\{x\}$ is $\Tau$-open. If $x\notin I\!Z(X)$, then $\{x\}\in\Tau'_\Phi\subseteq\Tau$ by the definition of the topologies $\Tau'_\Phi$ and $\Tau$. To see that $\{x\}$ is $\Tau$-closed, consider the set $W=Y\setminus\{x\}$. Since the topology $\Tau'_\Phi$ is Hausdorff, $X\cap W=X\setminus\{x\}\in\Tau'_\Phi$. On the other hand, for every $d\in D$ the divergent ultrafilter $d\U$ is free and hence $X\setminus\{x\}\in d\U$ and there is a set $U\in\U$ such that $dU\subseteq X\setminus\{x\}\subseteq W$. By Definition~\ref{d3:Tau}, the set $W=X\setminus\{x\}$ is $\Tau$-open and hence $\{x\}$ is $\Tau$-clopen.
\end{proof} 

For every $x\in D$ and $U\in\U$, consider the set
$$\Lambda(x\U;U)\defeq\{x\U\}\cup xU.$$
By Lemma~\ref{l:Ux}, there exists a set $U_x\in\U$ such that $U_x\subseteq A$ and $xU_x\cap\bigcup_{g\in I\!Z(X)\cap H_e}\Lambdae(g;\nabla(A,\emptyset))=\emptyset$.

\begin{claim}\label{cl3:ultra} For every $x\in D$ the family
$$\mathcal B_{x\U}\defeq\{\Lambda(x\U;U):U\in\U,\; U\subseteq U_x\}$$consists of $\Tau$-clopen sets and is a neighborhood base of the topology $\Tau$ at $x\U$.
\end{claim}

\begin{proof} Since $x\in D\subseteq {\Uparrow}\mu$ and $A\subseteq{\Downarrow}\mu\setminus{\Downarrow}e$, for every $U\in\U$ with $U\subseteq U_x$, the set $$X\cap\Lambda(x\U;U)\subseteq xU_x\subseteq xA\subseteq {\Downarrow}\mu\setminus{\Downarrow}e$$ consists of isolated points of  the space $(X,\Tau'_\Phi)$ and hence is $\Tau'_\Phi$-open. Now fix any $d\in D$ with $d\U\in\Lambda(x\U;U)$. Then $d\U=x\U$ and there exists $U'\in\U$ such that $dU'\subseteq x(U\cap U_x)\subseteq \Lambda(x\U;U)$. By Definition~\ref{d3:Tau}, the set $\Lambda(x\U;U)$ is $\Tau$-open.

To show that it is $\Tau$-closed, fix any $b\in Y\setminus\Lambda(x\U;U)$. 
If $b\notin X$, then $b=y\U$ for some $y\in D$. By Lemma~\ref{l:cdiv}(8), $x\U\ne y\U$ implies $xU\subseteq xU_x\subseteq xA\notin y\U$ and thus there exists a set $V\in\U$ such that $xU\cap yV=\emptyset$. Then $\Lambda(y\U;V)$ is a $\Tau$-open neigborhood of $b=y\U$, disjoint with the set $\Lambda(x\U;U)$.

Now assume that $b\in X$. If $be\ne b$ or $b\notin I\!Z(X)$, then by Claim~\ref{cl3:isolated}, $\{b\}$ is $\Tau$-open neighborhood of $b$, disjoint with the set $\Lambda(x\U;U)$. If $b=be\notin{\Uparrow}e$, then the set $X\setminus{\Uparrow}e\in\Tau'_\Phi\subseteq\Tau$ is a $\Tau$-open neighborhood of $b$, which is disjoint with the set $\Lambda(x\U;U)\subseteq{\Uparrow}e$.

 It remains to analyze the case of $b=be\in{\Uparrow}e\cap I\!Z(X)$. By Lemma~\ref{l:Tamura}, the group $H_e$ is an ideal in ${\Uparrow}e$. Then $b=be\in I\!Z(X)\cap H_e$. By the choice of the set $U_x\in\Phi$, the set $\Lambdae(b;\nabla(A,\emptyset))\in\Tau_\Phi\subseteq\Tau'_\Phi\subseteq\Tau$ is a $\Tau$-open neighborhood of $b$, which is disjoint with the set $\Lambda(x\U;U)$.

By the definition of the topology $\Tau$, the family $\mathcal B_{x\U}$ is a neighborhood base of $\Tau$ at $x\U$.
\end{proof}

\begin{claim}\label{cl3:Tau-base-e} For every $x\in I\!Z(X)\cap H_e$ the family $$\mathcal B_x\defeq\{\Lambdae(x;V):V\in\Phi\}$$consists of $\Tau$-clopen sets and is a neighborhood base of the topology $\Tau$ at $x$.
\end{claim}

\begin{proof} By Lemma~\ref{l:base-e}, the family $\mathcal B_x$ consists of $\Tau_\Phi$-open sets and is a neighborhood base of the topology $\Tau_\Phi$ at $x$. By the definition of the topologies $\Tau'_\Phi$ and $\Tau$, the family $\mathcal B_x$ remains a neighborhood base at $x$ in the topologies $\Tau'_\Phi\subseteq\Tau$.
By Theorem~\ref{t:TS}(5a,3), the family $\mathcal B_x$ consists of $\Tau_\Phi$-clopen sets. Since $\Tau_\Phi\subseteq\Tau_\Phi'$, the sets in the family $\mathcal B_x$ remain $\Tau_\Phi'$-clopen. To see that they are $\Tau$-clopen, take any set $\Lambdae(x;V)\in\mathcal B_x$. Given any point $y\in Y\setminus\Lambdae(x;V)$, we should find a neigbhorhood $O_y\in\Tau$ of $y$ such that $O_y\cap\Lambdae(x;V)=\emptyset$. If $y\in X$, then such neighborhood exists by the $\Tau_\Phi$-closedness of the set $\Lambdae(x;V)$. If $y\notin X$, then $y=d\U$ for some $d\in D$ and then the neighborhood $O_y\defeq\Lambda(d\U;U_d)$ is disjoint with $\Lambdae(x;V)$ by Lemma~\ref{l:Ux}.
\end{proof}

\begin{claim}\label{cl3:Tau-base} For every $x=xe\in I\!Z(X)\setminus{\Uparrow}e$ the family $$\mathcal B_x\defeq\{\Lambdae(x;V)\setminus{\Uparrow}e:V\in\Phi\}$$consists of $\Tau$-clopen sets and is a neighborhood base of the topology $\Tau$ at $x$.
\end{claim}

\begin{proof} By Lemma~\ref{l:base-e} and the definition of the topology $\Tau'_\Phi$, the family $\mathcal B_x$ consists of $\Tau'_\Phi$-open sets and is a neighborhood base of the topology $\Tau'_\Phi$ at $x$. By the definition of the topology $\Tau$, the family $\mathcal B_x$ remains a neighborhood base at $x$ in the topology $\Tau$.
By Theorem~\ref{t:TS}(5a,3), the family $\mathcal B_x$ consists of $\Tau_\Phi$-closed sets. Since $\Tau_\Phi\subseteq\Tau_\Phi'$, the sets in the family $\mathcal B_x$ are  $\Tau_\Phi'$-clopen. To see that they are $\Tau$-clopen, take any set $\Lambdae(x;V)\setminus{\Uparrow}e\in\mathcal B_x$. Given any point $y\in Y\setminus(\Lambdae(x;V)\setminus{\Uparrow}e)$, we should find a neighborhood $O_y\in\Tau$ of $y$ such that $O_y\cap(\Lambdae(x,V)\setminus{\Uparrow}e)$. If $y\in X$, then such neighborhood exists by the $\Tau'_\Phi$-closedness of the set $\Lambdae(x;V)\setminus{\Uparrow}e$. If $y\notin X$, then $y=d\U$ for some $d\in D$ and then the neighborhood $O_y\defeq\Lambda(d\U;U_d)$ is disjoint with $\Lambdae(x;V)\setminus{\Uparrow}e\subseteq\Lambdae(x;V)$ by Lemma~\ref{l:Ux}.
\end{proof}

Claims~\ref{cl3:isolated}--\ref{cl3:Tau-base} imply the following 

\begin{lemma}\label{l3:zero-dim} The topological space $(Y,\Tau)$ is Hausdorff and zero-dimensional.
\end{lemma}

\begin{lemma}\label{l:conti2} The binary operation $*:Y\times Y\to Y$ is continuous.
\end{lemma}

\begin{proof} The continuity of the semigroup operation of $X$ with respect to the topology $\Tau'_\Phi$ implies that the binary operation $*$ is continuous on the open subspace $X\times X$ of $Y\times Y$. It remains to check the continuity of the operation $*$ at each pair $(a,b)\in(Y\times Y)\setminus(X\times X)$. Given any neighborhood $O_{ab}\in\Tau$ of $a*b$, we should find $\Tau$-open neighborhoods $O_a,O_b$ of $a,b$ such that  $O_a*O_b\subseteq O_{ab}$. Since the space $(Y,\Tau)$ is zero-dimensional, we lose no generality assuming that the set $O_{ab}$ is $\Tau$-clopen.

Depending on the location of the points $a,b$, we consider three cases.
\smallskip

1. First assume that $a$ and $b$ do not belong to $X$. Then $a=x\U$ and $b=y\U$ for some $x,y\in D$. By Lemma~\ref{l:cdiv}, the filter $x\U y\U=xy\U\U$ converges to $xy\lambda^2=a*b\in O_{ab}$. Since $O_{ab}\cap X$ is a neighborhood of $xy\lambda^2$ in the topological space $(X,\Tau'_\Phi)$, there exists a set $U\in\U$ such that $xUyU\subseteq O_{ab}$ and $U\subseteq A$.  Consider the $\Tau$-open sets $O_a\defeq\Lambda(x\U;U)$ and $O_b\defeq\Lambda(y\U;U)$. We claim that $O_a*O_b\subseteq O_{ab}$. Take any elements $a'\in O_a$ and $b'\in O_b$.
\smallskip

1.1. If $a',b'\notin X$, then $a'=x\U=a$ and $b'=y\U=a$ and hence $a'*b'=a*b\in O_{ab}$.
\smallskip

1.2. If $a',b'\in X$, then $a'*b'=a'b'\in xUyU\subseteq O_{ab}$ and we are done.
\smallskip

1.3. Next, assume that $a'\in X$ and $b'\notin X$. Then $b'=b=y\U$. Since $a'\in X\cap \Lambda(x\U;U)=xU$, there exists $u\in U\subseteq A$ such that $a'=xu$. Lemma~\ref{l:pi2} and Claim~\ref{cl3:e<mu} imply  $\pi(a'y)\le\pi(a')\le\pi(u)<\mu$ and hence $a'y\notin{\Uparrow}\mu$. By Lemma~\ref{l:cdiv}(5), the ultrafilter $a'y\U$ is principal. Then  $a'*b'=xu*y\U\in xuyU\subseteq xUyU\subseteq  O_{ab}$.
\smallskip

1.4. The case $a'\notin X$ and $b'\in X$ can be considered by analogy with case 1.3.
\smallskip

2. Next, assume that $a\in X$ and $b\notin X$. Then $b=y\U$ for some $y\in D$. Now consider two cases.
\smallskip

2.1. First assume that $ay\notin{\Uparrow}\mu$. By Lemma~\ref{l:cdiv}(5), $ay\U$ is a principal ultrafilter and hence $ayU'=\{ay\U\}\subseteq X$ for some $U'\in\U$. If $a\ne ae$ or $a\notin I\!Z(X)$, then $O_a\defeq\{a\}\in\Tau$ by Claim~\ref{cl3:isolated}. By Claim~\ref{cl3:ultra}, $O_b\defeq\Lambda(y\U;U')\in \Tau$. Then $O_a*O_b=a\Lambda(y\U;U')=\{ay\U\}\cup ayU'=\{ay\U\}=\{a*b\}\subseteq O_{ab}$. So, assume that $a=ae\in I\!Z(X)$.

Choose any $u'\in U'\cap A$ and observe that $ayU'=\{ayu'\}=\{ay\U\}$. It follows  from $a=ae\in I\! Z(X)$ and $e\in EZ(X)$ that $ayu'e=ayu'\in I\! Z(X)$. 
\smallskip

2.1.1. If $ayu'\in{\Uparrow}e$, then $ayu'=ayu'e\in({\Uparrow}e)e\subseteq H_e$ by Lemma~\ref{l:Tamura}. By Claim~\ref{cl3:Tau-base-e}, there exists $V\in\Phi$ such that $\Lambdae(ayu';V)\subseteq O_{ab}$.  Find sets $U\in\U$ and $F\in[EZ(X)\setminus{\downarrow}e]^{<\w}$ such that $U\subseteq U'\cap A$ and $\nabla(U,F)\subseteq V$. By Claims~\ref{cl3:Tau-base-e} and \ref{cl3:Tau-base}, $O_a\defeq \Lambdae(a;\nabla(U,F))$ is a neighborhood of $a$ in the topological space $(Y,\Tau)$.

We claim that for any $a'\in O_a$ we have $a'yU\subseteq O_{ab}$.  If $a'=a$, then $a'yU=ayU=\{ayu'\}\subseteq\Lambdae(ayu';V)\subseteq O_{ab}$. So, we assume that $a'\ne a$ and hence $a'=x'v$ for some $x'\in\frac{a}e$ and $v\in \nabla(U,F)\subseteq Z(X)$. Then for every $u\in U$ we have $a'yu=x'vyu=x'yuv$. Observe that $x'yue=x'eyu=ayu=ayu'$ and hence $a'yu=x'yuv\in \frac{ayu'}e{\cdot}\nabla(U,F)\subseteq\Lambdae(ayu';\nabla(U,F))\subseteq \Lambdae(ayu';V)\subseteq O_{ab}$.

By the definition of the topology $\Tau$, the ultrafilter $a'y\U$ belongs to the closure of the set $a'yU\subseteq O_{ab}$ in the topological space $(Y,\Tau)$. Since the set $O_{ab}$ is $\Tau$-closed, $a'y\U\subseteq O_{ab}$. Then for the neighborhood $O_b\defeq\Lambda(y\U;U)$ we have $O_a*O_b=\bigcup_{a'\in O_a}\big(\{a'y\U\}\cup a'yU\big)\subseteq O_{ab}$. 
\smallskip

2.1.2. If $ayu'\notin{\Uparrow}e$, then by Claim~\ref{cl3:Tau-base}, there exists $V\in\Phi$ such that $\Lambdae(ayu;V)\setminus {\Uparrow}e\subseteq O_{ab}$. Find sets $U\in\U$ and $F\in[EZ(X)\setminus{\downarrow}e]^{<\w}$ such that $U\subseteq U'\cap A$ and $\nabla(U,F)\subseteq V$. It follows from $y\in D\subseteq{\Uparrow}\mu\subseteq{\Uparrow}e$, $u'\in U'\cap A\subseteq{\Uparrow}e$, and $ayu'\notin{\Uparrow}e$ that $a\notin{\Uparrow}e$. By Claim~\ref{cl3:Tau-base}, $O_a\defeq \Lambdae(a;\nabla(U,F))\setminus{\Uparrow}e$ is a neighborhood of $a$ in the topological space $(Y,\Tau)$. By Claim~\ref{cl3:ultra}, $O_b\defeq \Lambda(y\U;U)$ is a neighborhood of $b=y\U$ in $(Y,\Tau)$.

By analogy with the case 2.1.1 we can show that $O_a*O_b\subseteq O_{ab}$.
\smallskip

2.2. Next, assume that $ay\in{\Uparrow}\mu$. Then $a\in{\Uparrow}\mu$ and $a\ne ae$ by Claim~\ref{cl3:e<mu}. By Claim~\ref{cl3:isolated}, $O_a\defeq\{a\}\in\Tau$. 
\smallskip

2.2.1. If $ay\in D$, then $a*b=ay\U$ and by Claim~\ref{cl3:ultra}, there exists $U\in\U$ such that $\Lambda(ay\U;U)\subseteq O_{ab}$. Then for the neighborhood $O_b\defeq\Lambda(y\U;U)$ we have $O_a*O_b=a\Lambda(y\U;U)=\{ay\U\}\cup ayU=\Lambda(ay\U;U)\subseteq O_{ab}$. 
\smallskip

2.2.2. If $ay\notin D$, then the ultraflter $ay\U$ is convergent. By Lemma~\ref{l:cdiv}(6), $ay\U=ay\U'$ and $a*b=\lim ay\U=ay\lambda$. By Lemma~\ref{l:cdiv}(1), $\lambda\in I\!Z(X)\cap H_e$. It follows from $ay\in{\Uparrow}\mu\subseteq{\Uparrow}e$ and $\lambda\in I\!Z(X)\cap H_e$ that $ay\lambda\in I\!Z(X)\cap H_e$, see Lemma~\ref{l:Tamura}.

By Claim~\ref{cl3:Tau-base-e}, there exist a set $V\in\Phi$ such that  and  $\Lambdae(ay\lambda;V)\subseteq O_{ab}$. Since the ultrafilter $ay\U$ converges to $ay\lambda\in \Lambdae(ay\lambda;V)\in\Tau'_\Phi$, there exists $U\in \U$ such that $ayU\subseteq \Lambdae(ay\lambda;V)$.
Then for the neighborhood $O_b\defeq\Lambda(y\U;U)$ we have $O_a*O_b=a*\Lambda(y\U;U)=\{a*y\U\}\cup ayU=\{ay\lambda\}\cup ayU\subseteq\Lambdae(ay\lambda;V)\subseteq O_{ab}$. 
\smallskip

3. The case $a\notin X$ and $b\in X$ can be considered by analogy with the case 2.
\end{proof}

\begin{lemma}\label{l:final-iC} $(Y,\Tau)$ is a Hausdorff zero-dimensional topological semigroup containing $X$ as a nonclosed dense subsemigroup.
\end{lemma}

\begin{proof} By Lemma~\ref{l3:zero-dim}, the space $(Y,\Tau)$ is Hausdorff and zero-dimensional and by Lemma~\ref{l:conti2}, the binary operation $*:Y\times Y \to Y$ is continuous. The density of $X$ in the Hausdorff topological space $(Y,\Tau)$ and the associativity of the semigroup operation $*{\restriction}_{X\times X}$ of $X$ implies that the binary operation $*$ is associative and hence $(Y,\Tau)$ is a topological semigroup containing $X$ as a nonclosed dense subsemigroup.
\end{proof}

Lemma~\ref{l:final-iC} implies that the semigroup $X$ is not injectively $\mathsf{T_{\!z}S}$-closed, which contradicts our assumption and completes the proof of Proposition~\ref{p:vsim-grobikam-grobik2}.
\end{proof}

\begin{corollary} Let $X$ be an injectively $\mathsf{T_{\!z}S}$-closed eventually Clifford nonsingular viable semigroup. If the poset $E(X)$ is well-founded and all maximal subgroups of $X$ are projectively $Z$-bounded, then the ideal center $I\!Z(X)$ of $X$ is injectively $\mathsf{T_{\!2}S}$-closed.
\end{corollary}

\begin{proof} By Theorem~\ref{t:Z}, the center $Z(X)$ of $X$ is chain-finite, group-finite, periodic, and nonsingular, and so is the ideal center $I\!Z(X)$ of $X$. By Proposition~\ref{p:vsim-grobikam-grobik2}, the semigroup $I\!Z(X)$ is not Clifford-singular. By Propositions~\ref{p:vsim-grobikam-grobik} and \ref{p2:vsim-grobikam-grobik}, the semigroup $Z(X)$ is bounded and so is the semigroup $I\!Z(X)$. Therefore the semigroup $I\!Z(X)$ is chain-finite, group-finite, bounded, nonsingular and not Clifford-singular. By Proposition~\ref{p:if}, $I\!Z(X)$ is injectively $\mathsf{T_{\!2}S}$-closed.
\end{proof}

\section{Proof of Theorem~\ref{t:main-iC}}\label{s:pf}

Given any commutative semigroup $X$, we should prove the equivalence of the following conditions:
\begin{enumerate}
\item $X$ is injectively $\mathsf{T_{\!2}S}$-closed;
\item $X$ is injectively $\mathsf{T_{\!z}S}$-closed;
\item $X$ is chain-finite, group-finite, bounded, nonsingular and not Clifford-singular.
\end{enumerate}

The implication $(1)\Ra(2)$ is trivial.
\smallskip

To prove that $(2)\Ra(3)$, assume that $X$ is injectively $\mathsf{T_{\!z}S}$-closed. By Theorem~\ref{t:Z}, the semigroup $Z(X)=X$ is chain-finite, periodic, nonsingular and group-finite. Being commutative and periodic, the semigroup $X$ is viable and eventually Clifford. Since the semigroup $X$ is chain-finite, the poset $E(X)$ is well-founded.
By Corollary~\ref{c:bounded}, the semigroup $Z(X)=X$ is bounded. By Proposition~\ref{p:vsim-grobikam-grobik2}, the semigroup $IZ(X)=X$ is not Clifford-singular.
\smallskip

The implication $(3)\Ra(1)$ follows from Proposition~\ref{p:if}.

\section{Properties of the semigroup from Example~\ref{ex:main}}\label{s:ex}

In this section we establish the properties (1)--(6) of the semigroup $X$ from Example~\ref{ex:main}. We recall that $X$ is the set $\w$ endowed with the binary operation
$$x*y=\begin{cases}
1&\mbox{if $x,y\in \{2n+3:n\in\w\}$ and $x\ne y$};\\
2n&\mbox{if $\{x,y\}\subseteq\{2n,2n+1\}$ for some $n\in\w$};\\
0&\mbox{otherwise}.
\end{cases}
$$
It is easy to check that the operation $*$ is commutative and associative, so $X$ indeed is a commuative semigroup.
\smallskip

1. By Theorem~\ref{t:main-iC}, the injective $\mathsf{T_{\!2}S}$-closedness of $X$ will follow as soon as we check that $X$ is chain-finite, group-finite, bounded, nonsingular and not Clifford-singular.

Observe that the set of idempotents $E(X)$ of $X$ coincides with the set $\{2n:n\in\w\}$ of even numbers. Each chain in $E(X)$ is countained in a doubleton $\{0,2n\}$, which implies that the semigroup $X$ is chain-finite. 

Each subgroup of $X$ is trivial, so $X$ is group-finite.

For every $x\in X$ we have $x^2\in E(X)$, which means that $X$ is bounded.

To show that $X$ is not singular, take any infinite set $A\subseteq X$. If $A$ contains two distinct even numbers $2n,2m$, then $\{2n,2m\}=\{2n*2n,2m*2m\}\subseteq AA$ and hence $AA$ is not a singeton. So, we assume that $A$ contains at most one even number. Then we can find two distinct numbers $n,m\ge 1$ such that $2n+1,2m+1\in A$ and conclude that $\{2n,2m\}=\{(2n+1)*(2n+1),(2m+1)*(2m+1)\}\subseteq AA$, so $AA$ is not a singleton and hence $X$ is nonsingular.

To see that $X$ is not Clifford-singular, take any infinite set $A\subseteq X\setminus H(X)=\{2n+1:n\in\w\}$ and find two distinct numbers $n,m\ge 1$ such that $2n+1,2m+1\in A$. Since $(2n+1)*(2m+1)=1\notin H(X)$, the set $AA$ is not contained in $H(X)$. So, $X$ is not Clifford-singular.

By Theorem~\ref{t:main-iC}, the semigroup $X$ is injectivey $\mathsf{T_{\!2}S}$-closed.
\smallskip

2. Being injectively $\mathsf{T_{\!2}S}$-closed, the semigroup $\mathsf{T_{\!2}S}$-closed and hence $\mathsf{T_{\!1}S}$-closed by Theorem~\ref{t:C}. 
\smallskip

3. Since the Clifford part $H(X)=E(X)=\{2n:n\in\w\}$ of the semigroup $X$ is infinite, $X$ is not injectively $\mathsf{T_{\!1}S}$-closed by Theorem~\ref{t:iT1}.
\smallskip

4. The definition of the semigroup operation $*$ ensures that $I\defeq\{0,1\}$ is an ideal in $X$.
\smallskip

5. Repeating the argument from the item 1, one can show that the quotient semigroup $Y=X/I$ is chain-finite, nonsingular and bounded. By Theorem~\ref{t:C}, $Y$ is $\mathsf{T_{\!1}S}$-closed. On the other hand, for the infinite set $A=\{2n+3:n\in\w\}=Y\setminus H(Y)$ we have $AA\subseteq H(Y)$, which means that $Y$ is Clifford-singular. By Theorem~\ref{t:main-iC}, $Y=X/I$ is not injectively $\mathsf{T_{\!z}S}$-closed.
\smallskip

6. It is easy to see that $J=I\cup E(X)=\{1\}\cup\{2n:n\in\w\}$ is an ideal in $X$ and for the infinite set $A=X\setminus J$ in the quotient semigroup $X/J$ we have $AA=\{J\}$, which means that the semigroup $X/J$ is singular and hence $X/J$ is not $\mathsf{T_{\!z}S}$-closed by Theorem~\ref{t:C}.

\section{Acknowledgement}

The author express his sincere thanks to Serhii Bardyla for many fruitful discussions on possible characterizations of injectively $\C$-closed semigroups, which eventually have led the author to finding the characterization Theorems~\ref{t:iT1} and \ref{t:main-iC}.


\begin{thebibliography}{}

\bibitem{AC} A. Arhangelskii, M. Choban, {\em Semitopological groups and the theorems of Montgomery and Ellis},
 C. R. Acad. Bulg. Sci. {\bf 62}(8) (2009), 917--922.

\bibitem{AC1}
A. Arhangelskii, M. Choban, {\em Completeness type properties of semitopological groups, and the theorems of Montgomery and Ellis}, Topology Proceedings {\bf 37} (2011), 33--60.

\bibitem{Ban} T.~Banakh, {\em Categorically closed topological groups}, Axioms {\bf 6}:3 (2017), 23.

\bibitem{ICT1S}  T.~Banakh, {\em Injectively and absolutely $\mathsf{T_{\!1}S}$-closed semigroups}, preprint ({\tt arxiv.org/abs/2208.00074}).

\bibitem{BanE} T.~Banakh, {\em $E$-separated semigroups}, preprint ({\tt arxiv.org/abs/2202.06298}).

\bibitem{BBm} T.~Banakh, S.~Bardyla,
\emph{Characterizing chain-finite and chain-compact topological semilattices}, Semigroup Forum {\bf 98} (2019), no. 2, 234--250.

\bibitem{BBh} T.~Banakh, S.~Bardyla, {\em
Completeness and absolute $H$-closedness of topological
semilattices}, Topology Appl.  {\bf 260} (2019), 189--202.

\bibitem{BB} T.~Banakh, S.~Bardyla, {\em Characterizing categorically closed commutative semigroups}, Journal of Algebra {\bf 591} (2022), 84--110.

\bibitem{CCCS} T.~Banakh, S.~Bardyla, {\em Categorically closed countable semigroups}, preprint, ({\tt arxiv.org/abs/2111.14154}).

\bibitem{ACVS}  T.~Banakh, S.~Bardyla, {\em Categorically closed viable semigroups}, preprint ({\tt arxiv.org/abs/2207.12778}).




\bibitem{BBc} T.~Banakh, S.~Bardyla, {\em Complete topologized posets and semilattices}, Topology Proc. {\bf 57} (2021), 177--196.



\bibitem{BBR} T.~Banakh, S.~Bardyla, A.~Ravsky, {\em  The closedness of complete subsemilattices in functionally Hausdorff semitopological semilattices}, Topology Appl. {\bf 267} (2019), 106874.

\bibitem{BGP} T.~Banakh, I.~Guran, I.~Protasov, {\em Algebraically determined topologies on permutation groups}, Topology Appl. {\bf 159}:9 (2012) 2258--2268.

\bibitem{BH} T.~Banakh, O.~Hryniv, {\em The binary quasiorder on semigroups}, Visnyk Lviv Univ. Series Mech. Math. {\bf 91} (2021), 28--39.

\bibitem{BM} T.~Banakh, H.~Mildenberger, {\em Cardinal invariants distinguishing permutation groups}, Europ. J. Math. {\bf 2}:2 (2016) 493--507.








\bibitem{BPS} T.~Banakh, I.~Protasov, O.~Sipacheva, {\em Topologization of sets endowed with an action of a monoid}, Topology Appl. {\bf 169} (2014) 161–174.

\bibitem{CCUS} T.~Banakh, M.~Vovk, {\em Categorically closed unipotent semigroups}, preprint ({\tt arxiv.org/abs/2208.00072}).




\bibitem{DS2}
D. Dikranjan, D. Shakhmatov, {\em The Markov--Zariski topology of an abelian group}, Journal of Algebra {\bf 324} (2010), 1125--1158.

\bibitem{DU}
D. Dikranjan, V. Uspenskij, {\em Categorically compact topological groups}, J. Pure Appl. Algebra {\bf 126} (1998), 149--168.

\bibitem{vD} E.K.~van Douwen, {\em The maximal totally bounded group topology on $G$ and the biggest minimal $G$-space, for abelian groups $G$}, Topology Appl. {\bf 34}:1 (1990),  69--91.

\bibitem{DI} J.~Dutka, A.~Ivanov, {\em Topologizable structures and Zariski topology}, Algebra Univers. {\bf 79} (2018), Paper No.72.



\bibitem{GLS}
G. Gierz, J. Lawson and A. Stralka, {\em Intrinsic topologies on semilattices of finite breadth}, Semigroup Forum {\bf 31} (1985), 1--17.

\bibitem{G}
M. Goto, {\em Absolutely closed Lie groups}, Math. Ann. {\bf 204} (1973), 337--341.


\bibitem{Ramsey} R.~Graham, B.~Rothschild, J.~Spencer, {\em Ramsey theory},  John Wiley \& Sons, Inc., Hoboken, NJ, 1990.


\bibitem{Gutik} O.~Gutik, {\em Topological properties of Taimanov semigroups}, Math. Bull. Shevchenko Sci. Soc. {\bf 13} (2016) 29--34.

\bibitem{GutikPagonRepovs2010}
O.~Gutik, D.~Pagon, D.~Repov\v{s}, \emph{On chains in H-closed topological pospaces}, Order \textbf{27}:1 (2010), 69--81.



\bibitem{GutikRepovs2008}
O.~Gutik, D.~Repov\v{s}, \textit{On linearly ordered $H$-closed
topological semilattices}, Semigroup Forum \textbf{77}:3 (2008),
474--481.


\bibitem{Hesse} G. Hesse, {\em Zur Topologisierbarkeit von Gruppen}, Dissertation, Univ. Hannover, Hannover, 1979.

\bibitem{HS} N.~Hindman, D.~Strauss, {\em Algebra in the Stone-\v Cech compactification. Theory and applications}, De Gruyter Textbook. Walter de Gruyter \& Co., Berlin, 1998. 

\bibitem{Howie} J.~Howie, {\em Fundamentals of Semigroup Theory}, Clarendon Press, Oxford, 1995.


\bibitem{Z1} V. Keyantuo, Y. Zelenyuk,  {\em Semigroup completions of locally compact Abelian groups},
Topology Appl. {\bf 263} (2019), 199--208.

\bibitem{KOO} A.~Klyachko, A.~Olshanskii, D.~Osin, {\em On topologizable and non-topologizable groups}, Topology Appl. {\bf 160} (2013), 2104--2120.


\bibitem{Kotov} M.~Kotov, {\em Topologizability of countable equationally Noetherian algebras}, Algebra Logika, {\bf 52}:2 (2013), 155--171.




\bibitem{L}
A. Luk\'acs, {\em Compact-Like Topological Groups}, Heldermann Verlag: Lemgo, Germany, 2009.


\bibitem{Markov} A.A.~Markov, {\em Three papers on topological groups: I. On the existence of periodic connected topological groups. II. On free topological groups. III. On unconditionally closed sets}, Amer. Math. Soc. Translation {\bf 30} (1950), 120 pp.

\bibitem{Ol} A.Yu.~Olshanski, {\em A remark on a countable non-topologizable group}, Vestnik Mosk. Gos. Univ. Mat. Mekh., No.3 (1980), 103.

\bibitem{Prot} I.V.~Protasov, {\em Combinatorics of Numbers}, VNTL Publ., Lviv, 1997, 70 pp.

\bibitem{PW} M.~Putcha, J.~Weissglass, {\em A semilattice decomposition into semigroups having at most one idempotent}, Pacific J. Math. {\bf 39} (1971), 225--228.

\bibitem{Shelah} S.~Shelah, {\em On a problem of Kurosh, Jonsson groups, and applications. Word problems, II}, (Conf. on Decision Problems in Algebra, Oxford, 1976), 373–394; Stud. Logic Foundations
Math., {\bf 95}, North-Holland, Amsterdam-New York, 1980.

\bibitem{Stepp69}
J.W.~Stepp, {\it A note on maximal locally compact semigroups}.
Proc. Amer. Math. Soc. {\bf 20}  (1969), 251--253.

\bibitem{Stepp75}
J.W.~Stepp, {\it Algebraic maximal semilattices}. Pacific J. Math.
{\bf 58}:1  (1975), 243--248.

\bibitem{Taimanov} A.D.~Taimanov, {\em An example of a semigroup which admits only the discrete topology}, Algebra i Logika,  {\bf 12}:1 (1973) 114--116 (in Russian).

\bibitem{Tamura82} T.~Tamura, {\em Semilattice indecomposable semigroups with a unique idempotent}, Semigroup Forum {\bf 24}:1 (1982), 77--82.


\bibitem{Z2} Y. Zelenyuk,  {\em Semigroup extensions of Abelian topological groups},
Semigroup Forum {\bf 98} (2019), 9--21.
\end{thebibliography}
\end{document}